\documentclass[11pt,a4paper]{article}

\RequirePackage[OT1]{fontenc}
\RequirePackage{amsthm,amsmath}
\RequirePackage[numbers]{natbib}
\RequirePackage[colorlinks,citecolor=blue,urlcolor=blue]{hyperref}

\usepackage{fullpage}
\usepackage{amsmath,amsthm,amssymb}
\usepackage{dsfont}
\usepackage{mathrsfs}
\usepackage{times}

\usepackage{cmap}
\usepackage[dvipsnames]{xcolor}
\usepackage[colorlinks]{hyperref}
\hypersetup{
    linkcolor=red,
    citecolor=OliveGreen,
    filecolor=black,
    urlcolor=black,
}
\usepackage{microtype}
\usepackage{graphicx}
\usepackage{comment}

\usepackage{xfrac}
\usepackage{enumitem}
\def\Argmin{\mathop{\hbox{\rm Argmin}}}
\def\Argmax{\mathop{\hbox{\rm Argmax}}}
\def\Var{{\He}}
\def\Varn{{\He}_n}
\def\Cov{{\boldsymbol \Sigma}}

\def\He{\mathbf{H}}
\def\Gr{\mathbf{G}}
\def\Ma{\mathbf{J}}
\def\bA{\mathbf{A}}
\def\Qu{\mathbf{Q}}
\def\Id{\mathbf{I}}
\def\csX{{\widetilde X}}

\def\s{\mathtt{s}}

\def\k{\varkappa}
\def\Rp{\R^{(+)}}
\def\vtheta{\vartheta}

\newcommand{\SC}{\textsf{SC}}

\newcommand{\B}{\mathfrak{B}}
\newcommand{\wt}{\widetilde}
\newcommand{\lsim}{\lesssim}
\newcommand{\gsim}{\gtrsim}
\newcommand{\lang}{\langle}
\newcommand{\rang}{\rangle}
\newcommand{\ud}{\textup{d}}
\newcommand{\R}{\mathds{R}}
\newcommand{\E}{\mathds{E}}
\newcommand{\T}{\top}
\newcommand{\half}{\tfrac{1}{2}}

\newcommand{\wh}{\widehat}
\newcommand{\cN}{\mathcal{N}}

\newcommand{\cS}{\mathcal{S}}
\newcommand{\cSc}{{\mathcal{S}_c}}
\newcommand{\cY}{\mathcal{Y}}
\newcommand{\cX}{\mathcal{X}}
\newcommand{\bP}{\mathbf{P}}
\newcommand{\Z}{\mathcal{Z}}
\newcommand{\X}{\mathcal{X}}
\newcommand{\Y}{\mathcal{Y}}
\newcommand{\tr}{\textup{Tr}}
\newcommand{\df}{\texttt{\textup{df}}}

\newcommand{\Prob}{\mathds{P}}
\newcommand{\prccq}{\preccurlyeq}
\newcommand{\succq}{\succcurlyeq}

\newcommand{\En}{\E_{n}}
\newcommand{\dto}{\rightsquigarrow}
\newcommand{\deff}{d_{\textup{eff}}}
\newcommand{\veps}{\varepsilon}
\newcommand{\Exp}{\textup{Exp}}
\newcommand{\vphi}{\varphi}

\newcommand{\cE}{\mathcal{E}}
\newcommand{\cEc}{{{\mathcal{E}_0^c}}}
\newcommand{\cZ}{\mathcal{Z}}
\newcommand{\Distrib}{\mathscr{P}}
\newcommand{\cP}{\mathcal{P}}
\newcommand{\ind}{\mathds{1}}

\newcommand{\dimacorr}[2]{{\color{magenta}#2}}

\newtheorem{theorem}{Theorem}[section]
\newtheorem{lemma}{Lemma}[section]

\newtheorem{corollary}{Corollary}[section]

\newtheorem{proposition}{Proposition}[section]
\newtheorem{remark}{Remark}[section]

\newtheorem{innercustomass}{Assumption}

\newenvironment{customass}[1]
{\renewcommand\theinnercustomass{#1}\innercustomass}
{\endinnercustomass}

\title{Finite-sample analysis of $M$-estimators using self-concordance\footnote{To appear in {\em Electronic Journal of Statistics} (as of November 2020).}}

\author{Dmitrii M. Ostrovskii\thanks{University of Southern California, Viterbi School of Engineering. 
3650 McKlintock Avenue, CA 90089, Los Angeles, USA. Email: \texttt{dmitrii.ostrovskii@inria.fr}.} 
	\and Francis Bach\thanks{INRIA Paris and \'Ecole Normale Sup\'erieure. 2 rue Simone Iff, 75012 Paris, France. Email: \texttt{francis.bach@inria.fr}.}}

\date{}

\begin{document}
\maketitle

\begin{abstract}

The classical asymptotic theory for parametric $M$-estimators guarantees that, in the limit of infinite sample size, the excess risk has a chi-square type distribution, even in the misspecified case. We demonstrate how \textit{self-concordance} of the loss allows to characterize the {\em critical sample size} sufficient to guarantee a chi-square type in-probability bound for the excess risk. Specifically, we consider two classes of losses: (i) self-concordant losses in the classical sense of Nesterov and Nemirovski, i.e., whose third derivative is uniformly bounded with the $3/2$ power of the second derivative; (ii) \textit{pseudo} self-concordant losses, for which the power is removed. These classes contain losses corresponding to several generalized linear models, including the logistic loss and pseudo-Huber losses. 

Our basic result under minimal assumptions bounds the critical sample size by $O(d \cdot d_{\text{eff}}),$ where $d$ the parameter dimension and $d_{\text{eff}}$ the effective dimension that accounts for model misspecification. In contrast to the existing results, we only impose \textit{local} assumptions that concern the population risk minimizer $\theta_*$. Namely, we assume that the calibrated predictors, i.e., predictors scaled by the square root of the second derivative of the loss, is subgaussian at $\theta_*$. Besides, for type-ii losses we require boundedness of certain measure of curvature of the population risk at $\theta_*$.

Our improved result bounds the critical sample size from above as $O(\max\{d_{\text{eff}}, d \log d\})$ under slightly stronger assumptions. Namely, the local assumptions must hold in the neighborhood of $\theta_*$ given by the Dikin ellipsoid of the population risk. Interestingly, we find that, for logistic regression with Gaussian design, there is no actual restriction of conditions: the subgaussian parameter and curvature measure remain near-constant over the Dikin ellipsoid. Finally, we extend some of these results to $\ell_1$-penalized estimators in high dimensions.

\end{abstract}

%



\section{Introduction and problem formulation}
\label{sec:intro}

Recall the standard statistical learning setup: given a set~$\Theta \subseteq \R^d$ that parametrizes the space of possible hypotheses, and observing a random~$Z \in \Z$ with unknown distribution $\cP$, one would like to minimize the \textit{population risk}~$L(\theta) := \E[\ell_Z(\theta)]$.
For each possible observation $z$ of $Z$, the \textit{loss} $\ell_{z}: \Theta \to \R$ specifies the cost of choosing~$\theta$ under the outcome~$\{Z = z\}$, and~$\E[\cdot]$ is the expectation with respect to the distribution~$\cP$.
This distribution is assumed unknown, so the population risk cannot be computed and minimized directly. 
Instead, one is granted access to the sample~$(Z_1, ..., Z_n)$ of independent copies of~$Z$, and uses it to construct an estimate~$\wh\theta$ of the population risk minimizer,
\[
\theta_* \in \Argmin_{\theta \in \Theta} L(\theta),
\]
assuming that such a minimizer exists.
As such, we can consider the empirical distribution~$\cP_n$ -- uniform probability measure supported on the sample -- and the empirical risk~$L_n(\theta)$, defined as the observable counterpart of~$L(\theta)$, namely,
\[
L_n(\theta) := \frac{1}{n} \sum_{i=1}^n \ell_{Z_i}(\theta).
\]
Ideally, we would like to have an estimator with small \textit{excess risk}~$L(\wh\theta) - L(\theta_*)$, in probability or in expectation over the sample.
Since for each fixed value~$\theta$ of the parameter, $L_n(\theta)$ is an unbiased estimate of $L(\theta)$ which converges to $L(\theta)$ almost surely by the law of large numbers, a natural candidate estimator of $\theta_*$ is the \textit{empirical risk minimizer} (ERM), defined as
\[
\wh\theta_n \in \Argmin_{\theta \in \Theta} L_n(\theta).
\]
In this paper, we are concerned with establishing high-probability {finite-sample} bounds on the excess risk~$L(\wh\theta_n) - L(\theta_*)$ of this estimator. 
The classical {Fisher theorem}~(\cite{lehmann2006theory}) implies the rescaled excess risk has a chi-square type limiting behavior, under weak conditions, when~$n \to \infty$. 
When stated informally, our goal in this paper is to characterize the {\em critical sample size}  sufficient to enter the this ``asymptotic regime'', i.e., to guarantee a chi-square type high-probability bound for the excess risk in {finite sample}.  
Elaborating on this goal in more detail and stating our results would be impossible without first giving a brief overview of the classical asymptotic theory. 
We give such overview in the next section.

\subsection{Classical asymptotic theory}
Our main focus in this paper is the setting where~$L_n(\theta)$ is a negative log-likelihood, that is~$\ell_z(\theta) = -\log p_\theta(z)$ where~$p_\theta(\cdot)$ is some probability density supported on~$\Z$.
In this case,~$\wh\theta_n$ maximizes the likelihood of observing the i.i.d.~sample~$(Z_1,..., Z_n)$ from~$\cP_{\theta}$ ranging over a \textit{parametric family}~$\Distrib = \{\cP_\theta, \; \theta \in \Theta\}$. 
In reality,~$\Distrib$ may or may not contain the actual data-generating distribution~$\cP$. When~$\cP \in \Distrib$, we say that the parametric model corresponding to $\Distrib$ is \textit{well-specified}; in this case, ERM becomes the maximum-likelihood estimator (MLE). 
Otherwise, the model is called \textit{misspecified}, and ERM can be regarded as MLE under model misspecification, or \textit{quasi} maximum likelihood estimator~\cite{white1982maximum}.
In this case,~$\cP_{\theta_*}$ corresponds to the ``projection'' of~$\cP$ onto the family~$\Distrib$ in the sense of the Kullback-Leibler divergence, and the quasi MLE approximates~$\cP_{\theta_*}$ by replacing~$\cP$ with the empirical distribution~$\cP_n$.

Our goal in this section is to give a brief overview of the main results of the asymptotic theory of~$M$-estimation.
Most of them, see monographs~\cite{lehmann2006theory,ibragimov2013statistical,vdv,borovkov}, rely on the \textit{local regularity} assumptions about the loss, allowing for second-order Taylor expansion of~$L(\theta)$ around~$\theta_*$.
In particular, it is assumed that~$L(\theta)$ is sufficiently smooth at~$\theta_*$, which is an interior point of~$\Theta$, so that the first-order optimality condition for~$\theta_*$ reduces to~$\nabla L(\theta_*) = 0$.
Moreover, the Hessian 
\[
\He := \nabla^2 L(\theta_*)
\]
is assumed to be non-degenerate, i.e.,~$\He \succ 0$. Finally, the empirical risk is assumed to be three times continuously differentiable at~$\theta_*$, see,~e.g.,~\cite{lehmann2006theory}. 
When combined together, these assumptions allow to derive, as a starting point, the \textit{local asymptotic normality} of quasi MLE: when~$n \to \infty$ with fixed~$d$, 
\begin{equation}
\label{eq:lan-fisher}
\begin{aligned}
\sqrt{n}\He^{1/2}(\wh\theta_n - \theta_*)  &\dto \cN(0, \He^{-1/2} \Gr \He^{-1/2}),
\end{aligned}
\end{equation}
where~$\dto$ denotes convergence in law, and~$\Gr$ is the covariance matrix of the loss gradient at~$\theta_*$ (also called Fisher's information matrix):
\[
\Gr := \E[\nabla \ell_Z(\theta_*)\nabla \ell_{Z}(\theta_*)^\T].
\]
Matrices~$\Gr$ and~$\He$ remain fixed as $n$ grows. Hence, under mild regularity assumptions,\footnotemark~one also has that the variance of~$\wh\theta_n$ decreases as~$O(1/n)$.
\footnotetext{It suffices for~$\rho_n := \sqrt{n}\|\wh\theta_n\|_2$ to be uniformly integrable, i.e.,~$\lim_{\veps \to 0}\sup_{n} \E[\rho_n \ind_{\rho_n \ge \veps}] = 0$. This is a very weak condition; see~\cite[Sec. 6.2]{klenke2013probability} for stronger (but easier to verify) conditions.}
Moreover, in the well-specified case~$\Gr = \He$, see, e.g.,~\cite{bartlett1953approximate}, which leads to~\textit{Fisher's theorem}:
\[
\sqrt{n} \He^{1/2}(\wh\theta_n - \theta_*) {\dto}\cN(0, \Id_d),
\] 
where~$\Id_d$ is the identity matrix of size~$d$.
Thus, denoting~$\|\cdot\|_{\Ma}$ the norm linked to positive semidefinite matrix~$\Ma$ by~$\|x\|_{\Ma} = \|\Ma^{1/2} x\|_{2}$, we have
$n\| \wh\theta_n-\theta_*\|_{\He}^2 {\dto}{\chi_d^2},$
where~$\chi_d^2$ is the chi-square law with $d$ degrees of freedom.
The second-order Taylor expansion of the average risk around $\theta_*$ then allows to derive the same asymptotic law for the scaled excess risk~$2n[L(\wh\theta_n) - L(\theta_*)]$ -- this result is known as \textit{Wilks' theorem}. 
In turn, this implies (under mild regularity conditions) that
\begin{equation}
\label{eq:crb-well-spec}
\En[L(\wh\theta_n)]  - L(\theta_*) = \frac{d}{2n} + o(n^{-1}), \; \text{as} \; n \to \infty,
\end{equation}
where~$\En$ is the expectation over the product distribution~$\cP^{\otimes n}$ of~$(Z_1, ..., Z_n)$.
More precisely, by the standard chi-square deviation bounds (see~e.g.,~\cite[Lemma 1]{lama2000}) one has that, with probability~$\ge 1-\delta$,
\begin{equation}
\label{eq:crb-prob-well-spec}
L(\wh\theta_n)  - L(\theta_*) = \frac{(\sqrt{d} + \sqrt{2\log(1/\delta)})^2}{2n} + o(n^{-1}).
\end{equation}
Finally, these $O(d/n)$ asymptotic bounds can be extended to the general situation of misspecified models by introducing the \textit{effective dimension}:
\[
\deff := \E [\| \nabla\ell_Z(\theta_*) \|_{\He^{-1}}^2] = \tr (\He^{-1/2} \Gr  \He^{-1/2}).
\]
Note that in a well-specified model,~$\deff = d$ since $\Gr = \He$; moreover, in the ill-specified case one can still have~$\deff  = O(d)$ ``in favorable circumstances'' -- we will consider one such situation, that of misspecified linear regression, later on.\footnotemark
\footnotetext{We can also have~$\deff < d$ if we get ``extremely lucky''. For example, consider the Gaussian shift model $y \sim \cN(\theta,1)$, and let in reality~$y \sim \cN(0,\sigma)$. Then~$\deff = \sigma^2$ while~$d = 1$.}
The expected excess risk bound~\eqref{eq:crb-well-spec} then changes to
\begin{equation}
\label{eq:crb}
\En[L(\wh\theta_n)]  - L(\theta_*) = \frac{\deff}{2n} + o(n^{-1}),
\end{equation}
and the corresponding in-probability bound (see again~\cite[Lemma 1]{lama2000})
~is
\begin{equation}
\label{eq:crb-prob}
L(\wh\theta_n)  - L(\theta_*) = \frac{\deff (1 + \sqrt{2\log(1/\delta)})^2}{2n} + o(n^{-1}).
\tag{$\star$}
\end{equation}
In fact, the main term in the right-hand side of~\eqref{eq:crb} is the minimum possible asymptotic variance of any unbiased estimator; this result is known as the Cram\'er-Rao bound.

For what goes next, it is important to note that the asymptotic approach can be summarized as follows:
\begin{itemize}
\item
First, the estimate is \textit{localized}: $\|\wh\theta_n - \theta_*\|_{\He}^2$ is upper-bounded with the squared ``natural'' norm of the score, $\|\nabla L_n(\theta_*)\|_{\He^{-1}}^2$, which can be controlled by the central limit theorem.
\item
Then, using the second-order Taylor expansion of~$L(\theta)$ around~$\theta_*$, similar behavior is obtained for the excess risk~$L(\wh\theta_n) - L(\theta_*)$.
\end{itemize}

Paying tribute to the clarity and historical significance of the classical asymptotic theory, one should keep in mind that its operating regime~$n \to \infty$ with fixed parameter dimension usually cannot be applied in the modern context. 
The recent works~\cite{donoho2016high,barbier2017phase} extend the classical results to the asymptotic high-dimensional regime~$d \to \infty$ with~$d = O(n)$, analyzing $M$-estimator as the fixed point of the approximate message passing algorithm. However, existing analysis of approximate message passing in finite samples is scarce: the only work we are aware of is~\cite{rush2018finite}, which only considers fixed-design linear regression.
Postponing a more detailed review of related work to Section~\ref{sec:related-work}, let us briefly overview the main approaches in finite-sample analysis.

\subsection{Finite-sample regime and empirical processes}
This work has been motivated by the following question:
\begin{quote}
\centering
{\em For what finite~$n$ the excess risk admits a chi-square type bound akin to~\eqref{eq:crb-prob}?}
\end{quote}
One rather general approach towards answering this question, i.e., addressing the fully finite-sample regime, has been outlined in~\cite{spokoiny2012parametric}, and can be described as follows. 
First, the parameter space~$\Theta$ is divided into the local subset, given as the intersection of~$\Theta$ and the (unit-radius) \textit{Dikin ellipsoid} of~$\theta_*$, 
\begin{equation}
\label{def:dikin-unit}
\Theta_1(\theta_*) := \{\theta \in \R^d: \|\theta - \theta_*\|_{\He} \le 1\},
\end{equation}
and the complement subset~$\Theta \setminus \Theta_1(\theta_*)$.
Then, the second step of the asymptotic approach is replaced with so-called \textit{quadratic bracketing:} the excess risk is ``sandwiched'' on~$\Theta_1(\theta_*)$ between two quadratic forms which correspond to the inflation and deflation of~$\|\theta - \theta_*\|_{\He}^2$.
On the other hand, the first step (localization of the estimate) is done via the control of the event $\{\wh\theta_n \notin \Theta_1(\theta_*)\}$, by bounding the uniform deviations of the empirical risk~$L_n(\theta) - L_n(\theta_*)$ via advanced tools from empirical process theory such as generic chaining~\cite{talagrand2006generic}.
This approach is quite powerful, allowing to derive the counterparts of asymptotic results in the non-asymptotic regime~$n \ge c_\delta \cdot \deff$, where the constant~$c_{\delta}$ only depends on the desired confidence level~$1-\delta$.
However, it requires rather strong \textit{global} assumptions on the pointwise deviations of the empirical risk process, which are necessary to control its uniform deviations, see~\cite[Sections~2.2 and~4]{spokoiny2012parametric}.
Close in spirit to~\cite{spokoiny2012parametric} are the techniques developed in~\cite{chernozhukov2017central} to study Gaussian approximation of the maxima of the sums of i.i.d.~random variables.
The main highlight of~\cite{chernozhukov2017central} is the ability to handle the regime of exponentially large dimensionality, with respect to the sample size, due to the special structure of the statistics under study.
However, much like in~\cite{spokoiny2012parametric}, the techniques of~\cite{chernozhukov2017central} rely on the advanced machinery of empirical processes.

Meanwhile, in the special case of random-design least-squares, finite-sample analysis is way simpler, and heavy-weight machinery of empirical processes is not needed. 
In this case, the problem  is reduced to the control of a {\em single} random matrix, the sample covariance matrix of the design vector, which encapsulates the second-order information about the risk. 
Our primal goal in this work is to extend these ideas to a wider class of models with non-quadratic losses of certain types, including the losses arising in conditional generalized linear models and robust regression.
For these classes of losses, one may carefully exploit their regularity properties, which allows to avoid using the empirical processes machinery -- and the associated \textit{global} conditions -- when localizing the empirical risk minimizer. 
Deferring further discussion of our contributions to Sec.~\ref{sec:outline} and related work to Sec.~\ref{sec:related-work}, let us overview the case of least-squares.

\subsection{Simple case: least-squares}
An original approach introduced in~\cite{hsu2012random} allows to obtain finite-sample excess risk bounds in the setting of unconstrained least-squares linear regression with random design.
Here,~$\Theta = \R^d$, and the observations take the form $Z=(X,Y)$ where $X \in \R^d$ and $Y \in \R$. 
The goal is to predict \textit{response}~$Y$ as a linear combination of \textit{design}~$X$ with parameter~$\theta \in \R^d$, and one takes~$\ell_Z(\theta)$ to be 
$
\ell_Z(\theta) = \tfrac{1}{2\sigma^2} (Y - X^\T\theta)^2.
$
ERM then reduces to the ordinary least-squares estimator. 
Least-squares correspond to the implicit assumption that the residual~$\veps = Y-X^\T\theta_*$  has Gaussian distribution $\veps \sim \cN(0,\sigma^2)$ with $\sigma > 0$, and is independent of~$X$, which allows to factor out the distribution of~$X$ from the model. 
Note that the rate $O(d/n)$ translates to the well-known minimax rate~$O(d\sigma^2/n)$ for the mean square error $\E[(Y - X^\T\theta)^2] - \sigma^2$. 
Moreover, sometimes the Gaussian assumption on~$\veps$ can be relaxed, and the misspecified situation becomes essentially as favorable as the well-specified one, at least from the asymptotic point of view.
Indeed, normalizing the noise to have unit variance, and~using that
\[
\nabla \ell_Z(\theta_*) = \veps X \quad \text{and} \quad \Var = \E [XX^\T],
\]
we get~$\deff = \E[\veps^2 \|\Var^{-1/2} X\|^2].$
Hence,~$\deff = d$ for any distribution of~$\veps$ with $\E[\veps^2] = 1$, provided that~$\veps$ and~$X$ are independent.
Moreover, assuming that~$Y$ and all one-dimensional marginals~of~$X$ have finite fourth moment, i.e.,
\[
\begin{aligned}
\sqrt{\E[Y^4|X=x]} &\le \kappa_\veps  \E[Y^2|X=x], \;\; \forall x \in \R^d,\\
\sqrt{\E[\langle u, X \rangle^4]} &\le \kappa_{X}  \E[\langle u, X \rangle^2], \;\; \forall u \in \R^d,
\end{aligned}
\]
we can bound~$\deff$ as~$\deff \le \kappa_X \cdot \kappa_\veps \cdot d$. In other words,~$\deff$ and~$d$ are comparable.

Now, the approach of~\cite{hsu2012random} exploits the fact that~$L(\theta)$ is a quadratic form,
\begin{equation}
\label{eq:ave-risk-linear}
L(\theta) - L(\theta_*) = \half \|\theta - \theta_* \|_\Var^2,
\end{equation}
and the empirical risk is a quadratic form corresponding to~$\Varn = \frac{1}{n}\sum_{i=1}^n X_i X_i^\T$:
\begin{equation*}
L_n(\theta)  - L_n(\theta_*) = \half \|\theta - \theta_* \|_{\Varn}^2 + \langle \nabla L_n(\theta_*), \theta-\theta_* \rangle.
\end{equation*}
As such, the \textit{global} curvature information about $L(\theta)$ is encapsulated in a single matrix $\Var$, and we have at our disposal an unbiased estimate~$\Varn$ of this matrix.
This observation allows to dramatically simplify the analysis: it suffices to control the deviations of~$\Varn$ from its expectation, which can be done using the standard tools of random matrix theory. 
In particular, in~\cite{vershynin2010introduction}, see also Theorem~\ref{th:covariance-subgaussian-simple} in Appendix, it is shown that whenever~$X$ is $K$-subgaussian in all directions, and
\begin{equation}
\label{eq:cov-subgaussian-complexity-intro}
n \gtrsim K^4(d + \log(1/\delta)),
\end{equation}
where symbol~$\gtrsim$ hides a constant factor, with probability at least~$1-\delta$ it holds
\begin{equation}
\label{eq:cov-dist-equiv-intro}
\tfrac{1}{2} \|\Delta\|^2_{\Var} \le \|\Delta\|^2_{\Varn} \le 2 \|\Delta\|^2_{\Var}, \quad \forall \Delta \in \R^d.
\end{equation}
In other words, the sample second-moment matrix~$\Varn$ approximates~$\Var$, up to a constant factor, in the sense of the corresponding Mahalanobis distances (in particular,~$\Varn$ is non-degenerate whenever~$\Var$ is).
This result can then be exploited as follows: since~$\nabla L_n(\wh\theta_n) = 0$, and~$\Varn \succ 0$, 
\begin{equation}
\label{eq:quadratic-dual}
\|\wh\theta_n - \theta_* \|_{\Varn}^2 = \|\nabla L_n(\theta_*)\|_{\Var_n^{-1}}^2.
\end{equation}
Using~\eqref{eq:cov-dist-equiv-intro}, this gives
~$
\tfrac{1}{2} \|\wh\theta_n - \theta_* \|_{\Var}^2 \le 2\|\nabla L_n(\theta_*)\|_{\Var^{-1}}^2,
$
which, via~\eqref{eq:ave-risk-linear}, results in
\[
L(\wh \theta_n) - L(\theta_*) \le 2\|\nabla L_n(\theta_*)\|_{\Var^{-1}}^2.
\]
Finally, a non-asymptotic version of~\eqref{eq:crb-prob} is obtained by controlling the squared norm~$\|\nabla L_n(\theta_*)\|_{\Var^{-1}}^2$ under light-tailed (say, subgaussian or subexponential) assumptions on~$\nabla \ell_Z(\theta_*) = \veps X$, through standard concentration inequalities.
These light-tailed assumptions can further be relaxed to fourth-moment assumption, using the generalized median-of-means estimator (see~\cite{hsu2016loss}),
On the other hand, it is much more challenging to get rid of the light-tailed assumption on~$X$, as obtaining covariance estimators with guarantees of the type~\eqref{eq:cov-dist-equiv-intro} under weak moment assumptions is by itself a non-trivial problem. Recently, this problem has been addressed in~\cite{heavy-covariance}, whose authors then proposed an estimator for ridge and ridgeless regression with near-optimal high-probability guarantees under heavy-tailed assumptions on~$X$ (see~\cite[Theorem 6.1]{heavy-covariance}).\footnotemark
\footnotetext{Another possibility is to use a rejection sampling argument similar to the one employed in the proof of our Theorem~\ref{th:crb-mine-complex}. 
This, however, prohibits us from taking small values of the confidence parameter~$\delta$, namely, those decreasing polynomially fast with~$\min(d,n)$, cf.~\eqref{eq:restricted-risk-condition}.}

The remarkable feature of the outlined analysis is that, as soon as the curvature of~$L(\theta)$, as given by~$\Var$, is reliably estimated, the localization step is \textit{``automatic''} due to~\eqref{eq:quadratic-dual}.
The only requirement is for~$n$ to reach the lower bound~\eqref{eq:cov-subgaussian-complexity-intro}, so that one could relate the norms~$\| \cdot \|_{\Varn}$ and~$\|\cdot\|_\Var$.
The crucial fact here is that for the quadratic loss, the curvature information is \textit{global}, i.e., is encoded in a single matrix.
However, for more general losses this is not the case, and there seems to be no direct way of extending the above argument. 
As discussed before, the known solution to the problem~\cite{spokoiny2012parametric} involved localization of the estimate, through the control of the \textit{global} uniform deviations of~$L_n(\theta)$, to a neighborhood of~$\theta_*$ where a local quadratic approximation can be used; this approach requires global assumptions on the pointwise deviations of~$L_n(\theta)$.
Yet, we will show that in some other models beyond linear regression with quadratic loss, the \textit{local} analysis suffices to provide localization of the estimate,
and the complicated and opaque localization step using generic chaining, as in~\cite{spokoiny2012parametric}, can be circumvented.

\subsection{Contributions and outline}
\label{sec:outline}

Our analysis applies to \textit{linear prediction models}: observing a pair~$Z=(X,Y)$ with~$X\in \cX \subseteq \R^d$ and~$Y\in\cY \subseteq \R$, one predicts~$Y$ through linear combination~$\eta=X^\T\theta$ with~$\theta\in\Theta\subseteq\R^d$.
Accordingly, we consider losses given by
\[
\ell_Z(\theta) = \ell(Y,X^\T\theta)
\]
for some function~$\ell: \cY \times \R \to \R$ assumed to be sufficiently smooth in its second argument. 
This subsumes regression~($\Y = \R$) and classification~($\Y = \{0,1\}$). 
Moreover, we assume the ability to bound the third derivative of~$\ell(y,\eta)$ with respect to~$\eta$ via the second derivative in two alternative ways, as will be detailed in Section~\ref{sec:assumptions}. 
Such \textit{self-concordance} assumptions originate from~\cite{nemirovski1994interior}, where they were used in the context of interior-point methods; later on, they were modified and used in the statistical analysis of logistic regression~\cite{bach2010self-concordant,bach-moulines}. We consider both variants of self-concordance in our analysis, and show that the canonical self-concordance assumption, introduced in~\cite{nemirovski1994interior}, leads to somewhat better bounds on the critical sample size than its modification suggested in~\cite{bach2010self-concordant} (see Sections~\ref{sec:basic}--\ref{sec:improved}). 
In addition to self-concordance of the loss, we make some assumptions on the \textit{local} behavior of the gradient and Hessian of the empirical risk at the population risk minimizer~$\theta_*$, or its neighbordhood given by the unit Dikin ellipsoid~\eqref{def:dikin-unit} of the population risk at~$\theta_*$.
To prove our main results (cf.~Theorems~\ref{th:crb-mine-improved}--\ref{th:crb-fake-improved}), we carefully combine these assumptions through a non-standard covering argument, which allows us to control the uniform deviations of~$\nabla^2 L_n(\theta)$ from~$\nabla^2 L(\theta)$ over the Dikin ellipsoid, and implies localization of the estimator.
We mention once again that \textit{global} assumptions in the vein of~\cite{spokoiny2012parametric} about the deviations of the empirical risk, its gradient and Hessian can be avoided by using self-concordance. 

Our framework includes random-design least-squares linear regression as a baseline. 
However, as we show in Section~\ref{sec:assumptions}, it is in fact much more general.
First, it encompasses some conditional \textit{generalized linear models} with random design.
Here we find that both versions of self-concordance are related to natural assumptions about the moments of~$Y$, and discover several generalized linear models amenable to our analysis, including logistic regression.
Second, we can address some common losses in \textit{robust estimation}, which turn out to be pseudo self-concordant in the sense of~\cite{bach2010self-concordant}. 
Moreover, we show how to slightly modify these losses to make them \textit{canonically} self-concordant, while preserving their first- and second-order structure.
According to our theory, this leads to the improved statistical performance of the $M$-estimator, as characterized by the sufficient sample size to reach the asymptotically optimal rate for the excess risk.

Our analysis carries out the following plan.
First, the local assumptions allow to make sure that starting from the certain sample size, the sample Hessian
\[
\He_n = \nabla^2 L_n(\theta_*)
\]
approximates the true Hessian\dimacorr{$\He = \E[\ell''(Y,X^\T\theta_*)XX^\T]$}~{$\He = \nabla^2 L(\theta_*)$} up to a constant factor, completely analogous to the case of least squares. 
After that, self-concordance comes at play. 
First, using simple analytic arguments, we prove that with high probability,~$\nabla^2 L_n(\theta)$ remains nearly constant in a Dikin ellipsoid of a smaller radius of order~$O(1/\sqrt{d})$, leading to a larger critical sample size than in the case of least-squares. 
We then use these initial results to prove that under slightly stronger -- but still local -- assumptions,~$\nabla^2 L_n(\theta)$ in fact remains constant in a \textit{constant-radius} Dikin ellipsoid, leading to the critical sample size comparable to that in least-squares (cf.~Theorems~\ref{th:crb-mine-improved}--\ref{th:crb-fake-improved}). 
This is done via a simple but somewhat non-trivial covering argument, which might be of independent interest.


Let us now give a more detailed overview of the obtained results.


In Section~\ref{sec:basic}, we show that for \textit{pseudo} self-concordant losses~\cite{bach2010self-concordant}, 
the asymptotically optimal (up to a costant factor) bound on the excess risk is guaranteed when the sample size reaches~$O(\rho \cdot d \cdot \deff)$ up to a logarithmic factor in~$1/\delta$,
where~$\rho$ is the \textit{local curvature} parameter linking~$\He$ and~$\Cov := \E[XX^\T]$ by
\[
\Cov \prccq \rho \He.
\] 
Moreover, for \textit{canonically} self-concordant losses in the sense of~\cite{nemirovski1994interior}, the dependency on~$\rho$ can be eliminated, and the critical sample size becomes~$ O(d \cdot \deff)$. 
We now give a simplified (and slightly vulgarized) formulation of these two results.

\begin{theorem}[Simplified formulation of Theorems~\ref{th:crb-fake-complex}--\ref{th:crb-mine-complex}]
\label{th:crb-preview}
Assume that~$\ell(y,\cdot)$ is self-concordant, for any~$y$, in the sense of Nesterov and Nemirovski~\cite{nemirovski1994interior}, i.e., 
\begin{equation}
\label{eq:sc-preview}
|\ell'''_{\eta}(y,\eta)|\le 2\ell''_{\eta}(y,\eta)^{3/2}, \quad \forall \eta \in \R,
\end{equation}
and that~$\ell'_{\eta}(Y,X^\top \theta_*) X =: \nabla \ell_Z(\theta_*)$ and~$\ell''_{\eta}(Y,X^\top \theta_*)^{1/2} X$ are subgaussian. 
Then
\begin{equation}
\label{eq:crb-all-preview}
L(\wh\theta_n) - L(\theta_*) 
\lsim \|\wh\theta_n - \theta_*\|_{\He}^2 
\lsim \|\nabla L_n(\theta_*)\|_{\He^{-1}}^2
\lsim \frac{\deff\log\left({e}/{\delta}\right)}{n}
\end{equation}
with probability~$\ge 1-\delta,$~$\delta \in (0,1)$, as long as
\begin{equation}
\label{eq:n-sc-complex-preview}
n \gsim \deff  \cdot d \cdot \log\left({ed}/{\delta}\right),
\end{equation}
where~$\lsim, \gsim$ hide constants.  
Moreover, if the loss satisfies the modified assumption
\begin{equation}
\label{eq:psc-preview}
|\ell'''_{\eta}(y,\eta)|\le \ell''_{\eta}(y,\eta), \quad \forall \eta \in \R
\end{equation}
instead of~\eqref{eq:sc-preview},~$X$ is as well subgaussian, and~$\Cov \prccq \rho  \He$, then~\eqref{eq:crb-all-preview} is valid once
\begin{equation}
\label{eq:n-psc-complex-preview}
n \gsim \rho \cdot \deff  \cdot d \cdot \log\left({ed}/{\delta}\right).
\end{equation}
\end{theorem}
While the only available generic upper bound on~$\rho$ is given by the inverse of the \textit{global} strong convexity modulus of the loss, and can be very large or even infinite in the case of unbounded predictors,
the \textit{actual} value of~$\rho$ depends on the data distribution, and is moderate when this distribution is not chosen adversarially, as discussed in~\cite[Sections~3.1,~4.2]{bach-moulines} and in our Section~\ref{sec:ass-distrib}. In this vein, we show in Appendix~\ref{sec:subgaussian-check} that
$
\rho \lsim 1+\|\theta_*\|_{\Cov}^3
$
in logistic regression with Gaussian design~$X \sim \cN(0, \Cov)$.
Motivated by this result, we propose canonically self-concordant losses for classification and robust regression in Section~\ref{sec:ass-qsc}.

In Section~\ref{sec:improved}, we obtain improved bounds for the critical sample size, scaling {\em near-linearly in the parameter dimension,} under slightly stronger assumption on the data distribution. 
Essentially, we now require that the \textit{calibrated design}
$
\wt X(\theta) := [\ell''(Y,X^\T\theta)]^{1/2}X,
$ 
is subgaussian uniformly over~$\theta$ in the set
\begin{equation}
\label{def:dikin}
\Theta_r(\theta_*) := \{\theta: \| \theta - \theta_*\|_{\He} \le r\}
\end{equation}
-- the $r$-radius \textit{Dikin ellipsoid} of the population risk at~$\theta_*$.
specifically, we require~$r = 1$ for canonically self-concordant losses, and~$r = 1/\sqrt{\rho}$ for pseudo self-concordant losses. 
This assumption is still local, and is not much more restrictive in some practical situations: in Appendix~\ref{sec:subgaussian-check} we show, informally, that in the case of logistic regression with Gaussian design, the tails of~$\wt X(\theta)$ over~$\theta \in \Theta_{1/\sqrt{\rho}}(\theta_*)$ are not heavier than those of~$\wt X(\theta_*)$ (see Proposition~\ref{prop:logistic-case-study}).
It allows to control the uniform deviations of the empirical Hessians from their means on~$\Theta_r(\theta_*)$, leading to the reduced sample size as per the following result.

\begin{theorem}[Simplified formulation of Theorems~\ref{th:crb-mine-improved}--\ref{th:crb-fake-improved}]
\label{th:crb-fake-improved-preview}
In addition to the premise of Theorem~\ref{th:crb-fake-improved-preview}, assume that the vectors~$\wt X(\theta) := [\ell''(Y,X^\T\theta)]^{1/2}X$ are subgaussian for~$\theta \in \Theta_r(\theta_*)$, cf.~\eqref{def:dikin}, with~$r = 1$ in the case of~\eqref{eq:sc-preview} 
and~$r = 1/\sqrt{\rho}$ in the case of~\eqref{eq:psc-preview}. 
Then bounds~\eqref{eq:crb-all-preview} in Theorem~\ref{th:crb-preview} are valid once
\begin{equation}
\label{eq:critical-improved-intro}
n \gsim 
\begin{cases}
\quad \; \deff \vee d \log d 
	& \text{under}~\eqref{eq:sc-preview}, \\
\rho \cdot \deff \vee d\log d 
 	& \text{under}~\eqref{eq:psc-preview}.
\end{cases}
\end{equation} 
\end{theorem}

The main technical challenge when proving this result is the fact that, while (pseudo) self-concordance of the \textit{population} risk over~$\Theta_r(\theta_*)$ with appropriate~$r$ follows from that of the loss function (by relating the directional derivatives of~$L(\theta)$ to the corresponding moments of~$\wt X(\theta)$), this fails to hold for the \textit{empirical} risk. 
Hence, we cannot uniformly control its Hessians on~$\Theta_r(\theta_*)$ by simply integrating the directional third derivatives of the empirical risk.
Instead, such control is attained by observing that self-concordance of the losses suffices to control Hessians in a smaller Dikin ellipsoid with radius~$O(1/d^{\kappa})$ for some~$\kappa \ge 1/2$, and combining this observation with a somewhat non-standard covering argument. 
We hypothesize that the bounds~\eqref{eq:critical-improved-intro} are optimal up to the~$\log(d)$ factor, i.e., ERM cannot provably achieve the nonasymptotic version of~\eqref{eq:crb-prob} in the regime where~$n$ is sublinear in~$\deff$ or~$d$. 
This hypothesis is motivated by the observation that~$n \gsim d$ is necessary to estimate the local norm~$\|\cdot\|_{\He}$, whereas~$n \gsim \deff$ is necessary to have~$\|\nabla L_n(\theta_*)\|_{\He} \le c$, which, in turn, allows to localize~$\wh\theta_n$ near~$\theta_*$.

In Section~\ref{sec:sparsity}, we extend some of the above results to the high-dimensional setup. 
Specifically, we obtain analogues of Theorem~\ref{th:crb-preview} for~$\ell_1$-regularized $M$-estimators, assuming that the optimal parameter~$\theta_*$ is $\s$-sparse, the matrices~$\Gr$ and~$\He$ are bounded in the operator norm, and the design is uncorrelated (the last assumption can in principle be relaxed).  
In the case of pseudo self-concordant losses (Theorem~\ref{th:crb-sparse-fake}), we replace~$\max(d,\deff)$ with~$O(\rho\s\log(d))$, both in the error rates and the minimal sample size requirements. 
Unfortunately, for canonically self-concordant losses, we do not get the expected improvement by~$\rho$ (see Theorem~\ref{th:crb-sparse-true}), and the bounds essentially remain the same as in the case of pseudo self-concordance. 
This, however, is not surprising, since sparsity and~$\ell_1$-regularization depend on the choice of the basis, and are not affine-invariant, which prevents us from fully exploiting self-concordance in the analysis by forcing to rely on the usual~$\ell_1$-~and~$\ell_2$-norms instead of~$\|\cdot\|_{\He}$. 
More detailed discussion of these results and their comparison with related work is deferred to Section~\ref{sec:sparsity}.


\subsection{Notation}
We write~$f \lesssim g$ or $f = O(g)$ to state that~$f(\cdot) \le Cg(\cdot)$ for any admissible arguments of~$f(\cdot)$,~$g(\cdot)$ and some constant~$C > 0$; analogously for~$f \gsim g$ or~$f = \Omega(g)$.
Notation~$f \approx g$ means~$f \lsim g \lsim f$.
$[n]$ is the set of integers~$\{1, 2, ..., n\}$. Throughout,~$\theta_*$ is the unique minimizer of $L(\theta)$, 
Similarly,~$\wh\theta_n$ is the minimizer of~$L_n(\theta)$, which will be (provably) unique with high probability in all cases.
Random vectors are denoted with capital letters (such as~$Z$), and matrices with bold capital letters (such as~$\bA$).~$\Id_d$ is the~$d \times d$ identity matrix. 
$\bA^\T$ is the transpose of~$\bA$. 
For two square matrices~$\bA_1,\bA_2$ of the same size, we write~$\bA_1 \prec \bA_2$ (resp.,~$\bA_1 \prccq \bA_2$) when~$\bA_2-\bA_1$ is positive (semi)definite.
We denote with~$\|\cdot\|_p$ the~$\ell_p$-norm on~$\R^d$ and the Schatten~$\ell_p$-norm of a matrix;
in particular,~$\|\bA\|_2$ is the Frobenius and~$\|\bA\|_\infty$ the operator norm. 
For~$\bA \succeq 0$,  we define the seminorm~$\|\theta\|_{\bA} := \|\bA^{1/2} \theta \|_2$. 

\section{Assumptions and examples}
\label{sec:assumptions}

Before introducing the assumptions, we remind that the loss~$\ell_Z: \Theta \to \R$ is modeled as~$\ell_Z(\theta) = \ell(Y,X^\T\theta)$ for some function~$\ell(y,\eta)$ on~$\cY \times \Rp$, where~$\cY$ is a subset of~$\R$, and~$\Rp$ is allowed to be either~$\R$ or the ray $\R^+$ of strictly positive numbers, which allows to encompass the exponential response model (cf.~Section~\ref{sec:ass-qsc}).
We refer to both~$\ell_Z(\theta)$ and~$\ell(y,\eta)$ as \textit{the loss}; which of the two we mean is clear from context.
The derivatives of~$\ell(y,\eta)$ are with respect to~$\eta$. 

\subsection{Self-concordance assumptions}
\label{sec:ass-qsc}
Let us introduce the assumptions related purely to the loss, rather than to the data distribution. 
Our standing assumption, which we silently use later on, is that the loss~$\ell_z(\cdot)$ is three times differentiable and convex on~$\Theta$ for any~$z \in \Z$.

We first present the assumption of~\textit{pseudo self-concordance}, introduced in~\cite{bach2010self-concordant} for the analysis of logistic regression.
The reader may refer to~\cite{sun2017generalized,tran2015composite,bach-moulines} for the uses of generalized self-concordance in the context of quasi-Newton algorithms.
\begin{customass}{SCa}
\label{ass:gsc-fake}
For any $y \in \Y$ and $\eta \in \Rp$, the loss satisfies 
\[
|\ell'''(y,\eta)| \le \ell''(y,\eta).
\]
\end{customass}

We also consider the \textit{canonical} self-concordance assumption first introduced in~\cite{nemirovski1994interior} in the context of interior-point algorithms. 
The constant~$2$ is standard in the literature, but can be replaced with arbitrary constant by rescaling the loss. 

\begin{customass}{SCb}
\label{ass:gsc-true}
For any $y \in \Y$ and $\eta \in \Rp$, the loss satisfies 
\[
|\ell'''(y,\eta)| \le 2[\ell''(y,\eta)]^{3/2}.
\]
\end{customass}


We now present some examples in which either of these assumptions is satisfied. 

\subsubsection{Generalized linear models over canonical exponential family}  
In generalized linear models (GLM) with canonical link function~(\cite{mccullagh1989generalized}), one has
\begin{equation}
\label{def:glm}
\ell(y,\eta) = - y\eta + a(\eta) - b(y),
\end{equation}
where the {\em cumulant}~$a(\eta): \Rp \to \R$ normalizes~$-\ell(y,\eta)$ to be a log-likelihood:
\[
a(\eta) = \log \int_{\cY} \exp(y\eta + b(y)) \, \textup{d}y.
\]
With~$\eta = X^\T\theta$, we have a conditional GLM for $Y$ given~$\eta = X^\T\theta$. 

Note that the second and third derivatives of $\ell(y,\eta)$ with respect to $\eta$ coincide with those of $a(\cdot)$, hence $\ell$ satisfies the basic smoothness/convexity assumption whenever $a(\cdot)$ is three times differentiable (as such,~$a(\cdot)$ must be convex). 
In fact, the cumulant derivatives correspond to the central moments of~$Y$:
\[
a'(\eta) = \E_{\eta}[Y], \quad
a''(\eta) = \E_{\eta}[(Y-\E_{\eta}[Y])^2], \quad
a'''(\eta) = \E_{\eta}[(Y- \E_{\eta}[Y])^3],
\] 
where~$\E_\eta[\cdot]$ is expectation with respect to the distribution with negative log-likelihood given by~\eqref{def:glm}. 
Hence, Assumption~\ref{ass:gsc-true} states precisely that the \textit{skewness} of the model distribution is bounded by a constant uniformly over $\eta \in \Rp$. 
This is the case in the \emph{exponential response} GLM where~$Y \sim \Exp(\eta)$ and~$a(\eta) = -\log(\eta)$ defined on~$\Rp = \R^+$.

On the other hand, Assumption~\ref{ass:gsc-fake} is satisfied whenever the third absolute central moment of~$Y$ is uniformly bounded by the variance of~$Y$, without the~${3}/{2}$ power. This is the case in \emph{Poisson regression:}~$Y \sim \text{Poisson}(\lambda)$ with~$\lambda = \exp(\eta)$; then~$b(y) = -\log(y!)$ and~$a(\eta) = \exp(\eta)$ so that~$a'''(\eta) = a''(\eta)$. This model is appropriate for count data where the rate of arrival itself depends multiplicatively on the canonical parameter~$\eta$; see, e.g.,~\cite{christensen2006log}. 
Perhaps most importantly, Assumption~\ref{ass:gsc-fake} is automatically satisfied in \textit{logistic regression} in which~$\cY = \{0,1\}$, and~$Y$ is modeled as a Bernoulli random variable with~$\Prob_{\eta}\{Y = 1\} = \sigma(\eta)$~where~$\sigma(\eta) = 1/(1+e^{-\eta})$ is the sigmoid function. In this case,~$a(\eta) = \log(1+e^\eta)$,~and one can verify that~$a'''(\eta) = a''(\eta)(1-2\sigma(\eta))$, so Assumption~\ref{ass:gsc-fake} is satisfied since~$|\sigma(\eta)| < 1$ for any~$\eta \in \R$.
Another way to see this is by looking at the cumulant and using that~$\cY = \{0,1\}$:
\[
|a'''(\eta)| \le |Y-\E_\eta[Y]| \cdot \E_{\eta}[(Y- \E_{\eta}[Y])^2] \le \E_{\eta}[(Y- \E_{\eta}[Y])^2] = a''(\eta).
\]


\subsubsection{Robust estimation}
Here, $\cY = \R$, and~$\ell(y,\eta) = \vphi(y-\eta)$ for some \textit{contrast}~$\vphi: \R \to \R$, a function minimized in the origin and usually even.
Crucially,~$\vphi(\cdot)$ must be globally Lipschitz-continuous, which guarantees robustness of the $M$-estimator, see~\cite{huber2011robust}.
On the other hand, from the statistical perspective, one can motivate contrasts that are locally quadratic, i.e., such that~$\vphi''(0)$ exists and is strictly positive, see, e.g.,~\cite{loh2017statistical}.\footnotemark
\footnotetext{However, this condition is \textit{not} necessary for the asymptotic normality of~$M$-estimator. 
For example, the sample median ($\vphi(t) = |t|$) in the model~$y = \theta + \varepsilon \in \R$ is asymptotically normal provided that the density of~$\varepsilon$ does not vanish at~$0$.}
These considerations, along with some minimax optimality results, lead to the Huber loss (see~\cite{huber1964}):
\begin{equation}
\label{def:huber-loss}
\vphi_{\tau}(t) = 
\left\{
\begin{aligned}
&{t^2}/{2}, 
&\quad|t| &\le \tau, \\
&\tau t    - {\tau^2}/{2}, 
&\quad |t| &> \tau.
\end{aligned}
\right.
\end{equation}
The Huber loss is parametrized by~$\tau > 0$, which allows to control the tradeoff between robustness and statistical performance. 
Indeed, on one hand,~$|\vphi'_\tau(t)| \le \tau$ for any $t \in \R$, and we make estimation more robust by decreasing~$\tau$; on the other hand, the variance of the corresponding $M$-estimator usually decreases as~$\tau$. 
However, finite-sample statistical analysis of the Huber loss is complicated by the fact that~$\vphi(t)$ is not $C^3$-smooth. 
This is also unfavorable from the algorithmic perspective, as it complicates the analysis of Newton-type algorithms for the computation of the $M$-estimator.
These issues can be circumvented if one instead uses \textit{pseudo-Huber losses}, which retain the favorable properties of the Huber loss, yet are $C^3$-smooth.
E.g., such are contrasts of the form~$\vphi_\tau(t) = \tau^2\vphi(t/\tau)$ with
\begin{align}
\label{def:pseudo-huber}
\vphi(t) = \log\left(\cosh(t)\right), \quad 
\vphi(t) = \sqrt{1 + {t^2}}-1.
\end{align}
In both cases, the resulting~$\phi''_{\tau}(\cdot)$ satisfies~$\phi''_{\tau}(0) = 1$~for any $\tau > 0$, and~$|\vphi'_{\tau}(t)| \le \tau$ for any $t \in \R$.
Moreover, simple algebra shows that both functions in~\eqref{def:pseudo-huber} satisfy Assumption~\ref{ass:gsc-fake} up to~$c = 3$, whence
$
|\vphi_{\tau}'''(t)| \le \tfrac{3}{\tau}\phi_\tau''(t).
$
As such, our theory is applicable to both these losses if they are properly renormalized. 

\subsubsection{Novel self-concordant losses}
\begin{figure}[t]
\begin{center}
\begin{minipage}{0.49\textwidth}
\centering
\includegraphics[width=0.96\textwidth]{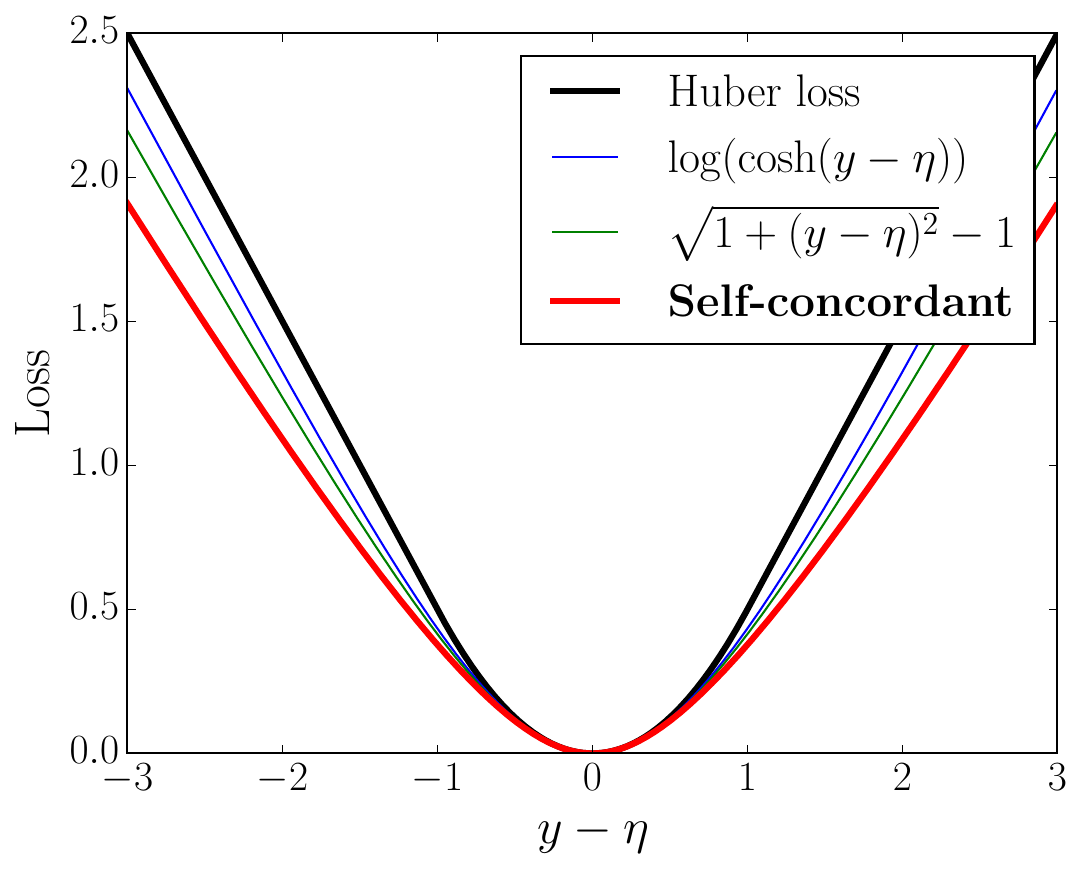}
\end{minipage}
\begin{minipage}{0.49\textwidth}
\centering
\includegraphics[width=0.96\textwidth]{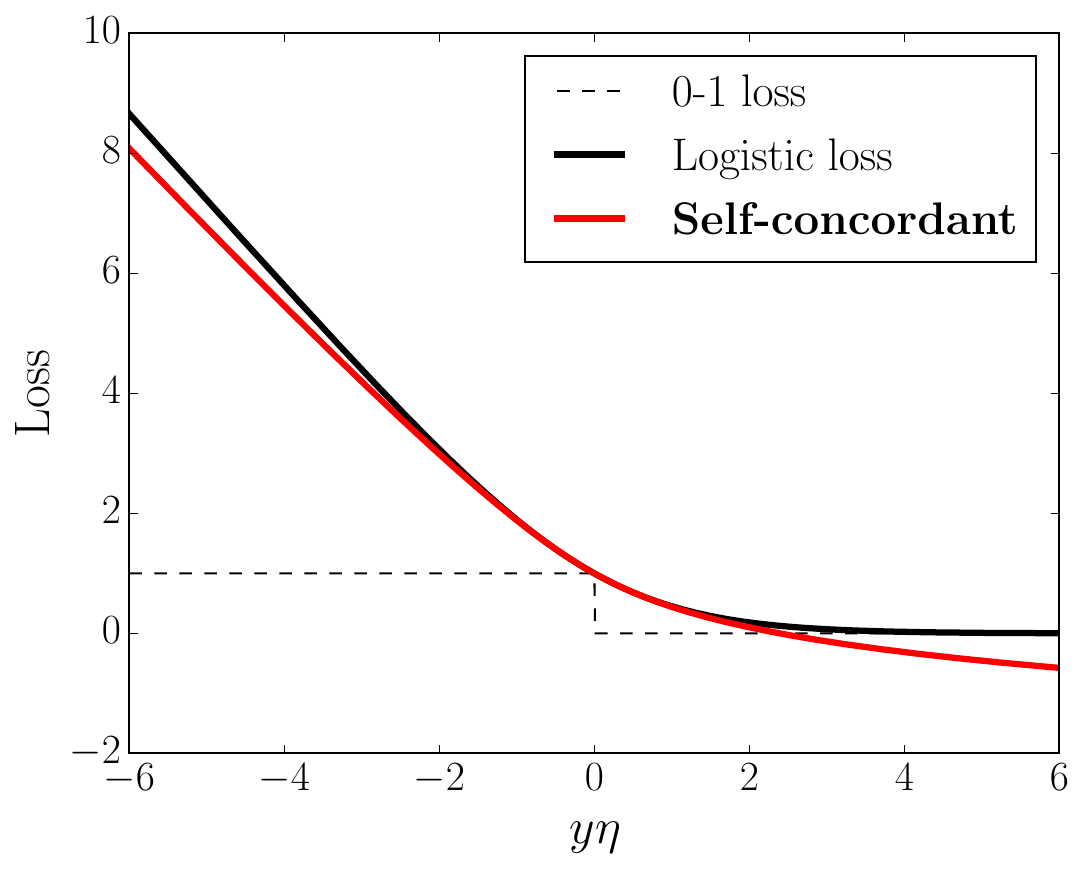}
\end{minipage}
\caption{{\em Left:} self-concordant pseudo-Huber loss, cf.~\eqref{def:our-robust-loss}. {\em Right}: self-concordant analogue of the logistic loss suitable for classification, cf.~\eqref{def:our-class-loss}.
Although our classification loss does not upper-bound the~0-1 loss on~$\R^+$, it can be lower-bounded as~$\Omega(-\log(y\eta))$ whenever~$y\eta > 0$.}
\label{fig:robust}
\end{center}
\end{figure}
Here we construct a \textit{canonically self-concordant} (up to a constant) pseudo-Huber loss, and similarly, a canonically self-concordant loss suitable for classification and similar to the logistic loss.
This construction is motivated by the observation that our theory has a somewhat tighter guarantee on the critical sample size (after which the fast rates occur) under the canonical self-concordance assumption. (However, in practice the situation might be different as we explore in Sec.~\ref{sec:numerical}.)
The key idea in this construction is that self-concordance is preserved under convex conjugation (see, e.g.,~\cite[Prop.~6]{sun2017generalized}), while at the same time one can control the range of the function through the domain of its convex conjugate (see~\cite{rockafellar1970convex}). 
Namely, consider~$\phi: (-1,1) \to \R^+$:
\begin{equation}
\label{eq:log-barrier}
\phi(u) = -\log(1-u^2)/2,
\end{equation}
that is, the negative log-barrier on~$[-1,1]$ normalized by~$\phi''(0) = 1$.
Its convex conjugate~$\vphi(t)$ can be explicitly computed:
\begin{equation}
\label{def:our-robust-loss}
\vphi(t) = \frac{1}{2} \left[ \sqrt{1 + {4t^2}} - 1 + \log\left(\frac{\sqrt{1+4t^2}-1}{2t^2}\right)\right].
\end{equation}
Note that~$\phi(\cdot)$ is even, satisfies~$\phi''(0) = 1$ and
$
|\phi'''(u)| \le 2\sqrt{2}[\phi''(u)]^{3/2},
$
since both functions~$\log(1 \pm u)$ satisfy Assumption~\ref{ass:gsc-true}.
By simple calculations detailed in Appendix~\ref{sec:extra-proofs},~$\vphi(t)$ defined in~\eqref{def:our-robust-loss} has all the same properties.
On the other hand, we have~$|\vphi'(t)| < 1$ since~$\phi(u)$ is a barrier on~$[-1,1]$.
Thus, $\vphi(t)$ has all properties desired for a robust loss, and besides is canonically self-concordant (albeit with constant $2\sqrt{2}$ instead of $2$). 
As illustrated in Figure~\ref{fig:robust}, the quality of approximating the Huber loss for the new loss is essentially as good as for the commonly used pseudo-Huber losses~\eqref{def:pseudo-huber}.
The new loss has a rescaled version~$\vphi_\tau(t) = \tau^2\vphi(t/\tau),$ for which~$\vphi_\tau''(0) = 1$,~$|\vphi_\tau'(t)| \le \tau$, and
$
|\vphi_\tau'''(t)| \le {(2/\tau)[\vphi_\tau''(t)]^{3/2}}.
$

Similarly, we can construct a self-concordant counterpart of the logistic loss suited for classification. In this case, we take~$\phi(u) = -\log(u(1+u))/2$, the normalized log-barrier of~$[-1,0]$,
whose convex conjugate is
\[
\phi^*(t) = \frac{1}{2}\left[-1-t + \sqrt{1+t^2} + \log\left(\frac{\sqrt{1 + t^2} - 1}{2t^2}\right)\right].
\]
The derivative of~$\phi^*(\cdot)$ must belong to~$(-1,0)$, and is canonically self-concordant (up to a constant) by the same reasoning as before. 
By rescaling and shifting it, we obtain the loss
\begin{equation}
\label{def:our-class-loss}
\ell(y,\eta) = 2+\frac{1}{2 \log 2}\left[-1-y\eta + \sqrt{1+(y\eta)^2} + \log\left(\frac{\sqrt{1 + (y\eta)^2} - 1}{2(y\eta)^2}\right)\right]
\end{equation}
which can be understood as a convex surrogate of the~$0$-$1$ loss similar to the logistic loss, see Figure~\ref{fig:robust}. 
However, this loss is negative for~$y\eta > 2.4,$ and therefore does not globally upper-bound the~0-1 loss. 
Fortunately, its right branch can be lower-bounded with~$\Omega(-\log(y\eta))$, so the resulting ``leakage'' is insignificant. 
On the other hand, this defect is unavoidable: one can show that a canonically self-concordant function on~$\R^+$ cannot have a horizontal asymptote: this would imply~$\vphi''(t) \to_{t \to +\infty} 0$, contradicting Assumption~\ref{ass:gsc-true} reformulated as~$|([\vphi''(t)]^{-1/2})'| \le 1$. 
Finally, let us remark that the ``leakage'' effect can also be quantified using the so-called calibration theory~\cite{bartlett2006convexity}.

\subsection{Distribution assumptions}
\label{sec:ass-distrib}

\paragraph{Preliminaries.}
We now introduce additional assumptions that ar e related to the distribution of the design scaled by the derivatives of the loss at the true optimum~$\theta_*$. 
All these assumptions are fully \textit{local}, i.e., they only concern the true optimal point~$\theta_*$.
We begin with the basic assumptions. First, we assume the existence of the matrices
\[
\Cov := \E[XX^\T], \quad 
\Gr := \E[\nabla \ell_Z(\theta_*)\nabla \ell_Z(\theta_*)^\T], \quad 
\He := \E[\nabla^2 \ell_Z(\theta_*)];
\]
Generally,~$\Cov \ne \He$ (unless for least-squares), and~$\Gr \ne \He$ (unless in a well-specified model). 
Recall that $\E[\nabla \ell_Z(\theta_*)] = 0$; as such,~$\Gr$ is the covariance matrix of~$\nabla \ell_Z(\theta_*).$
For future reference we also note that, for any~$\theta \in \Theta$, one has 
\begin{equation}
\label{eq:derivatives}
\nabla \ell_Z(\theta) = \ell'(Y,X^\T\theta) X, \quad \nabla^2 \ell_Z(\theta) = \ell''(Y,X^\T\theta) X X^\T.
\end{equation}
We assume that~$X^\T \theta \in \Rp$ for any~$\theta \in \Theta$ and~$X \in \cX$.
This assumption is non-trivial only when~$\Rp = \R^+$ which is of interest in the exponential response model. 
In this case, one can assume~$\Theta \subseteq \R^d_+$ and~$\cX \subseteq \R^d_+$ where~$\R_+^d$ is the positive orthant, or replace the pair~$(\R^d_+, \R^d_+)$ with other pairs of mutually dual convex cones in~$\R^d$.

Following~\cite{vershynin2010introduction}, we use the formalism of subgaussian, or~$\psi_2$-norms.
The $\psi_2$-norm~$\|\xi\|_{\psi_2}$ of a random variable~$\xi \in \R$ can be defined in a number of equivalent ways (see Appendix~\ref{sec:prob-tools}), e.g., as
$
\|\xi\|_{\psi_2} := \{\sigma > 0 : \E [e^{\xi^2/\sigma^2}] \le e \}.
$
This definition extends to random vectors~$Z \in \R^d$ in a standard way: 
\[
\|Z\|_{\psi_2} := \sup \{ \| \langle Z, \theta \rangle \|_{\psi_2}: \|\theta\|_2 \le 1 \}.
\]
In other words, $\|Z\|_{\psi_2}$ is the maximal $\|\cdot\|_{\psi_2}$-norm for all one-dimensional marginals of~$Z$. 
See Appendix~\ref{sec:prob-tools} on more details on subgaussian random variables.



\begin{customass}{D0}
\label{ass:design-subg}
The decorrelated design is subgaussian: it holds
\[
\|\Cov^{-1/2} X \|_{\psi_2} \le K_0.
\]
\end{customass}
Assumption~\ref{ass:design-subg} is often satisfied with a constant~$K_0$ not depending on~$n$ or~$d$.
For example, this is the case for zero-mean Gaussian design $X \sim \cN(0,\Cov)$, or design with independent Bernoulli components.
Moreover, it can be shown that affine transformation of the design~$X$ that satisfies Assumption~\ref{ass:design-subg} also satisfies it, with at worst twice larger~$K_0$ (see Lemma~\ref{lem:affine-comparison} in Appendix).

\begin{customass}{D1}
\label{ass:first-subg}
The decorrelated loss gradient at~$\theta_*$ is subgaussian:
\[
\| \Gr^{-1/2} \nabla \ell_Z(\theta_*)\|_{\psi_2} \le K_1.
\]
\end{customass}
Note that Assumption~\ref{ass:first-subg} can be reformulated in terms of the design vector scaled by the loss derivative at~$\theta_*$ since~$\nabla \ell_Z(\theta_*) = \ell'(Y,X^\T\theta_*) X$.
Similarly, we can consider the random vector
\begin{equation}
\label{def:calibrated-design}
\csX := [\ell''(Y,X^\T\theta_*)]^{1/2}X
\end{equation}
which we call the~\textit{calibrated design}.
Note that~$\csX$ is linked with~$\He$ by~$\E[\csX \csX^\T] = \He$, cf.~\eqref{eq:derivatives}.
As stated next, we assume that the calibrated design is subgaussian. 
This allows to control the deviations of~$\He_n$ using Theorem~\ref{th:covariance-subgaussian-simple} in Appendix.

\begin{customass}{D2}
\label{ass:second-subg}
The calibrated design~$\csX := [\ell''(Y,X^\T\theta_*)]^{1/2}X$ satisfies
\[
\| \He^{-1/2} \csX\|_{\psi_2} \le K_2.
\]
\end{customass}
Assumption~\ref{ass:second-subg} can be reformulated in terms of the loss Hessian~$\nabla^2 \ell_Z(\theta_*)$ due to~\eqref{eq:derivatives}.
However, this formulation does not give new ideas, and we omit it.
\begin{remark}
\label{rem:subgaussian-norm-lower}
The quantities~$K_0$,~$K_1$,~$K_2$ are necessarily bounded with some absolute constant \textit{from below}.
This fact follows from the moment characterization of the $\psi_2$-norm (Item~2 of Lemma~\ref{lem:subg-properties} in Appendix), combined with the bound~$(\E|\xi|^4)^{1/4} \ge (\E|\xi|^2)^{1/2}$ for any random variable~$\xi \in \R$, and allows to simplify the formulation of the subsequent results.
\end{remark}
\begin{remark}
\label{rem:distrib-restrictive}
Assumptions~\ref{ass:first-subg}--\ref{ass:second-subg} are quite restrictive, even under Assumption~\ref{ass:design-subg}.
In particular, in GLMs with canonical link function (cf.~Section~\ref{sec:ass-qsc}), the calibrated design at point~$\theta_*$ is given by~$\wt X(\theta) = [a''(X^\T\theta)]^{1/2}X$ where~$a(\eta)$ is the cumulant function. 
The transform~$[a''(X^\T\theta)]^{1/2}$ that scales~$X$ along a direction~$\theta$ can be highly-non-linear, breaking subgaussianity for~$\wt X(\theta)$. For example, Assumption~\ref{ass:second-subg} does not hold in Poisson regression.
Another limitation of our approach is that the constants~$K_1,K_2$ in Assumptions~\ref{ass:first-subg}--\ref{ass:second-subg} can depend on the magnitude of~$\theta_*$. 
In fact, for logistic regression with Gaussian design~$X \sim \cN(0,\Cov)$, one has
\begin{align*}
K_2 &\lsim \log(1+\|\theta_*\|_\Cov) \sqrt{1+\|\theta_*\|_\Cov}.
\end{align*}
This proof of this estimate (see Appendix~\ref{sec:subgaussian-check}) is highly non-trivial, and relies on the Gaussianity of~$X$. We also show that 
\[
K_1 \lsim 1+\|\theta_*\|_\Cov^{3/2}
\]
if the logistic model for~$Y|X$ is well-specified.
This improves to~$K_1 \lsim 1+\|\theta_*\|_\Cov^{1/2}$ if the subgaussian norm~$\|\cdot\|_{\psi_2}$ is replaced with the subexponential norm~$\|\cdot\|_{\psi_1}$ (see~Appendix~\ref{sec:subgaussian-check} and Section~\ref{sec:basic} for details).
In other applications, one should carefully verify Assumptions~\ref{ass:first-subg}--\ref{ass:second-subg}, bounding the constants~$K_1$ and~$K_2$. 
This can be a non-trivial task itself, especially when the distribution of~$X$ is unknown.
\end{remark}

Finally, for pseudo self-concordant losses we need \textit{compatibility} of~$\Cov$ and~$\He$.
\begin{customass}{C}
\label{ass:curvature}
It holds~$\Cov \prccq \rho \He$ for some~$\rho < \infty$. 
\end{customass}
Assumption~\ref{ass:curvature} has already appeared in the statistical analysis of logistic regression in~\cite{bach-moulines}.
Note that the simplest \textit{generic} upper bound for~$\rho$ is  
\begin{equation}
\label{eq:rho-bound}
\rho \le {\left(\inf_{(y,\eta) \in \cY \times \Rp} \ell''(y,\eta)\right)^{-1}},
\end{equation}
and unless~$\ell''(y,\cdot)$ is strictly convex on~$\Rp$ (which is usually not the case), this bound is vacuous.
On the other hand, the infinum in~\eqref{eq:rho-bound} can be taken on the subset of~$\Rp$ corresponding to possible values of~$X^\T\theta_*$, but such bound can still be very conservative: for example, it only gives~$\rho = O(e^{RD})$ in the case of logistic regression with~$\|X\|_2 \le R$ \textit{a.s.} and~$\Theta = \{\theta \in \R^d: \|\theta\|_2 \le D\}$.
However, the \textit{actual} value of~$\rho$ depends on the true distribution of the data, and is usually much smaller, see, e.g., dicsussion in~\cite[Sections~3.1,~4.2]{bach-moulines} for the case of logistic regression.
For example, consider a ``quasi well-specified'' robust regression model:~$\ell(Y,X^\T\theta) = \vphi(Y-X^\T\theta)$ with even contrast~$\vphi(\cdot)$ and unconstrained parameter. Suppose that the true distribution of~$Y$ is given by 
~$Y = X^\T\theta_* + \veps.$
with~$\veps$ being independent from~$X$, zero-mean, and symmetrically distributed. One can check that in this case,~$L(\theta)$ is minimized at~$\theta_*$, and~$\rho = 1/\E[\vphi''(\veps)]$.
On the other hand, the worst-case bounds on~$\rho$ can be enforced if the data distribution is chosen \textit{adversarially}. 
In particular, for logistic regression~\cite{hazan2014logistic} construct an adversarial distribution that enforces~$\rho = \Omega(e^{RD})$ as long as~$n = O(e^{RD})$.

\section{Results under minimal assumptions}
\label{sec:basic}

In this section, we present extensions of the asymptotic deviation bound~\eqref{eq:crb-prob} to the finite-sample regime {\em under minimal assumptions.} 
We then refine these results in Section~\ref{sec:improved}, under a slightly strengthened version of Assumption~\ref{ass:second-subg}, through a more subtle analysis. 
In the proofs, we use some probabilistic tools collected in~Appendix~\ref{sec:prob-tools}; in particular, we use
deviation bounds for the quadratic forms (Theorem~\ref{th:subg-quad}) and for sample covariance matrices (Theorem~\ref{th:covariance-subgaussian-simple}) of subgaussian vectors. 
We also use technical results on (pseudo) self-concordant functions collected in Appendix~\ref{sec:qsc-properties}. Some of them appear to be new, and are of independent interest. 
To improve readability, we defer the proofs to Appendix~\ref{sec:extra-proofs}.

\paragraph{Preliminaries.}
In the results which we are about to present, there is a technical difficulty arising due to the unboundedness of the vectors~$X$ and~$\wt X$, cf.~\eqref{def:calibrated-design}.\footnote{
This issue arises due to working with individual losses; as a result, it does not appear in our refined results, presented in Section~\ref{sec:improved}, in which we analyze the empirical risk ``as a whole''.} 
To this end, we observe that, due to Assumptions~\ref{ass:design-subg} and~\ref{ass:second-subg}, these vectors admit~$O(\sqrt{d})$ high-probability bound on their norms -- more precisely, the events
\[
\cE_0 := \left\{\|X\|_{\Cov^{-1}} \lesssim K_0 \sqrt{d \log \left({e}/{\delta}\right)} \right\}, \quad 
\cE_2 := \left\{ \|\csX\|_{\He^{-1}} \lesssim K_2 \sqrt{d \log \left({e}/{\delta}\right)} \right\}
\]
hold with probability~$\ge 1-\delta$, correspondingly, under~Assumptions~\ref{ass:design-subg} and~\ref{ass:second-subg}.
To exploit this fact, we replace the population risk~$L(\theta)$ with the \textit{restricted risks}:
\begin{equation}
\label{eq:restricted-risks}
L_{\cE_0}(\theta) := \E[\ell_Z(\theta)\ind\left\{ X \in \cE_0 \right\}]; \quad L_{\cE_2}(\theta) := \E[\ell_Z(\theta)\ind\{\csX \in \cE_2 \}], 
\end{equation}
where we exclude from averaging the low-probability outcomes in which the norms of~$X$ and~$\csX$ are too large.  
Provided that~$\delta$ is small enough, we can show that~$\nabla L_{\cE_0}(\theta_*)\approx \nabla L_{\cE_2}(\theta_*) \approx 0$ and~$\nabla^2 L_{\cE_0}(\theta_*) \approx \nabla^2 L_{\cE_2}(\theta_*) \approx \nabla^2 L(\theta_*)$, so that the second-order structure of the population risk is preserved; at the same time, we can now work with~$X$ and~$\wt X$ as if they were almost surely bounded.

We now present our basic result for~$M$-estimators with self-concordant losses.
\begin{theorem}
\label{th:crb-fake-complex}
Let Assumptions~\ref{ass:gsc-fake},~\ref{ass:design-subg},~\ref{ass:first-subg}, \ref{ass:second-subg}, and~\ref{ass:curvature} hold. 
Whenever
\begin{equation}
\label{eq:n-fake-complex}
n \gtrsim \max\left\{K_2^4 \left(d + \log\left({1}/{\delta}\right)\right), \;\; \rho K_0^2 K_1^2 \deff \, d \log\left({ed}/{\delta}\right)\right\},
\end{equation}
with probability at least~$1-\delta$ it holds
\begin{align}
\label{eq:score-finite-complex}
\|\nabla L_n(\theta_*)\|_{\He^{-1}}^2 &\lsim \frac{K_1^2\deff\log\left({e}/{\delta}\right)}{n},\\
\label{eq:crb-finite-complex}
\|\wh\theta_n - \theta_*\|_{\He}^2 &\lsim \|\nabla L_n(\theta_*)\|_{\He^{-1}}^2.
\end{align}
Moreover, one has 
\begin{equation}
\label{eq:restricted-risk-bound-neat}
L_{\cE_0}(\wh\theta_n) - L_{\cE_0}(\theta_*) \lsim \frac{K_1^2\deff\log\left({e}/{\delta}\right)}{n}
\end{equation}
provided that
\begin{equation}
\label{eq:restricted-risk-condition}
\delta \lesssim \min \left\{ \left(\frac{1}{\sqrt{n \log(e\deff)}}\right)^{1 + 1/\log(\deff)}, \left(\frac{1}{K_2^2 d \log(ed)}\right)^{1 + 1/\log(d)}\right\}.
\end{equation}
\end{theorem}
The main message of Theorem~\ref{th:crb-fake-complex} is that, under minimal assumptions, the ``quadratic'' behavior of the population risk, as given by~\eqref{eq:score-finite-complex}--\eqref{eq:restricted-risk-bound-neat}, is guaranteed for sample sizes growing quadratically in parameter dimension -- more precisely, for~$n = \wt\Omega(\rho \cdot d \cdot \deff)$, cf.~the second bound in~\eqref{eq:n-fake-complex}, where~$\wt\Omega$ hides subgaussian constants and the logarithmic factor in~$\delta$. 
Technically, the curvature parameter~$\rho$ appears in~\eqref{eq:n-fake-complex} because of the ``incorrect'' power of the second derivative in Assumption~\ref{ass:gsc-fake} as compared to power~$3/2$ in~Assumption~\ref{ass:gsc-true}. 
Indeed, for canonically self-concordant losses, the factor~$\rho K_0^2$ in the bound for the critical sample size get replaced with~$K_2^2$, and Assumptions~\ref{ass:curvature} and~\ref{ass:design-subg} are not needed. 


\begin{theorem}
\label{th:crb-mine-complex}
Let Assumptions~\ref{ass:gsc-true},~\ref{ass:first-subg},~\ref{ass:second-subg} hold, and assume that~$\delta$ satisfies~\eqref{eq:restricted-risk-condition}. Then,~\eqref{eq:score-finite-complex}--\eqref{eq:restricted-risk-bound-neat} are satisfied, with~$L_{\cE_2}(\cdot)$ instead of~$L_{\cE_0}(\cdot)$, whenever 
\begin{equation}
\label{eq:n-mine-complex}
n \gtrsim \max\left\{K_2^4 \left(d + \log\left({1}/{\delta}\right)\right), \;K_1^2 K_2^2 \deff \, d  \log\left({ed}/{\delta}\right)\right\}.
\end{equation}
\end{theorem}

We also note that both of the above results include a technical condition~\eqref{eq:restricted-risk-condition} that does not minimal violation probability~$\delta$.  
This condition is mild, as the admissible~$\delta$ depends polynomially on~$n$ and~$d$. 
Moreover, this condition can be dropped if one reinforces Assumption~\ref{ass:design-subg} (resp.,~\ref{ass:second-subg}) by assuming that~$\Cov^{-1/2} X$ (resp.,~$\He^{-1/2} \wt X$) is almost surely bounded. The corresponding modifications of Theorems~\ref{th:crb-fake-complex}--\ref{th:crb-mine-complex} are given in the {\em arXiv} version of this paper~\cite[Thms~3.1--3.2]{ostrovskii2018finite}.

As we previously discussed (cf.~Remark~\ref{rem:distrib-restrictive}), Assumptions~\ref{ass:design-subg}--\ref{ass:second-subg}, although local, are quite restrictive, as they assume light-tailed behavior.
Next we discuss how these assumptions can be relaxed.

\paragraph{Extension to heavy-tailed distributions.}
To extend the results, we might use the confidence-boosting technique based on a version of the multi-dimensional sample median as proposed in~\cite{hsu2016loss}. 
This allows to completely get rid of Assumption~\ref{ass:first-subg}, only assuming the existence of the covariance matrix~$\Gr(\theta_*)$. To use the technique, one first divides the sample into~$k = \log(e/\delta)$ non-overlapping subsamples, and computes the corresponding~$M$-estimators $\wh \theta^{(1)}, ..., \wh\theta^{(k)}$ over each subsample. Then, one aggregates them through~\cite[Algorithm~3]{hsu2016loss}, by using
\[
\text{dist}^{(i)}(\theta) := \|\theta-\wh\theta^{(i)}\|_{\wh\He^{(i)}}, \quad \wh\He^{(i)} := \nabla^2 L_n(\wh\theta^{(i)})
\] 
as the random distance oracle related to~$\wh\theta^{(i)}$. The final estimator is~$\wh\theta^{(\wh i)}$ with
\[
\wh{i} \in \Argmin_{i \in [k]} \left\{ \text{Median} \left[ \left(\text{dist}^{(j)}(\wh\theta^{(i)})\right)_{j \in [k]} \right] \right\}.
\]
By Chebyshev's inequality, each~$\wh\theta^{(i)}$ admits a fixed-probability version of~\eqref{eq:score-finite-complex}, say, with~$\delta = 2/3$. On the other hand, for each~$i \in [k]$, one has 
\[
\half \He \prccq \nabla^2 L_n(\wh\theta^{(i)}) \prccq 2\He
\]
with fixed probability.
Indeed,
$
\half L(\theta_*) \prccq L_n(\theta_*) \prccq 2 L(\theta_*)
$ 
by the analysis in Theorems~\ref{th:crb-fake-complex}--\ref{th:crb-mine-complex}.
Then, our integration argument (cf. the proof of~Lemmas~\ref{lem:bach-1d}--\ref{lem:mine-1d} in appendix) allows to relate~$L_n(\theta_*)$ to~$L_n(\wh\theta^{(i)})$ and results in
$
\half L_n(\theta_*) \prccq L_n(\wh\theta^{(i)}) \prccq 2 L_n(\theta_*).
$ 
Finally, the estimators over different subsamples are mutually independent. 
Thus, we can apply Theorem~11 of~\cite{hsu2016loss}, which finally yields~\eqref{eq:restricted-risk-bound-neat}. 

A similar technique also allows to somewhat weaken Assumptions~\ref{ass:design-subg} and~\ref{ass:second-subg}, replacing the subgaussian norm~$\|\cdot\|_{\psi_2}$ with the subexponential norm~$\|\cdot\|_{\psi_1}$ at the expense of an extra logarithmic factor. (By definition,~$X \in \R^d$ satisfies~$\|X\|_{\psi_1} \le K$ if for any~$u$ on the unit sphere one has~$(\E[|\lang X, u \rang|^p])^{1/p} \lsim K p$, compared to~$K\sqrt{p}$ in the case of~$\|\cdot\|_{\psi_2}$, cf. Lemma~\ref{lem:subg-properties} in Appendix.) 
This can be done by replacing Theorem~\ref{th:covariance-subgaussian-simple} (high-probability bound for subgaussian distributions) with \cite[Theorem~5.48]{vershynin2010introduction} (fixed-probability bound for subexponential distributions), controlling~$\E[\max_{i\in [n]} \|X_i\|_{\He}^2]$ and~$\E[\max_{i\in [n]} \|\wt X_i\|_{\He}^2]$ via Bernstein's inequality (Theorem~\ref{th:subg-quad} in Appendix).
However, this technique is limited to subexponential distributions of~$X$ and~$\wt X$ as required by~\cite[Theorem~5.48]{vershynin2010introduction}. 

On the other hand, replacing Assumptions~\ref{ass:design-subg} and~\ref{ass:second-subg} with finite-moment assumptions (ideally, finite kurtoses of vectors~$X$ and~$\wt X$) is challenging. 
First of all, sample covariance estimators~$\wh \Cov$ and~$\wh \He$ would have to be replaced by some estimators~$\bar \Cov$ and~$\bar \He$ that admit affine-invariant bounds of the form
\begin{equation}
\label{eq:psd-bounds}
\half \Cov \prccq \bar \Cov \prccq 2 \Cov, \quad 
\half \He \prccq \bar \He \prccq 2 \He
\end{equation}
with high probability, under the existence of only finite moments (ideally, the fourth moment) of~$X$ and~$\wt X$ in any direction. 
Such estimators were recently obtained in~\cite{heavy-covariance} based on the iterative appication of the truncated covariance estimator analyzed in~\cite{wei2017estimation}. 
Computing such an estimator on the hold-out sample would allow to get rid of Assumption~\ref{ass:design-subg} in Theorem~\ref{th:crb-fake-complex}. 
However, this technique by itself does not allow relax Assumption~\ref{ass:second-subg}, note first that we do not know the true minimizer~$\theta_*$, and hence cannot directly compute the robust estimator~$\bar \He$. 
A possible remedy, leading to the extension of Theorems~\ref{th:crb-fake-complex}--\ref{th:crb-mine-complex}, is to apply an approximation technique on top of the affine-invariant covariance estimator, similarly to the one used below to prove Theorems~\ref{th:crb-mine-improved}--\ref{th:crb-fake-improved} with improved critical sample size.
As we will discuss in the end of Section~\ref{sec:improved}, 
this would allow to get rid of Assumptions~\ref{ass:design-subg} and~\ref{ass:second-subg} in Theorems~\ref{th:crb-fake-complex}--\ref{th:crb-mine-complex} but not in Theorems~\ref{th:crb-mine-improved}--\ref{th:crb-fake-improved}.
\section{Improved results: near-linear critical sample size}
\label{sec:improved}

As we demonstrate next, the previously obtained bounds on the critical sample size can be improved: essentially, the product of $\deff$ and~$d$ can be replaced with their maximum. 
This requires to slightly strengthen Assumption~\ref{ass:second-subg} as follows.

\begin{customass}{D2$^*$}
\label{ass:second-subg++}
The calibrated design~$\csX(\theta) := [\ell''(Y,X^\T\theta)]^{{1}/{2}}X$ 
satisfies
\[
\| \He(\theta)^{-1/2} \csX(\theta)\|_{\psi_2} \le \bar{K}_2(r),
\]
where~$\He(\theta) = \E[\csX(\theta) \csX(\theta)^{\T}]$, for any~$\theta$ in the Dikin ellipsoid $\Theta_r(\theta_*)$ given by
\[
\Theta_r(\theta_*) := \{ \theta \in \R^d: \|\theta - \theta_*\|_{\He(\theta_*)} \le r\}.
\] 
\end{customass}
Note that Assumption~\ref{ass:second-subg} corresponds to Assumption~\ref{ass:second-subg++} with~$r = 0$, the correspondence being given by~$K_2 = \bar K_2(0)$. On the other hand, the strengthened assumption is still \textit{local}, i.e., it only concerns the points $r$-close to~$\theta_*$, in the local Hessian metric, rather than in the whole domain~$\Theta$.
With the new assumption at hand, we now state the improved result for canonically self-concordant losses.
\begin{theorem}
\label{th:crb-mine-improved}
Assume~\ref{ass:gsc-true},~\ref{ass:first-subg}, and~\ref{ass:second-subg++} with~$r \gsim 1$.
Then,~\eqref{eq:score-finite-complex},~\eqref{eq:crb-finite-complex} and 
\begin{equation}
\label{eq:excess-risk-bound-neat}
L(\wh\theta_n) - L(\theta_*) \lsim \frac{K_1^2\deff\log\left({e}/{\delta}\right)}{n}
\end{equation}
hold as long as
\begin{equation}
\label{eq:n-mine-improved}
n \gtrsim \max\left\{\bar K_2^4(r) d\log\left({ed}/{\delta}\right), \;K_1^2 \bar K_2^6(r) \deff  \log\left({e}/{\delta}\right)\right\}.
\end{equation}
\end{theorem}
Let us briefly explain the key ideas behind this result. 
First of all, recall that the extra factor~$d$ in the bound of Theorem~\ref{th:crb-mine-complex} appears because self-concordance of the \textit{individual losses} only allows to obtain a second-order approximation of the empirical risk in a small Dikin ellipsoid with radius~$O(1/\sqrt{d})$, due to the fact that~$\|\wt X\|_{\He^{-1}} = \Omega(\sqrt{d})$ with high probability. 
This second-order approximation then allows to localize the estimate as soon as~$\|\nabla L_n(\theta_*)\|_{\He^{-1}}$ becomes smaller than the radius of the ellipsoid in which such an approximation holds, cf. the proof of Proposition~\ref{prop:local}.
Hence, the extra factor~$d$ would be eliminated if we managed to provide a second-order Taylor approximation of~$L_n(\theta)$ in the constant-radius Dikin ellipsoid~$\Theta_{c}(\theta_*)$.
The immediately arising difficulty is that unlike the individual losses, the empirical risk is \textit{not} self-concordant, hence, the desired Taylor approximation cannot be obtained purely by integration.
Instead, we conduct a somewhat non-standard argument (see Figure~\ref{fig:covering}) which combines (i) self-concordance of the \textit{population risk} following from Assumption~\ref{ass:second-subg++}; (ii) self-concordance of the individual losses; (iii) a covering argument in which ellipsoid~$\Theta_c(\theta_*)$  is covered with small ellipsoids with radius~$O(1/d^\gamma)$ for some~$\gamma \ge 1/2$. 
In particular, we choose~$\gamma = 2$: this simplifies the calculations in the final step of the proof without breaking~\eqref{eq:n-mine-improved} since~$d^\gamma$ enters the analysis under logarithm, when bounding covering numbers.

\begin{figure}[t]
\centering
\includegraphics[width=0.8\textwidth]{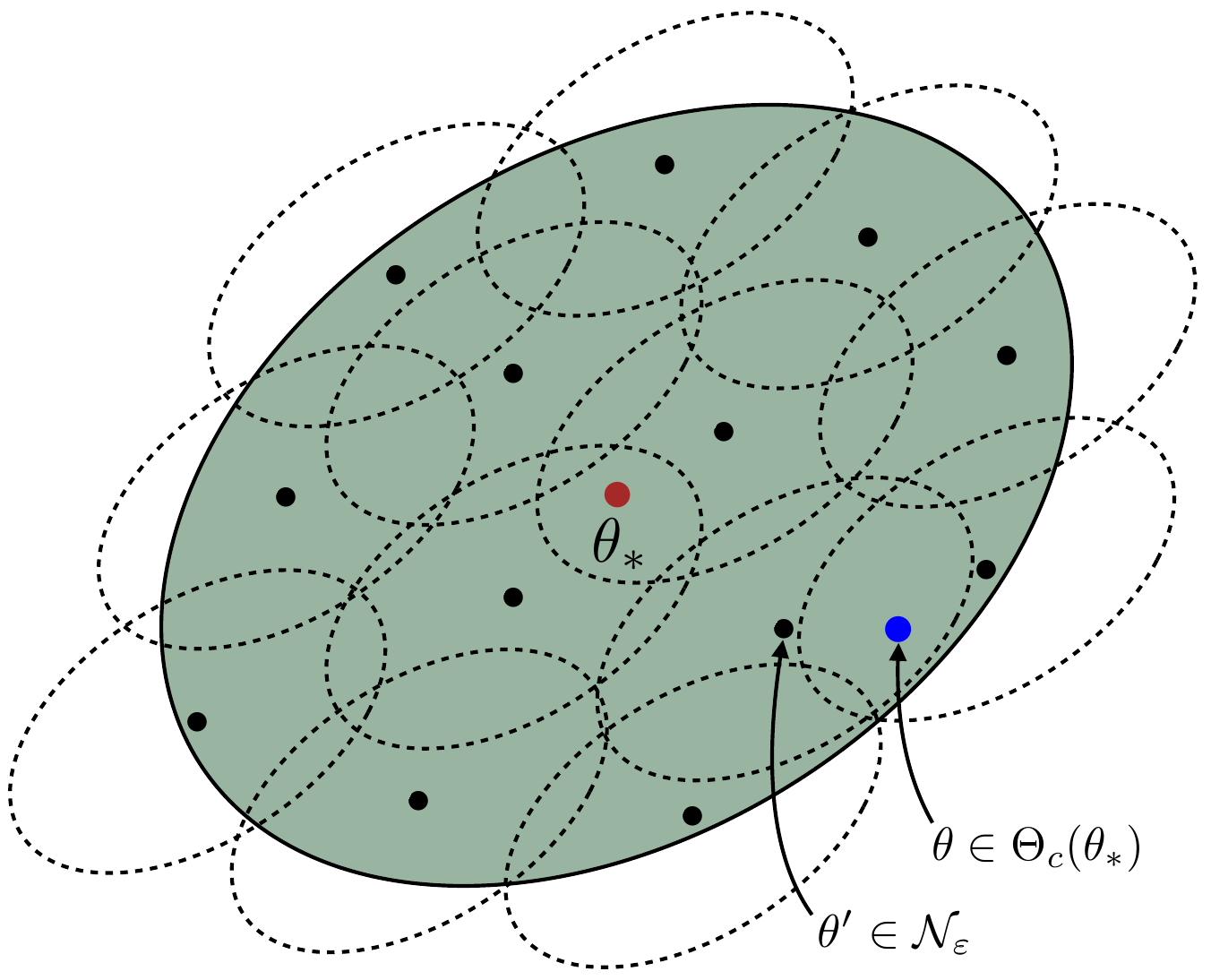}
\caption{
The crucial step in the proof of Theorem~\ref{th:crb-mine-improved} is to ensure that~$\half \He(\theta_*) \prccq \He_n(\theta) \prccq 2\He(\theta_*)$ holds with high probability uniformly over the constant-radius Dikin ellipsoid~$\Theta_c(\theta_*)$ (in green).
Using Assumption~\ref{ass:second-subg++}, we first prove that~$\half \He(\theta_*) \prccq \He(\theta)\prccq 2\He(\theta_*)$ for any~$\theta \in \Theta_c(\theta_*)$. On the other hand, self-concordance of individual losses provides a constant-order approximation of~$\He_n(\cdot)$ within a smaller ellipsoid with radius~$O(1/d^{\gamma})$, for some~$\gamma \ge 1/2$, around $\theta$.
As such, the problem is reduced to the control of the uniform deviations of~$\He_n(\theta)$ from~$\He(\theta)$ for~$\theta \in \cN_\veps$, where~$\cN_\veps$ is the epsilon-net of~$\Theta_1(\theta_*)$ with respect to the norm~$\|\cdot\|_{\He(\theta_*)}$ with~$\veps = O(1/d^{\gamma})$.  This is done by using Theorem~\ref{th:covariance-subgaussian-simple}.
}
\label{fig:covering}
\end{figure}

Next we present a counterpart of Theorem~\ref{th:crb-mine-improved} for pseudo self-concordant losses.
As one might expect, the bound on the critical sample size degrades by~$\rho$.
\begin{theorem}
\label{th:crb-fake-improved}
Assume~\ref{ass:gsc-fake},~\ref{ass:design-subg},~\ref{ass:first-subg},\ref{ass:curvature}, and~\ref{ass:second-subg++} with~$r \gsim 1/\sqrt{\rho}$. Then,~\eqref{eq:score-finite-complex},~\eqref{eq:crb-finite-complex} and~\eqref{eq:excess-risk-bound-neat} hold whenever
\begin{equation}
\label{eq:n-fake-improved}
n \gtrsim \max\left\{\bar K_2^4(r) d\log\left({ed}/{\delta}\right), \; \rho K_0^2 K_1^2 \bar K_2^4(r) \deff  \log\left({e}/{\delta}\right)\right\}.
\end{equation}
\end{theorem}
The two results above merit some discussion.

First, note that, in the case of pseudo self-concordance, the radius of the Dikin ellipsoid in which Assumption~\ref{ass:second-subg++} is required to hold is~$\sqrt{\rho}$ times smaller than in the case of canonical self-concordance. 
As it will become clear from the proof of Theorem~\ref{th:crb-fake-improved}, this deflation is related to the fact that we cannot control the Hessians of~$L(\theta)$ over Dikin ellipsoids with a larger radius, even when Assumption~\ref{ass:second-subg++} holds on such an ellipsoid. 
On the other hand, decreasing the radius $r$ of the Dikin ellipsoid allows to control~$\bar K_2(r)$: in Appendix~\ref{sec:subgaussian-check} we show that, in logistic regression with Gaussian design~$X \sim \cN(0,\Cov)$,
\[
\bar K_2^2(r) \lsim \bar K_2^2(0) + r\sqrt{\rho}.
\] 
Thus Assumption~\ref{ass:second-subg++} with~$r = 1/\sqrt{\rho}$ is essentially \textit{equivalent} to Assumption~\ref{ass:second-subg}.

Second, note that the second threshold in~\eqref{eq:n-mine-improved} has an extra~$\bar K_2^4(r)$ factor compared to that in~\eqref{eq:n-mine-complex} if we do not distinguish between~$\bar K_2(r)$ and~$K_2 = \bar K_2(0)$, and similarly when comparing~\eqref{eq:n-fake-improved} and~\eqref{eq:n-fake-complex}. 
This can be a substantial difference since~$K_2$ and~$\bar K_2(r)$ can both depend on the norm of~$\theta_*$.
In fact, in Appendix~\ref{sec:subgaussian-check} (Proposition~\ref{prop:logistic-case-study}) we show, by a technical calculation, that in logistic regression with~$X \sim \cN(0,\Cov)$ one has
\[
\rho \lsim (1+\|\theta_*\|_\Cov)^{3},
\]
this bound being tight, while the bound on~$\bar K_2(1/\sqrt{\rho})$ is
\[
\bar K_2(1/\sqrt{\rho}) \lsim 
\sqrt{1+\|\theta_*\|_\Cov}
\]
up to a logarithmic factor.
Thus,~$\bar K_2^4(1/\sqrt{\rho})$ can potentially be as large as~$\rho^{2/3}$.
On the other hand, when the distribution of~$\wt X(\theta)$ is \textit{log-concave} and centrally symmetric at any~$\theta \in \Theta_r(\theta_*)$, the factor~$\bar K_2^4(r)$ can be eliminated. 
This amounts to using the improved relation between the third and second moments of the marginals of~$\He(\theta)^{-1/2} \wt X(\theta)$ in step~$\boldsymbol{1^o}$ of the analysis in Theorems~\ref{th:crb-mine-improved}--\ref{th:crb-fake-improved}:
\[
\E [|\lang \He(\theta)^{-1/2} \wt X(\theta),u\rang|^3| \le 7(\E[\lang u, \He(\theta)^{-1/2} \wt X(\theta),u\rang^2])^{3/2},
\]
as follows from~\cite[Lem.~2]{bubeck2014entropic} by simple algebra, using log-concavity of~$\He(\theta)^{-\frac{1}{2}} \wt X(\theta)$.


\paragraph{Extending Theorems~\ref{th:crb-mine-improved}--\ref{th:crb-fake-improved} to heavy-tailed distributions.}
One fact playing the key role in the proofs of the last two theorems is that in the bound
\begin{equation}
\label{eq:confidence-additive}
K^2(d + \log(e/\delta))
\end{equation}
for the sample complexity of estimating a single covariance matrix, the confidence term~$\log(1/\delta)$ is \emph{additive with~$d$.} This allows to take the union bound over an exponential in~$d$ number of events correponding to the centers of the epsilon net, while still preserving a near-linear in~$d$ sample complexity. 

As discussed in Section~\ref{sec:basic}, the main technical challenge when trying to extend our results to heavy-tailed distributions is posed by Assumption~\ref{ass:second-subg}, which for Theorems~\ref{th:crb-mine-improved}--\ref{th:crb-fake-improved} gets strengthened to~Assumption~\ref{ass:second-subg++}. To get rid of it, one could replace the empirical Hessians~$\wh\He(\theta)$ by some estimator~$\bar \He(\theta)$ that estimates~$\He(\theta)$ with high confidence in the positive-semidefinite sense (cf.~\eqref{eq:psd-bounds}) under weak moment assumptions. Given such estimators, we can essentially repeat the covering argument in the analysis of Theorems~\ref{th:crb-mine-improved}--\ref{th:crb-fake-improved}, replacing the Hessian estimate in any~$\theta \in \Theta_{r}(\theta_*)$ (with~$r = 1$ or~$r = 1/\sqrt{\rho}$) with the estimate~$\bar \He(\theta')$ in the closest center~$\theta'$ of the cover, and replacing empirical risk minimization with a version of stochastic quasi-Newton algorithm with~$\bar \He(\theta')$ as the Hessian oracle for~$\He(\theta)$. Unfortunately, the only known to us estimator that provably satisfies a high-confidence affine-invariant bound under weak moment assumptions is the one from~\cite{heavy-covariance}, and its sample complexity scales as
\[
K^2 d \log(e/\delta), 
\]
i.e., the confidence term enters {\em multiplicatively} with~$d$.
After taking the union bound over~$d^{O(d)}$ events, this bound becomes \textit{quadratic} in~$d$.
While this is sufficient to extend Theorems~\ref{th:crb-fake-complex}--\ref{th:crb-mine-complex}, the argument in Theorems~\ref{th:crb-mine-improved}--\ref{th:crb-fake-improved} is destroyed. 
Thus, extending the latter theorems, and obtaining near-linear sample complexity, has to rely on $\prccq$-type covariance estimation with additive confidence, cf.~\eqref{eq:confidence-additive}. The closest in this direction is the recent work~\cite{mendelson2018robust} which establishes a high-probability bound in the operator norm,
$
\|\wh \Cov - \Cov\| \le c\|\Cov\|,
$
holding with probability~$\ge 1-\delta$ when~$n \ge C({\kappa})  [\boldsymbol{r}(\Cov) + \log(1/\delta)]$, where~$C(\kappa)$ is a constant depending only on the kurtosis, and~$\boldsymbol{r}(\Cov) := \tr(\Cov)/\|\Cov\| \le d$ is the effective rank. Unfortunately, it is challenging to apply this result in our context, since the operator-norm bounds cannot be translated to~$\prccq$-type guarantees akin to~\eqref{eq:psd-bounds} when the estimator is not affine-equivariant. Some progress in this direction has recently been obtained in~\cite{heavy-covariance}; see~\cite[Sec.~2.3]{heavy-covariance} for a detailed discussion.


\section{High-dimensional setup}
\label{sec:sparsity}

Our next goal is to extend the results obtained so far to the high-dimensional setting. 
Namely, we assume that~$\Theta = \R^d$ with~$d \gg n$, and that the optimal parameter~$\theta_*$ is \textit{sparse}, i.e., the number of non-zero components of~$\theta_*$ is at most~$\s \ll d$. 
Note that if the support~$\cS$ of~$\theta_*$ was known, a reasonable estimator could be obtained by replacing~$X$ with its projection~$X_\cS$ on~$\cS$, and minimizing the empirical risk on~$\cS$.
As in the case of quadratic loss, and the classical Lasso estimator, this would lead to the improvement over the results of Section~\ref{sec:basic}--\ref{sec:improved}: the ambient dimension~$d$ would be replaced with~$\s$, and~$\deff$ with the quantity~$\tr(\He_\cS^{-1} \Gr^{\vphantom{-1}}_{\cS})$ where~$\Gr_\cS = \E[\ell'(Y, X_{\cS}^\T \theta_*)X_\cS^{\vphantom{\T}} X_\cS^\T]$ and~$\He_\cS = \E[\ell''(Y, X_{\cS}^\T \theta_*)X_{\cS}^{\vphantom{\T}} X_{\cS}^\T]$. 
However, in reality~$\cS$ is unknown, and the common recommendation is to use the~$\ell_1$-penalized~$M$-estimator, given by
\begin{equation}
\label{def:lasso}
\wh\theta_{\lambda,n} \in \Argmin_{\theta \in \R^d} L_n(\theta) + \lambda \|\theta\|_1.
\end{equation}
In the case of quadratic loss, it is well-known that the risk of the $\ell_1$-penalized estimator, when measured in terms of the~$\ell_1$-loss or the ``prediction'' loss corresponding to the design covariance matrix, is within a logarithmic in $d$ factor from the ``ideal'' risk of the projection oracle, provided that the penalization parameter~$\lambda$ is appropriately chosen, and the design is near-isotropic and subgaussian -- see, e.g.,~\cite{tibshirani},~\cite{candes2007dantzig},~\cite{bickel2009simultaneous},~\cite{juditsky2011verifiable}.
While the statistical theory for the quadratic loss is almost complete, this is not yet the case for general $M$-estimators. 
Here our goal is to partially close this gap, providing analogues of Theorems~\ref{th:crb-fake-complex} and~\ref{th:crb-mine-complex} in the high-dimensional setting. 
These results extend those obtained in~\cite{bach2010self-concordant} for the logistic loss using pseudo self-concordance, and are close to those proved in~\cite{van2012quasi}; we discuss the connections with these works in the end of this section.
Finally, notice that we do not prove analogues of Theorems~\ref{th:crb-mine-improved}--\ref{th:crb-fake-improved}, which would have resulted in a near-linear, rather than quadratic, dependency of the critical sample size from~$\s$.
We leave such extensions for future work.

We now introduce the final assumption complimentary to Assumption~\ref{ass:curvature}.
\begin{customass}{C$^*$}
\label{ass:lipshitz}
One has~$\Cov = \Id$. Moreover, for some~$\k_1,\k_2 > 0$ it holds
\[
\Gr \prccq \k_1 \Id, \quad \He \prccq \k_2 \Id.
\]
\end{customass}

Together, Assumptions~\ref{ass:curvature} and~\ref{ass:lipshitz} imply the bounds in operator norm:
\[
\|\Gr\|_{\infty} \le \k_1, \quad \|\He\|_{\infty} \le \k_2, \quad \|\He^{-1}\|_{\infty} \ge {1}/{\rho}.
\]
Moreover, we can reasonably expect that in the ill-specified case,~$\Gr \succq \He$, which is a stronger version of the natural inequality~$\deff \ge d$.
When this is the case, the eigenvalues of both~$\He$ and~$\Gr$ belong to the interval~$[\rho^{-1}, \overline\k]$ where~$\overline\k := \max(\k_1,\k_2)$. 
Then, the product
\[
Q := \rho\overline\k
\]
can be considered as the condition number of the estimation problem at hand.
In particular, we are about to see that the excess risk bounds, as well the bounds for the critical sample size, get inflated by~$Q$ in the high-dimensional regime.
This reflects the requirement that the problem should be well-conditioned with respect to the \textit{standard} coordinate basis, since both $\ell_0$-``norm'' and~$\ell_1$-norm depend on the choice of the basis. 
Some further remarks are given below.
\begin{itemize}
\item
Similarly to the bound~\eqref{eq:rho-bound}, we can always bound~$\k_1$ and~$\k_2$:
\[
\k_1 \le \sup_{(y,\eta) \in \cY \times \R} |\ell'(y,\eta)|, \quad \k_2 \le \sup_{(y,\eta) \in \cY \times \R} \ell''(y,\eta).
\]
Arguably, these bounds are more informative than the bound~\eqref{eq:rho-bound} for~$\rho$, as they involve the suprema of the loss derivatives (e.g., the right-hand sides are constants for pseudo-Huber and logistic losses).
\item
Correlated designs can also be considered, but this would lead to the inflation of the bounds by the condition number of~$\Cov$. 
This is natural, as~$\ell_1$-regularization fixes the basis, and the estimator is not affine-invariant.
\end{itemize}

The next result characterizes the statistical properties of the $\ell_1$-penalized $M$-estimator~\eqref{def:lasso} with a canonically self-concordant loss, extending Theorem~\ref{th:crb-fake-complex}.
\begin{theorem}
\label{th:crb-sparse-fake}
Assume~\ref{ass:gsc-fake},~\ref{ass:design-subg},~\ref{ass:first-subg},~\ref{ass:second-subg},~\ref{ass:curvature}, \ref{ass:lipshitz}, and~$|\theta_*|_{0} \le \s$. 
\begin{enumerate}
\item
Whenever
\begin{equation}
\label{eq:sparse-condtion-hessian}
n \gsim \max \left\{ \rho \k_2 K_2^4 \s \log\left({ed}/{\delta}\right), \; \rho^2 \k_1 K_0^2 K_1^2 \s^2 \log\left({edn}/{\delta}\right) \right\},
\end{equation}
and the regularization parameter satisfies
\begin{equation}
\label{eq:sparse-condition-lambda}
K_1\sqrt{\frac{\k_1 \log (ed/\delta)}{n}} \lsim \lambda \lsim \frac{1}{\rho K_0 \s \sqrt{\log(e d n/\delta)}},
\end{equation}
we have that with probability at least~$1-\delta$,
\begin{equation}
\label{eq:sparse-errors}
\begin{aligned}
\|\wh\theta_{\lambda,n} - \theta_*\|_1 \lsim \rho \s \lambda, \quad \|\wh\theta_{\lambda,n} - \theta_*\|_{\He}^2 \lsim \rho \s \lambda^2.
\end{aligned}
\end{equation}
\item
Define~$\cE := \{\|X\|_{\infty} \lsim K_0 \sqrt{\log \left({ed}/{\delta}\right)}\}$.
Then, $\Prob(\cE) \ge 1-\delta$, and whenever
\begin{equation*}
\delta \lsim \left(\frac{\lambda}{ K_1 \sqrt{\k_1\log(ed)}}\right)^{1 + \frac{1}{\log(d)}},
\end{equation*}
the restricted risk~$L_{\cE}(\theta) := \E[\ell_Z(\theta)\ind_\cE(X)]$ w.p.~at least~$1-\delta$ satisfies
\begin{equation}
\label{eq:sparse-excess-risk}
L_{\cE}(\wh\theta_{\lambda,n}) - L_{\cE}(\theta_*) \lsim \rho \s \lambda^2.
\end{equation}
\end{enumerate}
\end{theorem}

Clearly, the right choice of~$\lambda$ is the one attaining the lower bound in~\eqref{eq:sparse-condition-lambda}:
\[
\lambda \approx K_1\sqrt{\frac{\k_1 \log (ed/\delta)}{n}}
\]
This choice is always possible since the left-hand side in~\eqref{eq:sparse-condition-lambda} is upper-bounded with the right-hand side due to the second bound in~\eqref{eq:sparse-condtion-hessian}.
With such~$\lambda$, both the prediction error and the (restricted) excess risk~$L_{\cE}(\wh\theta_{\lambda,n}) - L_{\cE}(\theta_*)$ are at most
\[
O\left(\frac{Q \s \log (ed/\delta)}{n}\right)
\]
whenever~$n \gsim \max(Q\s,\rho Q\s^2)\log(ed/\delta),$ ignoring the dependence on the subgaussian constants.
Thus, in the case of pseudo self-concordant losses,~$d$ and~$\deff$ both get replaced with~$\s$, at the expense of extra~$O(Q\log d)$ factor in the bounds.

Next we state a version of Theorem~\ref{th:crb-sparse-fake} for canonically self-concordant losses.
\begin{theorem}
\label{th:crb-sparse-true}
Assume~\ref{ass:gsc-true},~\ref{ass:first-subg},~\ref{ass:second-subg},~\ref{ass:curvature}, \ref{ass:lipshitz}, and~$|\theta_*|_{0} \le \s$. 
\begin{enumerate}
\item
Whenever
\begin{equation}
\label{eq:sparse-condtion-hessian-true}
n \gsim \max \left\{ \rho \k_2 K_2^4 \s \log\left({ed}/{\delta}\right), \; \rho^2 \k_1 \k_2 K_1^2 K_2^2 \s^2 \log\left({edn}/{\delta}\right) \right\}
\end{equation}
and the regularization parameter satisfies
\begin{equation}
\label{eq:sparse-condition-lambda-true}
K_1\sqrt{\frac{\k_1 \log (ed/\delta)}{n}} \lsim \lambda \lsim \frac{1}{\rho K_2 \s \sqrt{\k_2 \log(e d n/\delta)}},
\end{equation}
we have that with probability at least~$1-\delta$,
\begin{equation}
\label{eq:sparse-errors-true}
\begin{aligned}
\|\wh\theta_{\lambda,n} - \theta_*\|_1 \lsim \rho \s \lambda, \quad \|\wh\theta_{\lambda,n} - \theta_*\|_{\He}^2 \lsim \rho \s \lambda^2.
\end{aligned}
\end{equation}
\item
The event~$\cE := \{\|\wt X\|_{\infty} \lsim K_2 \sqrt{\k_2 \log \left({ed}/{\delta}\right)}\}$ satisfies $\Prob(\cE) \ge 1-\delta$. 
Moreover, whenever
\begin{equation*}
\delta \lsim \left(\frac{\lambda}{ K_1 \sqrt{\k_1\log(ed)}}\right)^{1 + \frac{1}{\log(d)}},
\end{equation*}
the restricted risk~$L_{\cE}(\theta) := \E[\ell_Z(\theta)\ind_\cE(X)]$ w.p.~at least~$1-\delta$ satisfies
\begin{equation}
\label{eq:sparse-excess-risk-true}
L_{\cE}(\wh\theta_{\lambda,n}) - L_{\cE}(\theta_*) \lsim \rho \s \lambda^2.
\end{equation}
\end{enumerate}
\end{theorem}

\paragraph{Comparison of Theorems~\ref{th:crb-sparse-fake} and~\ref{th:crb-sparse-true}.}
The usual gain of~$\rho$ that we have observed so far for canonically viz.~pseudo self-concordant losses is not preserved in~$\ell_1$-regularized estimators. Instead, the second bound in~\eqref{eq:sparse-condtion-hessian} and the upper bound in~\eqref{eq:sparse-condition-lambda} get inflated with~$\k_2$, and the critical sample size, given the ``ideal'' choice of the regularization parameter corresponding to the lower bound in~\eqref{eq:sparse-condition-lambda-true}, becomes~$n \gsim \max(Q\s,Q^2\s^2)\log(ed/\delta).$
Essentially, the reason for that is that $\ell_1$-regularization does not ``know'' anything about the matrices~$\He$ and~$\He_n$, and, in a sense, violates the affine-invariant structure of the proofs for non-regularized~$M$-estimators. This seems to be a fundamental problem with~$\ell_1$-regularization, rather than the artifacts of our proofs, since~$\ell_1$-regularized $M$-estimators are \textit{themselves} not affine-invariant.
As such, we believe the additional factors of~$Q$ and~$Q^2$ to be unimprovable in the high-dimensional setup without further assumptions.

\paragraph{Comparison with prior work.}
Theorem~\ref{th:crb-sparse-fake} extends the result of~\cite[Theorem~5]{bach2010self-concordant} for logistic regression with fixed design, obtained using the pseudo self-concordance of the logistic loss.
While the established error bounds are similar, our results have important novelties. First, we analyze the random-design setting, whereas~\cite{bach2010self-concordant} assumes fixed design. 
Second, the result of~\cite{bach2010self-concordant} requires larger sample size, scaling with the product of~$\s$ and~$R^2$ where~$R$ is an upper bound on~$\|X\|_2$. Typically,~$R$ scales as~$\Omega(\sqrt{d})$ (e.g., this is the case where the design is pre-generated by sampling from a subgaussian distribution), thus~\cite{bach2010self-concordant} essentially proves the bound~$O(\s d)$ for the critical sample size.

On the other hand, our results can be compared to those in~\cite{van2012quasi} who establish the rate~$O(\lambda\s)$ for the~$\ell_1$-error and~$O(\lambda^2 \s)$ for the prediction error (see their Theorems~5.2 and~7.3), addressing a larger class of models including GLMs with non-canonical link functions, and general convex robust losses. However, in order to control the precision of the local quadratic approximations of the risk, the authors of~\cite{van2012quasi} assume that~$\ell''(Y,X^\T\theta_*)$ is bounded from below (Conditions~A4 and B), which can only be guaranteed by assuming that~$\theta_*$ is bounded in~$\ell_1$-norm. Thus, their results do not address the case of unbounded parameter.
Remarkably, these results similarly require the sample size to scale as~$\Omega(\s^2 \log d)$.

\begin{remark}
\label{rem:sketching}
In the proofs of Theorems~\ref{th:crb-sparse-fake}--\ref{th:crb-sparse-true}, matrices~$\He$ and~$\He_n$ only interact with residual~$\Delta$ which with high probability satisfies the restricted subspace condition~\eqref{eq:compatibility}.
Hence, we can strengthen the result, replacing Assumption~\ref{ass:curvature} and the inequality~$\He \prccq \k_2 \Id$ in Assumption~\ref{ass:lipshitz} with the requirement that
\[
 {\|\Delta\|_2^2}/\rho \le \|\Delta\|_{\He}^2 \le \k_2 \|\Delta\|_2^2
\]
in the case where~$\Delta \in \R^d$ is \emph{approximately sparse}, i.e., satisfies~$\|\Delta - [\Delta]_{\s}\|_1 \le 3\|[\Delta]_{\s}\|_1,$ where~$[\Delta]_{\s}$ is the projection of $\Delta$ to the span of its $\s$ largest coordinates.
This observation can be exploited to accelerate computation of the estimator~\eqref{def:lasso} when using proximal Newton-type methods (see~\cite{lee2014proximal}) via \emph{Hessian sketching}, i.e., by replacing the estimates~$\He_n(\theta)$ with the estimates~$\He_m(\theta) := \frac{1}{m} \sum_{j = 1}^m \wt X_j(\theta) \wt X_j(\theta)^\T$ computed from a small subsample.
\end{remark}

We defer further discussion of related work on~$\ell_1$-regularized~$M$-estimators to Section~\ref{sec:related-work}. 

\section{Numerical experiments}
\label{sec:numerical}
We now present two numerical experiments that illustrate our theoretical results.\footnote{All our codes are available online at \url{http://github.com/ostrodmit/self-concordant}.} 
\paragraph{Critical sample size grows linearly with model dimensionality.}
Here the point is to illustrate the results in Section~\ref{sec:improved}, namely Theorems~\ref{th:crb-mine-improved}--\ref{th:crb-fake-improved}. 
Recall that, in a nutshell, these results state that the fast~$O(d/n)$ rate for the excess risk becomes available starting from the critical sample size which is~$O(\deff \vee d)$, where~$O(\cdot)$ hides factors depending on the distribution-dependent constants~$K_0, K_1, \bar K_2$,~$\rho$ arising in~Assumptions~\ref{ass:design-subg},~\ref{ass:first-subg},~\ref{ass:second-subg++}, and~\ref{ass:curvature}.
In our first experiment (see Fig.~\ref{fig:exp-gauss-logistic}), we empirically demonstrate that the critical sample size indeed scales linearly with the parameter dimension. 
For growing sample size~$n = 10^k$,~$k \in [1,3]$, we generate an i.i.d.~sample~$(X_i, Y_i)_{i=1}^n$ with standard Gaussian design~$X_i \sim \cN(0,\Id_d)$ and conditional distribution of the (binary) label given by~$\Prob[Y_i = 1] = 1/(1+\exp(-X_i^\top\theta_*))$ (i.e., such that the logistic model is well-specified) or by~$\Prob[Y_i = 1] = 1-\phi(X^\top \theta^*)$, where~$\phi(\cdot)$ is the standard Gaussian c.d.f., which corresponds to the probit regression.
Thus, the logistic model for~$Y | X$ is well-specified in the second case
We take~$\theta_* = \ind_d/\sqrt{d}$ (thus~$\| \theta_* \|_2 = 1$) and consider the following three quantities for~$d \in \{8,16,32,64\}$:
\begin{enumerate}
\item Excess risk~$L(\wh \theta_n) - L(\theta_*)$ of the logistic regression estimator, i.e., for the~$M$-estimator with the logistic loss~$\ell(y,\eta) =  \log(1+e^{\eta}) - y\eta$. 
\item Excess risk~$L^{\SC}(\wh \theta_n^{\SC}) - L^{\SC}(\theta_*^{\SC})$ for the~$M$-estimator with loss~\eqref{def:our-class-loss} -- canonically self-concordant analogue of the logistic loss proposed in Sec.~\ref{sec:ass-qsc}. 
Here~$L^{\SC}(\theta) := \E[\ell^{\SC}(Y,X^\top \theta)]$ with~$\ell^{\SC}(y,\eta)$ given by~\eqref{def:our-class-loss};~$\theta_*^{\SC}$ minimizes~$L^{\SC}(\theta)$ and might be different from~$\theta_*$. Note that~$\ell^{\SC}(y,\cdot)$ and~$\ell(y,\cdot)$ have the same second-order Taylor expansion around~$\eta = 0$ (see~Fig.\ref{fig:robust}).
\item Excess risk~$L(\wh \theta_n^{\SC}) - L(\theta_*)$ that evaluates~$\wh \theta_n^{\SC}$ as a surrogate estimator.
\end{enumerate}
In all three cases, we approximate the excess risk via a test sample with~$N = 10^4$ observations, and we compute~$\theta_*^\SC$ by running~$\textsf{fmincon}$ optimization routine in Matlab (we use the constraint~$\|\theta\|_2 \le 2$ to avoid numerical instabilities). 
Then, for each value of~$d$ and the three notion of excess risk, we plot the excess risk against the sample size in the~$\log_{10}-\log_{10}$ scale.
The experiment is repeated~$T = 800$ times, and the averaged curve is then plotted along with a~$3\sigma$-confidence interval. 

The results are shown in Fig.~\ref{fig:exp-gauss-logistic}. We can distinctively see the elbow effect: the initial slow convergence rate (slope around~$-1/2$ on the log-log scale) changes to the fast rate (slope~$-1$) for larger sample size. 
This is observed for all three curves, all values of~$d$, and both conditional distributions of~$Y$.
\begin{itemize}
\item
For the logistic distribution,~$\wh \theta_n$ outperforms~$\wh\theta_n^{\SC}$ in the fast rate zone (i.e., with sample sizes above the critical level) in terms of their corresponding ``native'' risks as well as the logistic risk. 
This is expected: while~$\wh\theta_n$ is well-specified, estimator~$\wh\theta_n^{\SC}$ has to pay for model misspecification, and its excess risk depends on~$\deff$ rather than~$d$ (cf.~\eqref{eq:excess-risk-bound-neat} in Theorem~\eqref{th:crb-mine-improved}). 
Meanwhile, for smaller sample sizes~$L^{\SC}(\wh \theta_n^{\SC}) - L^{\SC}(\theta_*^{\SC})$ is smaller than the other two excess risks. 
This seems to be simply due to~$\ell^{\SC}(y,\eta)$ being smaller than~$\ell(y,\eta)$ away from~$\eta = 0$ (cf.~Fig.~\ref{fig:robust}).
\item
In the case of probit distribution, both estimators are misspecified, and turn out to have very close performance in terms of all three excess risks. 
\end{itemize}
Finally, and most importantly, we see that the ``elbow'' on the curves moves to the right {\em in (roughly) constant increments} as we increase~$d$ geometrically. 
This is what we expect: according to Theorems~\ref{th:crb-fake-improved}--\ref{th:crb-mine-improved}, the critical sample size grows linearly with~$d$ or~$\deff$ in the misspecified case 
(and here~$\deff$ is itself linear in~$d$).


\paragraph{Critical sample size growing as~$e^{RD}$ for ``bad'' design distributions.}
Here we empirically investigate the dependency of constants~$K_0, K_1, \bar K_2, \rho$ from the norm~$D = \|\theta_*\|_{\Cov}$ of the population risk minimizer.
Recall that in Appendix~\ref{sec:subgaussian-check} we provide polynomial bounds in the case of logistic regression with Gaussian design. 
However, for certain (quite artificial) distributions of design the dependency might be exponential as implied by the results of~\cite{hazan2014logistic}. 
In this experiment, we consider the adversarial distribution proposed in~\cite[Sec.~3.2]{hazan2014logistic}, in which~$X \in \R^{2}$ is supported on three points with carefully chosen probabilities (see~\cite[Figs.~3-4]{hazan2014logistic}) and~$Y \equiv 1$. 
The authors prove the~$\Omega(1/\sqrt{n})$ lower bound (and hence the absence of fast rate) for the excess risk long as~$n \lsim e^{RD}$, where~$R = \|X\|_{\Cov^{-1}}$. 
We empirically discover a similar phenomenon for the self-concordant loss~\eqref{def:our-class-loss}. 
To this end, we follow a similar protocol as in the previous experiment but generate the pairs~$(X_i, Y_i)$ according to the distribution in~\cite{hazan2014logistic} and linearly increase~$D$ while fixing~$d = 2$. 
The experiment is repeated~$T = 1600$ times for sample sizes~$n \in [10^1, 10^4]$, and the population risk is approximated via a test sample with size~$N = 5 \cdot 10^4$. 
We then plot the same three dependencies as in the previous experiment (again in the log-log scale) for~$D \in \{1,3,5,7\}$. 

The results are presented in Fig.~\ref{fig:exp-hazan-always-1}. 
For small sample sizes the curves oscillate, which seems to be due to the special low-dimensional structure of the design distribution. 
However, the upper envelope of the curve clearly exhibits the same ``elbow'' effect as before: the slope changes from roughly~$-1/2$ to~$-1$ for large sample sizes. 
Moreover, the horizontal location of the elbow moves in nearly uniform increments as we change~$D$ linearly, precisely as expected from the theory in~\cite{hazan2014logistic}. 
We also note that the ``transfer'' risk~$L(\wh \theta_n^{\SC}) - L(\theta_*)$ converges to a non-zero value, which shows that~$\theta_*^{\SC} \ne \theta_*$ for the distribution considered here.



\begin{figure}

\begin{minipage}{0.48\textwidth}
\centering
\quad\quad\quad Logistic distribution of~$Y|\eta$\\
\quad\quad\quad$d = 8$
\includegraphics[width=0.8\textwidth]{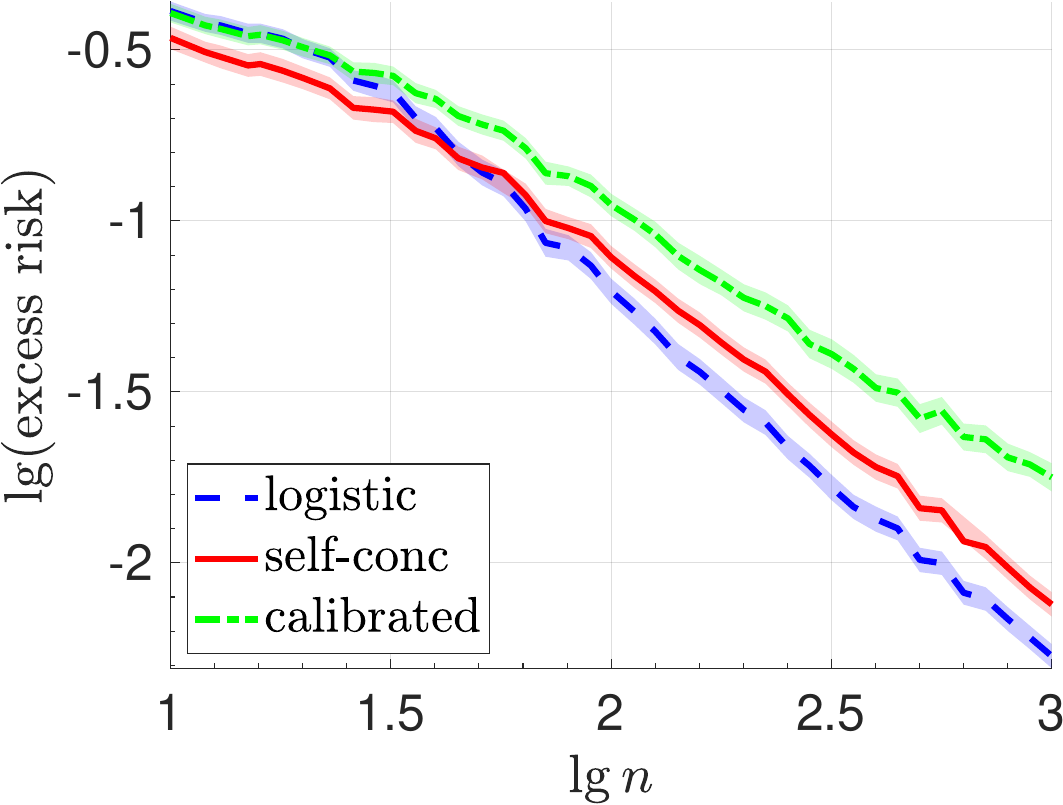}
\vspace{0.2cm}
\centering
\quad\quad\quad$d = 16$
\includegraphics[width=0.8\textwidth]{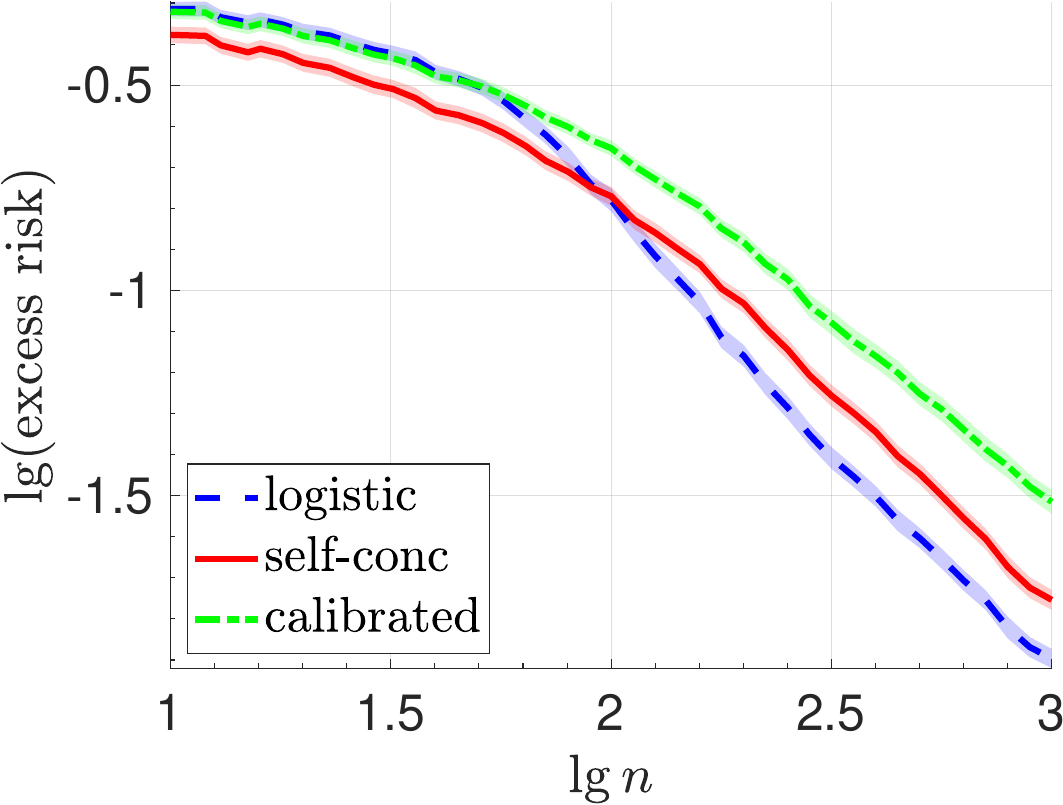}
\vspace{0.2cm}
\centering
\quad\quad\quad$d = 32$
\includegraphics[width=0.8\textwidth]{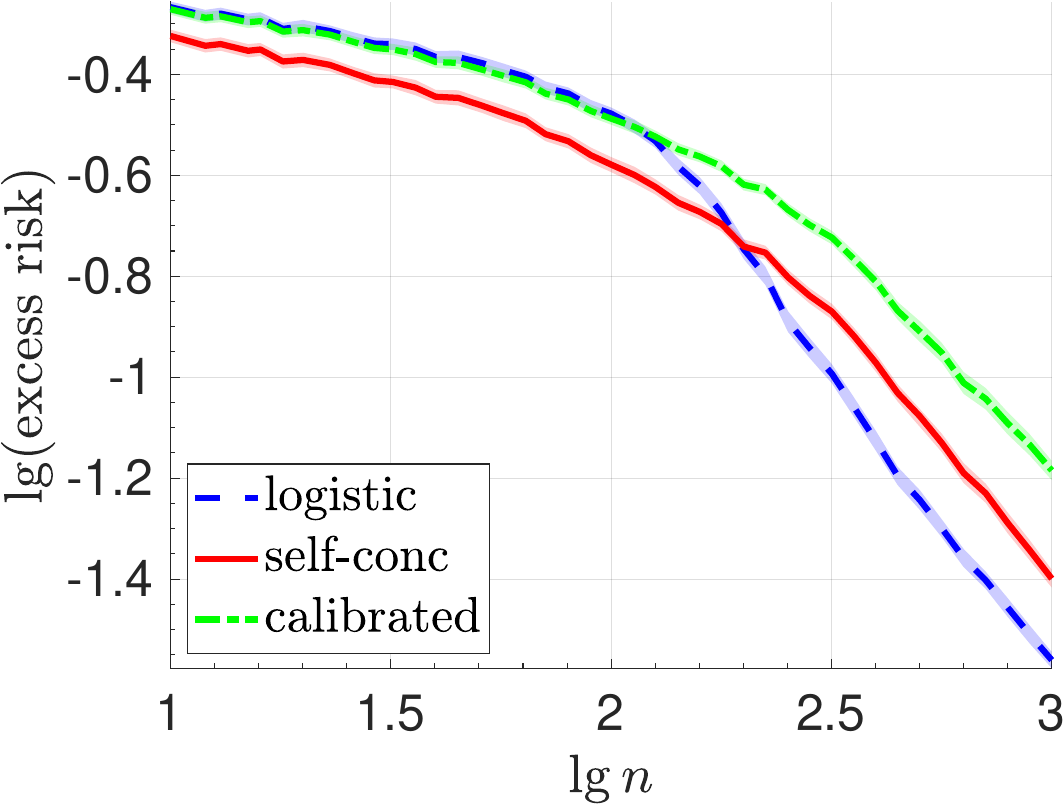}
\vspace{0.2cm}
\centering
\quad\quad\quad$d = 64$
\includegraphics[width=0.8\textwidth]{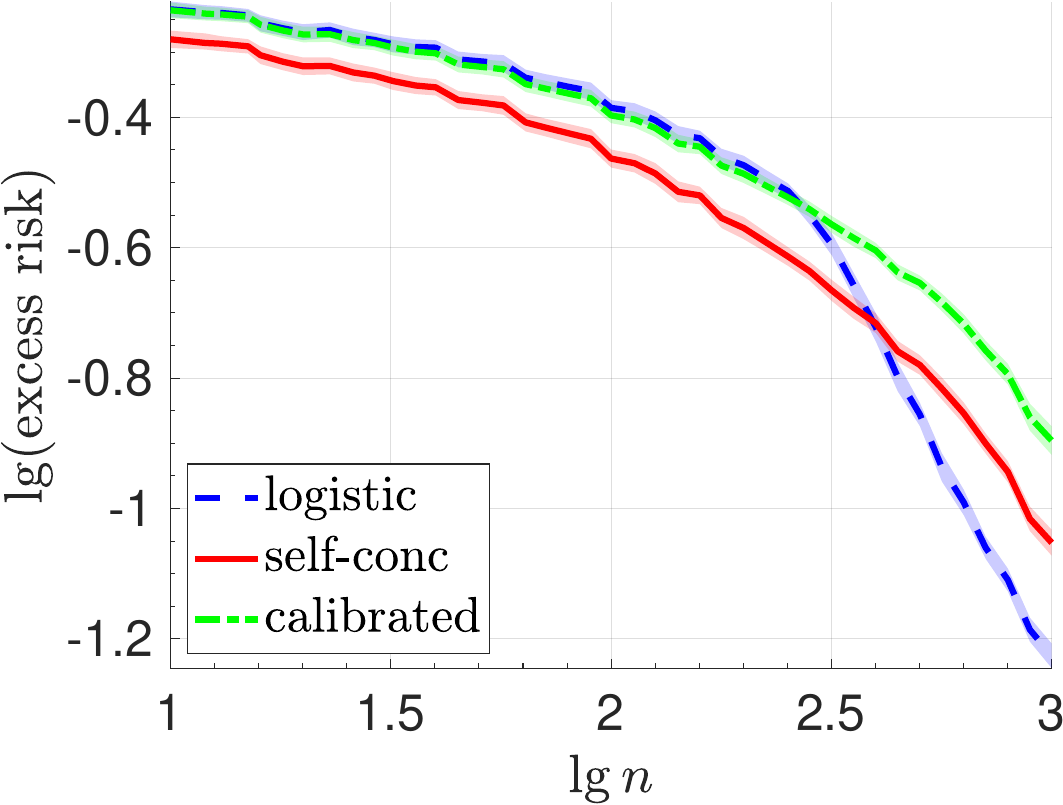}
\end{minipage}
\begin{minipage}{0.48\textwidth}
\centering
\quad\quad\quad Probit distribution of~$Y|\eta$\\
\quad\quad\quad$d = 8$
\includegraphics[width=0.8\textwidth]{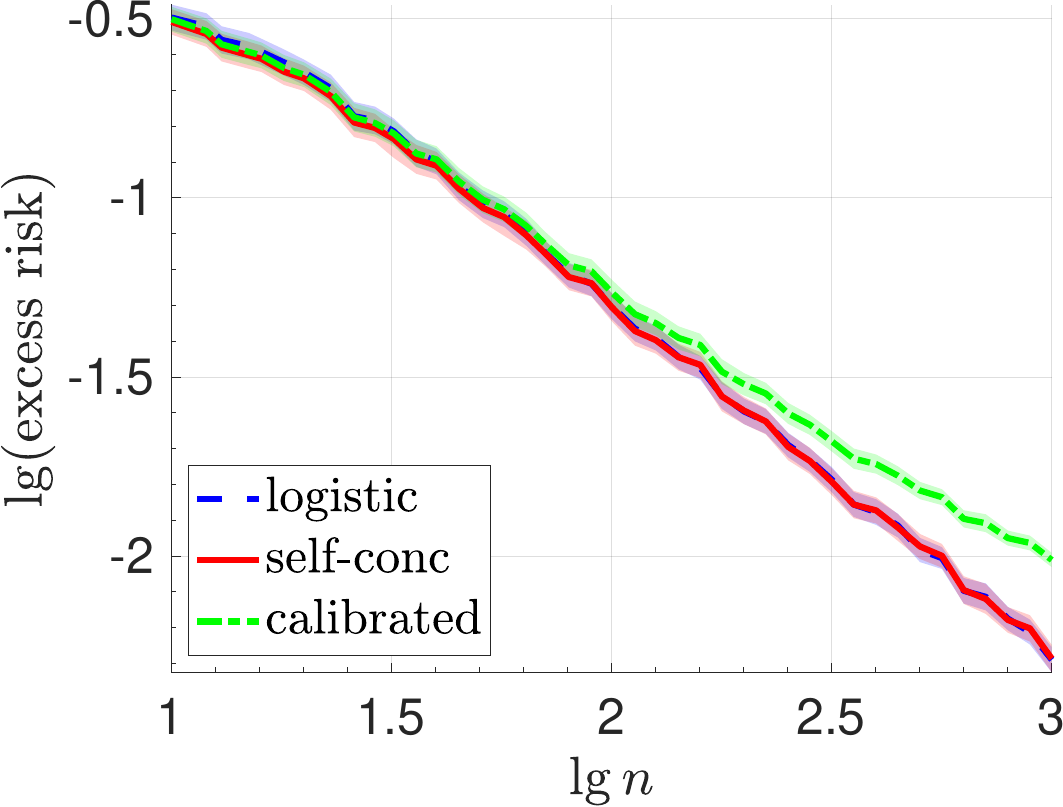}
\vspace{0.2cm}
\centering
\quad\quad\quad$d = 16$
\includegraphics[width=0.8\textwidth]{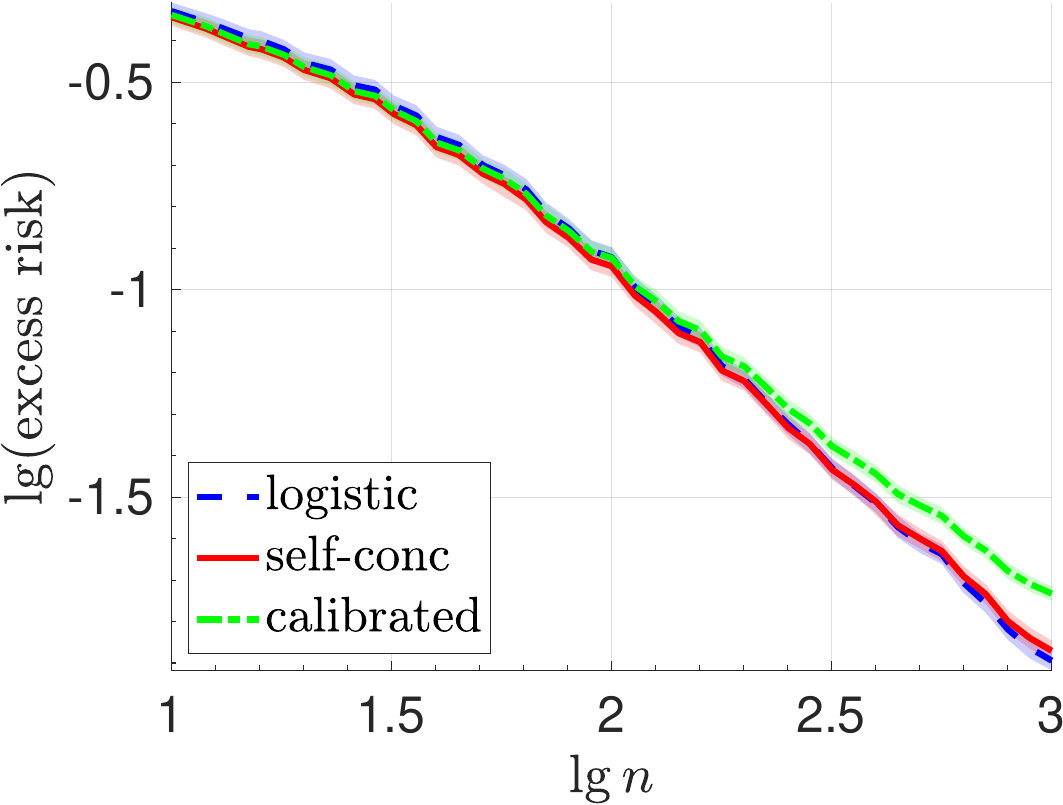}
\vspace{0.2cm}
\centering
\quad\quad\quad$d = 32$
\includegraphics[width=0.8\textwidth]{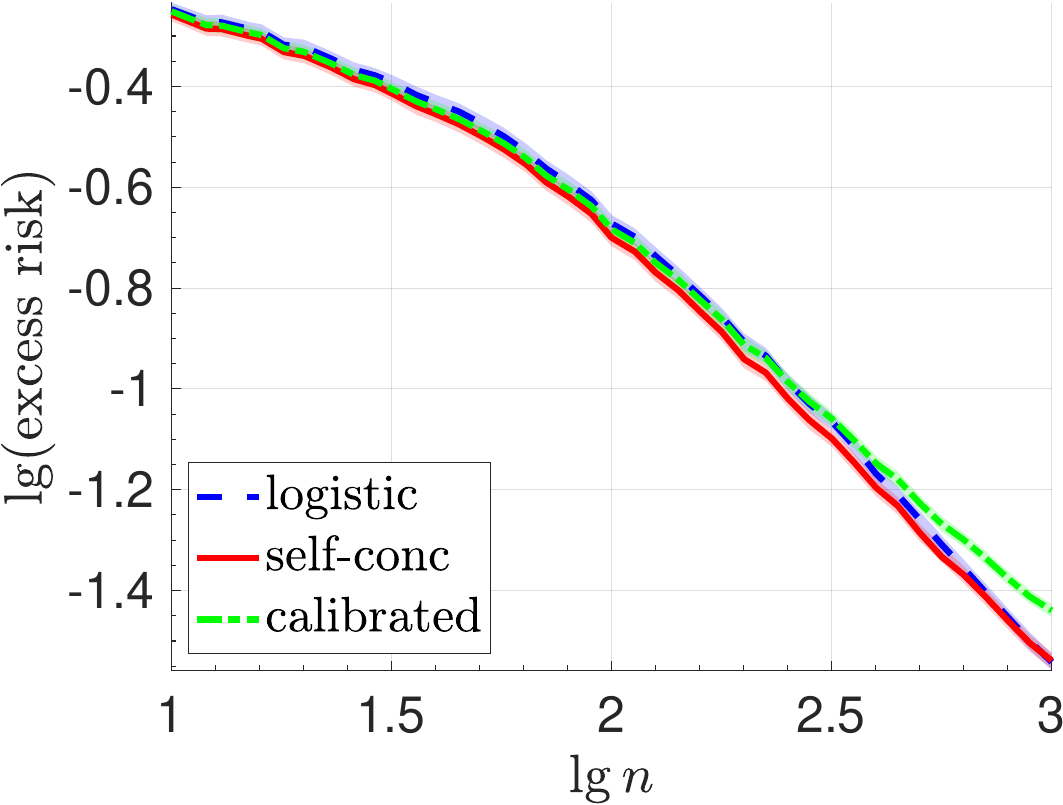}
\vspace{0.2cm}
\centering
\quad\quad\quad$d = 64$
\includegraphics[width=0.8\textwidth]{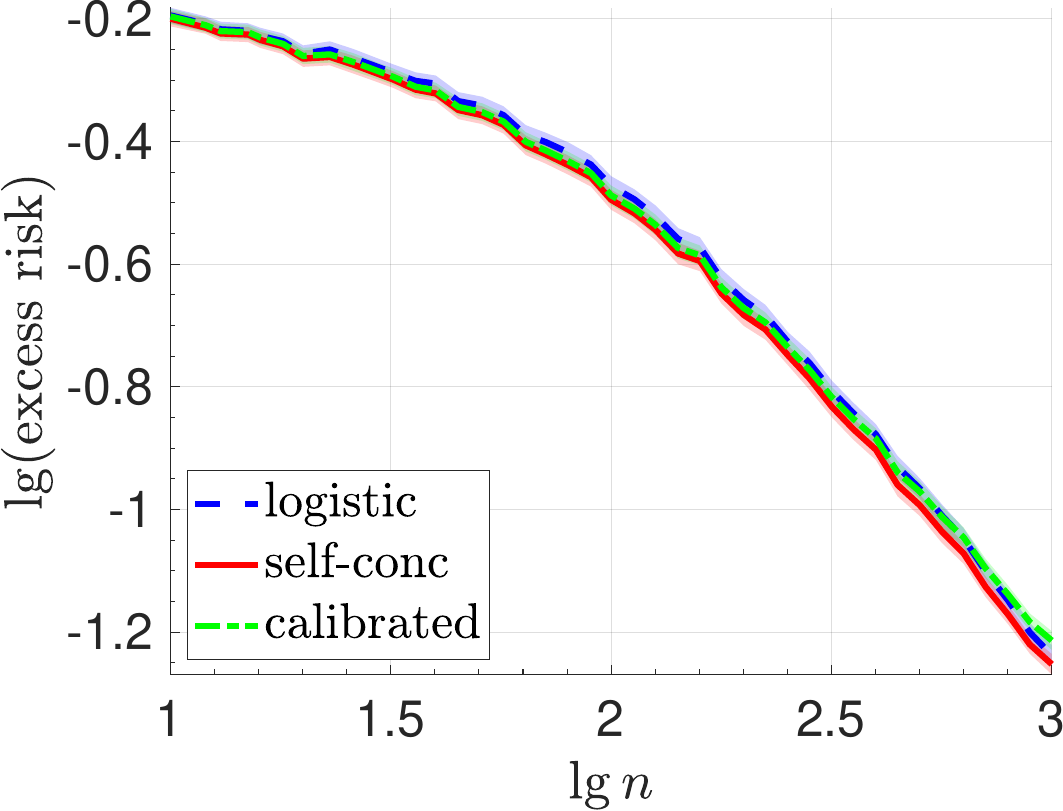}
\end{minipage}

\caption{
The comparison of two~$M$-estimators: with the logistic loss~(estimator~$\wh\theta_n$) and its canonically self-concordant analogue~\eqref{def:our-class-loss} (estimator~$\wh\theta_n^{\SC}$) in the first experiment. {\em``Logistic''},~{\em``self-conc''} and~{\em``calibrated''} correspond to the three notions of excess risk: the ``native'' risks for~$\wh \theta_n,\wh\theta_n^{\SC}$ and the ``transfer'' risk for~$\wh\theta_n^{\SC}$ with logistic loss (see p.~22). 
}
\label{fig:exp-gauss-logistic}
\end{figure}

\begin{figure}[t]
\begin{center}
\begin{minipage}{0.48\textwidth}
\centering
\quad\quad\quad$D = 1$
\includegraphics[width=0.8\textwidth]{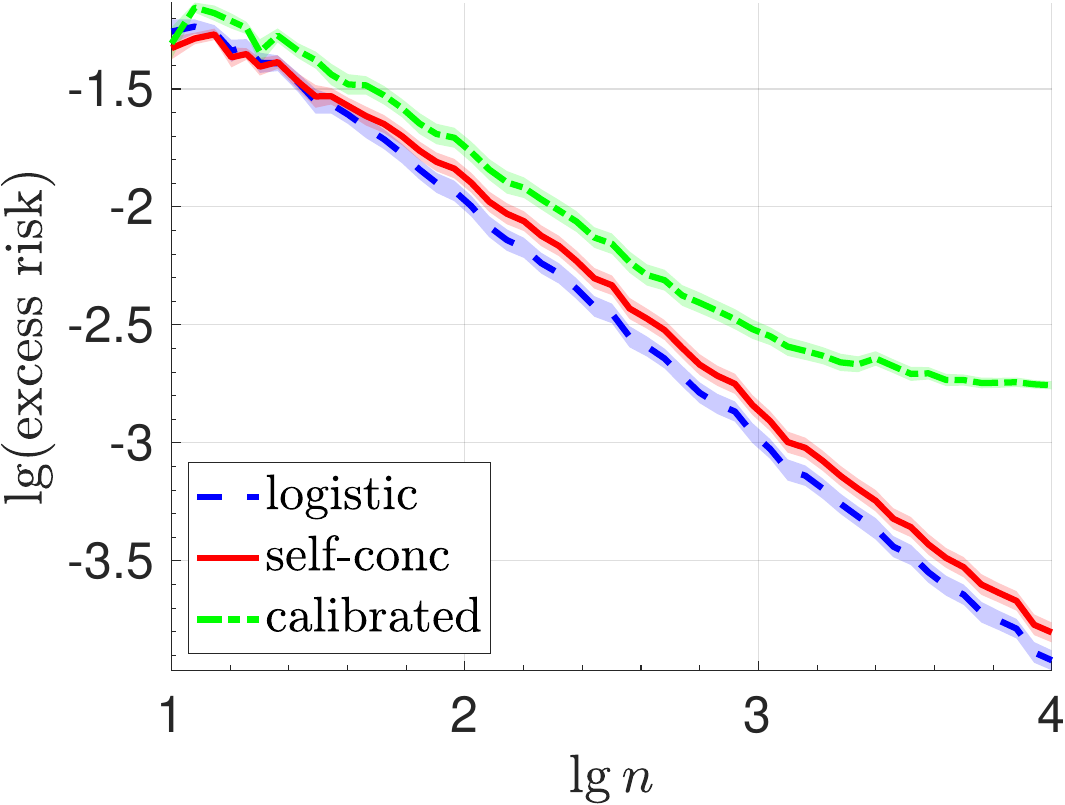}
\end{minipage}
\begin{minipage}{0.48\textwidth}
\centering
\quad\quad\quad$D = 3$
\includegraphics[width=0.8\textwidth]{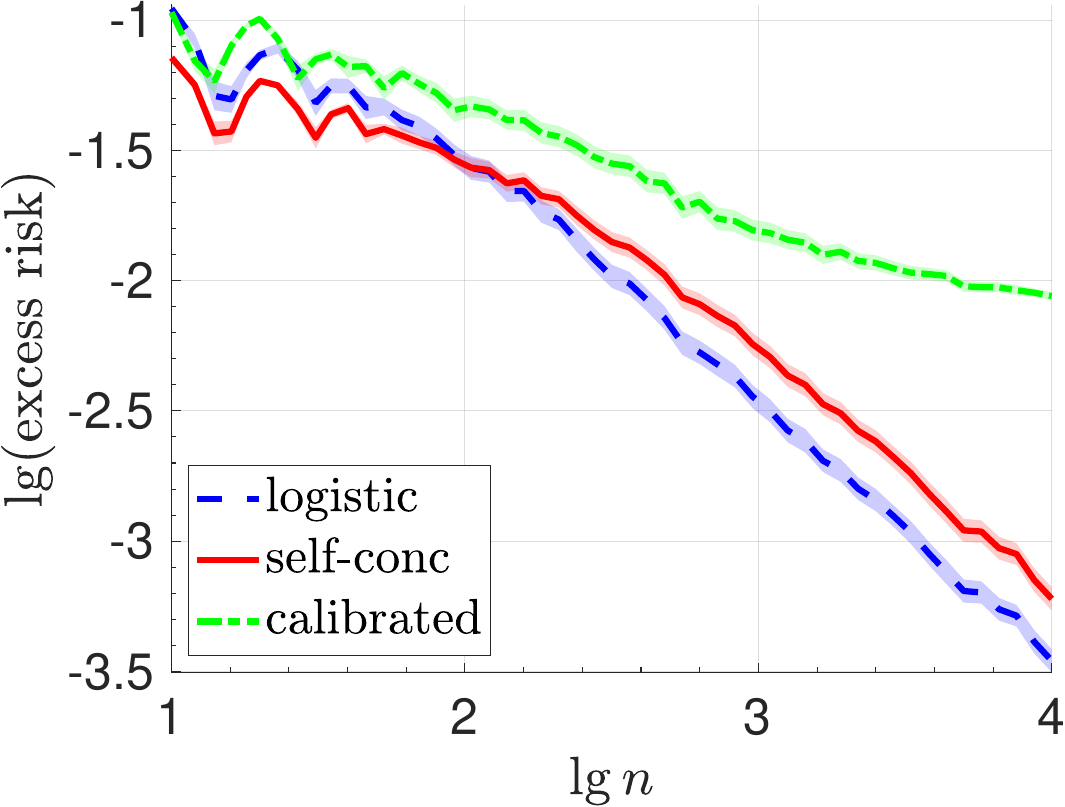}
\end{minipage}\\
\vspace{0.2cm}
\begin{minipage}{0.48\textwidth}
\centering
\quad\quad\quad$D = 5$
\includegraphics[width=0.8\textwidth]{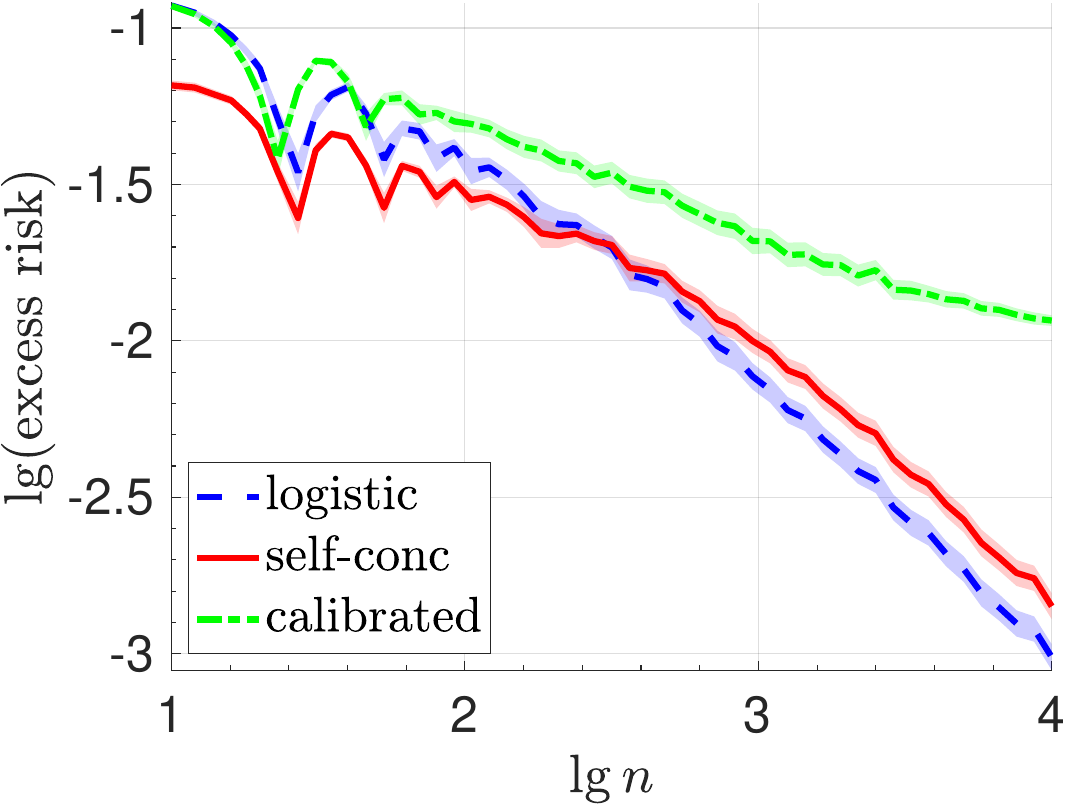}
\end{minipage}
\begin{minipage}{0.48\textwidth}
\centering
\quad\quad\quad$D = 7$
\includegraphics[width=0.8\textwidth]{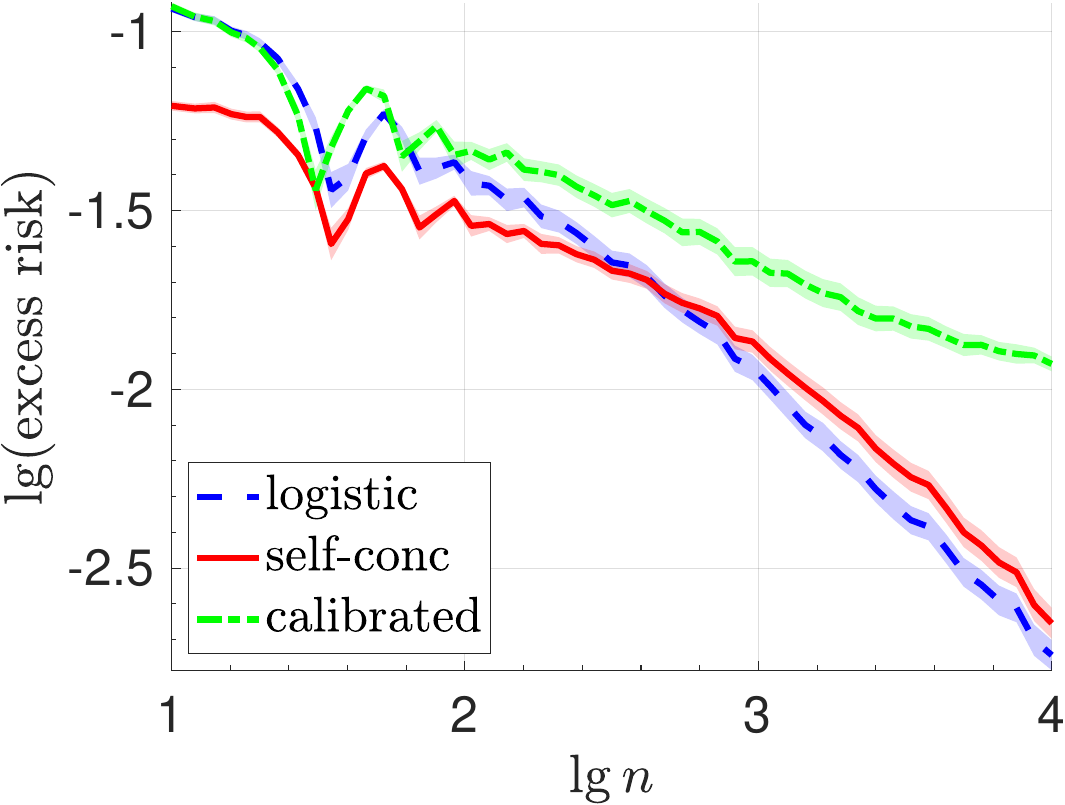}
\end{minipage}
\caption{
The comparison of~$M$-estimators~$\wh\theta_n,\wh\theta_n^{\SC}$ in the second experiment, using the adversarial data distribution from~\cite{hazan2014logistic} (see p.~22 for more details).
{\em``Logistic''},~{\em``self-conc''} and~{\em``calibrated''} correspond to the same three notions of excess risk as in the first experiment (see Fig.~\ref{fig:exp-gauss-logistic}).
}
\label{fig:exp-hazan-always-1}
\end{center}
\end{figure}

\section{Related work}
\label{sec:related-work}

\paragraph{Self-concordant analysis of logistic regression.}
Our approach is inspired by~\cite{bach2010self-concordant}, and we reuse and extend some of their technical results in our Propositions~\ref{prop:local}--\ref{prop:dikin}.
However, our results and analysis are crucially different from those in~\cite{bach2010self-concordant} in several ways.
First, we address the random-design setting, whereas in~\cite{bach2010self-concordant} the design is fixed. 
Second,~\cite{bach2010self-concordant} considers only pseudo self-concordant losses, focusing on logistic regression, whereas we also provide results for canonically self-concordant losses, and, crucially, compare the two cases.
Third, we obtain similar results for ill-specified models, whereas~\cite{bach2010self-concordant} only establishes a slow rate in this case.
Finally, and most importantly, while we use very similar tools to those in~\cite{bach2010self-concordant}, the ``core'' of our analysis is more direct.
Namely,~\cite{bach2010self-concordant} studies the minimizer~$\wh\theta_{\lambda,n}$ of \textit{the~$\ell_2$-penalized empirical risk} with strictly positive regularization parameter~$\lambda$, and moreover, imposes some technical condition on the minimal magnitude of $\lambda$, see their Eq.~(13).
Upon close inspection, this condition implies
\begin{equation}
\label{eq:degrees-of-freedom}
n \gsim \rho \cdot \df_{\lambda}^2, \quad \df_{\lambda} := \tr[\He (\He + \lambda \Id)^{-1}],
\end{equation}
where the \emph{degrees of freedom} parameter~$\df_{\lambda}$ replaces~$d$ in the $\ell_2$-penalized setting.
This, in turn, allows to carry out an argument analogous to ours, but applying Proposition~\ref{prop:dikin} to the \textit{regularized} empirical risk. 
However, $\ell_2$-penalization makes the analysis much more involved, as it rests on the comparison of the regularized risks, and accordingly, relates~$\theta_*$ and~$\wh \theta_{\lambda,n}$ through the intermediate point -- the minimizer~$\theta_\lambda$ of the regularized average risk.
The extra condition in~\cite{bach2010self-concordant}, which makes this analysis possible, is non-trivial, and requires some fine balance between the regularization parameter, sample size, and various types of degrees of freedom and biases. 
We manage to circumvent these difficulties for the plain ERM, including the ill-specified case, by realizing that the only condition needed to carry out the argument based on self-concordance, in the non-regularized case, is the sufficient sample size. 

\paragraph{Self-concordant analysis and improper algorithms.}
Another relevant work is~\cite{bach2014adaptivity} which studies logistic regression with random design, but analyzes an estimate computed by stochastic approximation with averaging.
While this estimator is more advantageous from the computational standpoint, the control of the distance to the optimum is more involved (see~\cite[Proposition~7]{bach2014adaptivity}) which leads to the suboptimal risk bound
\begin{equation}
\label{eq:SGD-rate}
\E_n[L(\wh\theta)] - L(\theta_*) \lsim \frac{R^2(R^4 D_0^4 + 1)}{\mu n},
\end{equation}
where~$\mu$ is the minimal eigenvalue of~$\He$,~$R$ is an upper bound for~$\|X\|_2$ and~$\sup_{\theta \in \Theta}\|\nabla \ell_Z(\theta)\|_2$, and~$D_0 := \|\theta_0 - \theta_*\|_2$ is the initial~$\ell_2$-distance from the optimum (in fact, if~$D_0$ is known up to a constant factor, $R^4 D_0^4$ in~\eqref{eq:SGD-rate} can be replaced with $R^2 D_0^2$). 
The bound~\eqref{eq:SGD-rate} reflects the fact that gradient descent trajectory is not affine-invariant, hence the distances are not ``measured'' in terms of the natural norm~$\|\cdot\|_{\He}$.
For the \textit{natural gradient}, that is, gradient descent on the tranformed problem~$\tilde\theta = \He^{1/2}\theta$, factor~$\mu$ would disappear from~\eqref{eq:SGD-rate}, but~$R$ would be replaced with~$\max(\deff, \rho \cdot d)$, and~$D_0$ with the initial prediction distance~$\|\theta_0 - \theta_*\|_{\He}$, which would lead to a bound scaling as the cube of~$\max(\deff, \rho d)$.
The follow-up work~\cite{bach-moulines} studies a version of the quasi-Newton method, solving the quadratic subproblems via stochastic approximation. 
This allows to conduct affine-invariant analysis of the outer loop, and results in 
\begin{equation}
\E_n[L(\wh\theta)] - L(\theta_*) \lsim \frac{\rho^2(R^4 D_0^4 + 1)\max(\deff, \rho d)}{n}
\end{equation}
whenever~$n \gsim (R^4 D_0^4 + 1)$. 
It should be noted that the curvature parameter~$\rho$ that appears in these results, as well as in our results for pseudo self-concordant losses, is \textit{problem-dependent}. 
In particular, it depends on the true distribution~$\cP$ of the data, and can be very large if this distribution is chosen adversarially. 
By constructing such an adversarial distribution,~\cite{hazan2014logistic} prove a \textit{lower bound}~$\Omega(\sqrt{RD/n})$, i.e., for the excess risk of any algorithm, in logistic regression in the finite-sample regime~$n = O(e^{RD})$. 
This implies that~$\rho$ grows super-polynomially in~$RD$ for this distribution.
Notably, the lower bound of~\cite{hazan2014logistic} is not applicable in the setting of \textit{improper prediction}, where one is allowed to estimate $\eta_* := X^\T \theta_*$ with any predictor~$\wh\eta: X \mapsto \R$, not necessarily  with a linear one. Making such an observation,~\cite{foster2018logistic} recently proposed an improper estimator which attains the excess risk~$O(d/n)$ up to logarithmic factors in~$RD$,~$n$, and~$1/\delta$. 
Their estimator reduces to Vovk's Aggregating Algorithm~\cite{vovk1998game} for online convex optimization, combined with a simple ``boosting the confidence'' scheme proposed in~\cite{mehta2016fast}.

\paragraph{Non-parametric setup and $\ell_2$-regularization.}
After the arXiv preprint of this work began circulating,\footnote{This happened in October 2018.} our analysis was extended in~\cite{marteau2019beyond} to~$M$-estimators with~$\ell_2$-regularization, including the case of infinite-dimensional parameter. The reader might refer to~\cite{marteau2019beyond} for an overview of related work in this direction. 
In a nutshell,~\cite{marteau2019beyond} proves asymptotically near-optimal bounds~$\wt{O}(\deff/n)$ on the ``variance'' term corresponding to the excess risk~$L(\wh \theta_{\lambda,n})-L(\theta_{\lambda})$, and the additional ``bias'' term~$L(\theta_{\lambda}) - L(\theta)$, under condition~\eqref{eq:degrees-of-freedom}, and without extra conditions in the vein of those in~\cite{bach2010self-concordant}. Moreover, it is shown that the classical source and capacity conditions~\cite{caponnetto2007optimal}, known to lead to faster non-parametric rates in ridge regression, can be extended to~$M$-estimators with self-concordant losses. However,~\cite{marteau2019beyond} does not extend our improved results with near-linear sample size (Theorems~\ref{th:crb-mine-improved}--\ref{th:crb-fake-improved}) to the~$\ell_2$-penalized case.
We believe that such extension is possible by replacing Theorem~\ref{th:covariance-subgaussian-simple} with a similar result for regularized covariance matrices such as~\cite[Thm.~9]{koltchinskii2017concentration}. This, however, would be of little practical interest, since typically under source condition,~$\df_{\lambda}$ is a constant depending only on the rate of decay of the eigenvalues of~$\He$. 

\paragraph{Quasi-Newton algorithms.}
We also mention in passing that recently there has been a surge of interest in stochastic quasi-Newton methods applied to the finite-sum setting with self-concordant losses, see, e.g.,~\cite{zhang2015disco},~\cite{zhou-goldfarb}. However, none of these works establishes generalization bounds for the associated estimator.
In fact, such bounds have recently been established in the work~\cite{marteau2019globally} for a certain (globally convergent) ad-hoc quasi-Newton scheme. 
These generalization bounds are similar to those established in~\cite{marteau2019beyond} for the exact ERM, with similar criticism.

\paragraph{Empirical processes.}
The use of empirical processes in the context of parametric estimation was pioneered in~\cite{spokoiny2012parametric}, which has been one of the main inspirations for our work.
Apart from the technical difficulty in the proofs, our main critique of~\cite{spokoiny2012parametric} is the \textit{global} conditions they impose -- most importantly, they require~$\nabla L_n(\theta)$ to be subgaussian uniformly over the whole parametric set~$\Theta$. 
As can be seen from the proof of Proposition~\ref{prop:logistic-case-study} in Appendix~\ref{sec:subgaussian-check}, verification of such global conditions can be quite technical; moreover, such conditions can in fact be way more restrictive than their local counterparts (see, e.g., the bounds in Proposition~\ref{prop:logistic-case-study} which degrade drastically when~$\|\theta_*\|_{\Cov} \gg 1$, and which are sharp as can be seen from the analysis). 

Recently we learned about the work~\cite{mei2018landscape} that applies empirical processes to study constrained empirical risk minimization with smooth \textit{non-convex} losses. 
(Its preprint came out after that of our work.) 
Essentially,~\cite{mei2018landscape} proves that in the regime~$n \gsim d \log d$ (resp.~$n \gsim \s \log d$ in the high-dimensional setup), the sample gradients~$L_n(\theta)$ and Hessians~$\nabla^2 L_n(\theta)$ uniformly converge to their population counterparts, assuming that~$\nabla L_n(\theta)$ is subgaussian, and~$\nabla^2 L_n(\theta)$ is subexponential on the whole domain~$\Theta$.
This allows to establish correspondence between the stationary points of~$L(\cdot)$  and~$L_n(\cdot)$. While the focus of~\cite{mei2018landscape} is different, we expect that one may also prove asymptotically optimal rates for the excess risk in this regime, i.e., prove ``local'' analogues of our improved results (cf.~Section~\ref{sec:improved}), by localizing a minimizer~$\wh\theta_n$ to the unit Dikin ellipsoid of the associated population risk minimizer~$\theta_*$.\footnotemark
\footnotetext{This, however, would require restating their Assumptions 1-3 in affine-invariant manner.}
However,~\cite{mei2018landscape} has the same limitations as~\cite{spokoiny2012parametric}, requiring ``global'' conditions on the tails of~$\nabla L_n(\theta)$ and~$\nabla^2 L_n(\theta)$. Similar criticism applies to the literature on~$\ell_1$-regularized~$M$-estimators in high dimensions~(\cite{van2012quasi,negahban2012unified,loh2017statistical}); we discuss these results in more detail  in Section~\ref{sec:sparsity}.



\paragraph{Further related work on~$\ell_1$-regularized~$M$-estimators.}

Interestingly,~\cite{zhang2017optimal} showed that, in absense of restricted-eigenvalue (RE) type conditions imposed on the (fixed) design matrix, decomposable regularizers only lead to slow~$O(1/\sqrt{n})$ rates, even with quadratic loss. Hence, the light-tailed design condition that we impose appears to be necessary when~$\ell_1$-regularized~$M$-estimators are considered. 

Not directly relevant to our results here, we note that, whenever the computational considerations are not important, 
the~$\ell_1$-regularization can perform worse than other types of regularization. In fact,~$\ell_0$-regularized estimators are known to achieve~$O(\s/n)$ rate for the prediction error without RE-type conditions (in the fixed-design setup)~\cite{zhang2017optimal}. Moreover, there are other (non-convex) penalties that have favorable statistical properties without incoherence, see, e.g.,~\cite{loh2017support,loh2015regularized,loh2017statistical,loh2011high}. 

After the preprint of this work had been publicized, the statistical performance of regularized~$M$-estimators has been studied in the asymptotic regime~$n,d \to \infty$ with~$d/n = c$, including the ``high-dimensional'' case with~$c \gg 1$ -- see~\cite{thrampoulidis2018precise,sur2019modern}.
\section{Conclusion}

Our work sheds light on the mechanism behind the transition to the fast-rate regime in~$M$-estimators with smooth losses.
Our analysis allows to deal with~$M$-estimators with losses satisfying self-concordance-type assumptions, including logistic regression, other generalized linear models, and robust regression.
Self-concordance assumptions allow to control the precision of the local quadratic approximations of empirical risk.
Simple analysis under minimal assumptions leads to a fast-rate guarantee for large sample sizes -- larger (in order) than~$d \cdot \deff$, where~$\deff$ is the effective dimensionality of the parametric model.
However, a refined analysis under slightly stronger assumptions leads to the~$O(\max \{\deff, d\log d \})$ sample size threshold.
This is done through a combination of self-concordance with a covering argument, allowing to control the uniform deviations of the empirical risk Hessian in the Dikin ellipsoid around the population risk minimizer.
We also extend the analysis to~$\ell_1$-regularized~$M$-estimators in the high-dimensional regime. 
Finally, we verify the empirical performance of a canonically self-concordant analogue of the logistic loss in numerical experiments.



\dimacorr{
\paragraph{Improved results for~$\ell_1$-regularized estimators.}
For~$\ell_1$-regularized estimators, we have only showed a suboptimal result that requires~$n =\Omega(\s^2)$ up to a logarithmic factor. Improving this sample size bound to a near-linear one is a long-standing open problem, see, e.g,~\cite{van2012quasi}. We believe that our covering argument in Section~\ref{sec:improved} can be extended to~$\ell_1$-penalized estimators, providing near-linear bounds for the critical sample size. 


\paragraph{Efficient algorithms.}
It would also be interesting to revisit algorithmically efficient procedures such as stochastic approximation, varianced reduction techniques, and quasi-Newton methods. In particular, one could be interested in extending Hessian-sketching procedures from least-squares linear regression to general $M$-estimators with (pseudo) self-concordant losses. On the other hand, the stand-of-the-art stochastic-approximation-type algorithm for logistic regression in~\cite{bach-moulines} as well relies on Hessian approximation, and is also worth revisiting in this connection.

\paragraph{Matrix-parametrized models.}
We did not investigate $M$-estimators with matrix-valued design and predictors, arising, for example, in covariance matrix estimation and independent component analysis.
In some of them, one commonly uses the log-determinant loss which is self-concordant in the sense of~\cite{nemirovski1994interior}.
Our techniques may shed light on the statistical performance of such estimators in finite-sample regimes with random measurements.
}{}

\section*{Acknowledgments}
The first author has been supported by the ERCIM Alain Bensoussan Fellowship while working on this project. 
The second author acknowledges the support of the European Research Council (grant SEQUOIA~724063).

\appendix
\section{Probabilistic tools}
\label{sec:prob-tools}

\subsection{Subgaussian distributions}

We recall the definition of subgaussian norm for random variables $\xi \in \R$:
\[
\|\xi\|_{\psi_2} := \inf \big\{\sigma > 0 : \;\; \E [e^{\xi^2/\sigma^2}] \le e \big\}.
\]
The lemma below provides equivalent definitions of the subgaussian norm.
\begin{lemma}[{\cite[Lemma~5.5]{vershynin2010introduction}}]
\label{lem:subg-properties}
There exists an absolute constant~$c > 0$ such that $\|\xi\|_{\psi_2} \le \sigma$ is equivalent to either of the following:
\begin{enumerate}
\item Subgaussian tails: for any $t \ge 0$,
$
\Prob\left\{|\xi| > t\right\} \le \exp\left(1-{ct^2}/{\sigma^2}\right).
$
\item Subgaussian moments: for any $p \ge 1$,
$
\E[|\xi|^{p}]^{1/p} \le c\sigma \sqrt{p}.
$
\end{enumerate}
Moreover, if~$\E[\xi] = 0$, each of these properties is equivalent to the moment bound
\[
\E \exp(t\xi) \le \exp(c\sigma^2t^2).
\]
\end{lemma}
Following~\cite{vershynin2010introduction}, we define the~$\|\cdot\|_{\psi_2}$-norm of a random vector as follows:
\begin{equation}
\label{def:subg-rd}
\|Z\|_{\psi_2} := \sup_{\theta \in \cS^{d-1}} \| \langle Z, \theta \rangle \|_{\psi_2},
\end{equation}
where $\cS^{d-1}$ is the unit sphere in $\R^d$.
Note that this is indeed a norm; in particular, it satisfies the triangle inequality: $\|Z_1 + Z_2 \|_{\psi_2} \le \|Z_1\|_{\psi_2} + \|Z_2 \|_{\psi_2}$ for any pair of random vectors $Z_1, Z_2$. 
Another elementary property is that~$\| \Ma Z \|_{\psi_2} \le \|\Ma\|_{\infty} \|Z\|_{\psi_2}$ for arbitrary matrix~$\Ma$. 
Some well-known properties of subgaussian random vectors are summarized in the following lemmas.

\begin{lemma}
\label{lem:subg-sup}
Let the entries of~$Z \in \R^d$ satisfy~$\|Z_i\|_{\psi_2} \le K$, $i \in [d]$. Then, with probability at least~$1-\delta$,
\[
\|Z\|_{\infty} \lsim K\sqrt{\log (ed/\delta)}.
\]
\end{lemma}
\begin{proof}
The claim follows from Item 1 of Lemma~\ref{lem:subg-properties} by the union bound. 
\end{proof}
Next we give a bound for the $p$-th moment of~$\|Z\|_{\infty}$. Although this bound is loose for any fixed~$p$, we only use it in the regime~$p \approx \log d$ where it is tight.\footnotemark 
\footnotetext{Tight bounds for all moments can be obtained via the Chernoff method combined with the general Orlicz norms $\|\cdot\|_{\psi_{\alpha}}$ with~$\alpha = 2/p$, see, e.g.,~\cite{pollard1990empirical}. 
It is beyond the scope of this paper.}
\begin{lemma}
\label{lem:subg-sup-moments}
In the assumptions of the previous lemma, for any~$p \ge 1$ it holds
\[
\E[\|Z\|_{\infty}^p]^{1/p} \lsim Kd^{1/p}\sqrt{p}.
\]
\end{lemma}
\begin{proof}
Using the bound from Lemma~\ref{lem:subg-sup}, we obtain
\[
\begin{aligned}
\E[\|Z\|_{\infty}^p] 
&= \int_{0}^\infty \mathds{P}\left\{\|Z\|_{\infty} \ge u \right\} \ud(u^{p}) 
\le ed \int_{0}^\infty e^{-\frac{c^2 u^2}{K^2}} \ud (u^{p}) \\
&\le ed \left(\frac{K}{c}\right)^p \frac{p}{2} \,  \Gamma\left(\frac{p}{2}\right) \le ed \left(\frac{K}{c}\right)^p \frac{p}{2} \left(\frac{p}{2}\right)^{p/2} =  \frac{d(K \sqrt{p})^{p}ep}{2(c\sqrt{2})^p}.
\end{aligned}
\]
We obtain the claim by extracting the $p$-th root and doing simple estimates.
\end{proof}

\begin{lemma}[Hoeffding-type inequality, follows from~{\cite[Lemma~5.9]{vershynin2010introduction}} via~\eqref{def:subg-rd}]
\label{lem:subg-sum}
Let $Z_1, ..., Z_n$ be i.i.d.~random vectors, then one has
$
\left\|\sum_{i=1}^n Z_i\right\|_{\psi_2}^2 \lesssim \sum_{i=1}^n \|Z_i\|_{\psi_2}^2.
$
\end{lemma}
The next result shows that the~$\|\cdot\|_{\psi_2}$-norm is stable under affine transforms.

\begin{lemma}[Subgaussian norm after affine transform and decorrelation]
\label{lem:affine-comparison}
Suppose that~$X \in \R^d$ satisfies~$\E[X] = 0$, $\Cov := \E[XX^\T]$, and~$\|\Cov^{-1/2} X\|_{\psi_2} \le K$. Then for any~$A \in \R^{d \times d}$, $b \in \R^d$, vector~$\widehat X = AX + b$ satisfies
\[
\|\widehat \Cov^{-1/2} \widehat X\|_{\psi_2} \lesssim K, \;\; \text{where} \;\; \widehat \Cov = \E[\widehat X \widehat X^\T].
\]
\end{lemma}
\begin{proof}
The quantity~$\Cov^{-1/2} X$ is invariant with respect to linear transforms, so it only remains to treat the case~$\widehat X = X + b$. 
Now, in this case,~$\widehat \Cov = \Cov + bb^\T$, and
\[
\|\widehat \Cov^{-1/2} \widehat X\|_{\psi_2} \le \|\widehat \Cov^{-1/2} X\|_{\psi_2} + \|\widehat \Cov^{-1/2} b\|_{\psi_2} \le \|\wh\Cov^{-1/2} X\|_{\psi_2} + \|\widehat \Cov^{-1/2} b\|_{2}.
\]
Since~$\wh\Cov \succq \Cov$, we have
$
\|\wh\Cov^{-1/2} X\|_{\psi_2} \le \|\Cov^{-1/2} X\|_{\psi_2} \le K.
$
On the other hand, 
\[
\|\widehat \Cov^{-1/2} b\|_{2}^2 = b^\T \widehat \Cov^{-1} b^\T \le 1. 
\]
by the Sherman-Morrison formula. 
Finally, note that~$K \gtrsim 1$, as follows from the inequality~$\E[\xi^4] \ge (\E[\xi^2])^2$ applied to~$\xi = \langle u, X \rangle$, and Item 2 of Lemma~\ref{lem:subg-properties}.
\end{proof} 

\subsection{Quadratic forms of subgaussian vectors}
\label{sec:tools-quad-forms}
We call random vector $Z \in \R^d$ \textit{isotropic} if $\E[Z] = 0$ and $\E [Z Z^\T] = \Id_d$.
The following result is a deviation bound for quadratic forms of isotropic subgaussian random vectors. 
It is obtained from~\cite[Theorem 2.1]{hsu2012tail} using the isotropicity assumption which allows to get rid of the $K^2$ factor ahead of $\tr(\Ma)$.
\begin{theorem}
\label{th:subg-quad}
Let $Z \in \R^d$ be an isotropic random vector with~$\|Z\|_{\psi_2} \le K$, and let $\Ma \in \R^{d \times d}$ be positive semidefinite. Then,
\[
\Prob\left\{ \|Z\|_{\Ma}^2 - \tr(\Ma) \ge t \right\} \le \exp\left(-c\min \left\{ \frac{t^2}{K^2\|\Ma\|_2^2}, \frac{t}{K\|\Ma\|_{\infty}} \right\} \right).
\]
In other words, with probability at least $1-\delta$ it holds
\[
\|Z\|_{\Ma}^2 - \tr(\Ma) \lesssim K^2\left(\|\Ma\|_2 \sqrt{\log\left({1}/{\delta}\right)} + \|\Ma\|_\infty \log\left({1}/{\delta}\right)\right). 
\]
\end{theorem}
\begin{corollary}
\label{cor:subg-norm}
We obtain a deviation bound for the $\ell_2$-norm of the projection of an isotropic subgaussian vector~$Z$ onto an arbitrary direction~$u \in \R^d$: with probability at least $1-\delta$ it holds
\begin{equation}
\label{eq:subg-norm}
|\langle u, Z \rangle| \lesssim \|u\|_2 K\sqrt{\log\left({e}/{\delta}\right)}.
\end{equation}
This follows, through some elementary algebra, by applying Theorem~\ref{th:subg-quad} to the rank-one matrix $\Ma = u u^\T$ which satisfies $\|\Ma\|_\infty = \|\Ma\|_2 = \tr(\Ma) = \|u\|_2^2$.
\end{corollary}
The next result follows from Theorem~\ref{th:subg-quad} since
$
\|\Ma\|_\infty \le \|\Ma\|_2 \le \tr(\Ma).
$
\begin{corollary}
\label{cor:subg-sqrt}
Under the premise of~Theorem~\ref{th:subg-quad},~$\zeta = \|Z\|_{\Ma}$ is subgaussian:
\[
\Prob\big\{ {\zeta} \ge c(1+t)K{\sqrt{\tr(\Ma)}}\big\} \le \exp(-t^2).
\]
As a consequence, 
$
\Prob\big\{ {\zeta} \ge cu{K\sqrt{\tr(\Ma)}}\big\} \le \exp\left(c_1-{u^2}/{c_2}\right),
$
so that 
\[
\left\| \zeta \right\|_{\psi_2} \le cK\sqrt{\tr(\Ma)}.
\]
\end{corollary}

\subsection{Sample covariance matrices}
\label{sec:tools-cov-matrices}
Next we focus on sample second-moment matrices of subgaussian vectors.

\begin{theorem}[{\cite[Theorem 5.39]{vershynin2010introduction}}]
\label{th:covariance-subgaussian-simple}
Assume that the random vector~$\csX \in \R^d$ satisfies~$\E[\csX \csX^\T] = \He$ and~$\|\He^{-1/2}\csX\|_{\psi_2} \le K.$
Let~$\He_n = \frac{1}{n} \sum_{i=1}^n \csX_i \csX_i^\T$ where~$\csX_1, ..., \csX_n$ are independent copies of~$\csX$. 
Whenever
\begin{equation*}
n \gtrsim K^4 (d + \log(1/\delta)),
\end{equation*}
with probability at least~$1-\delta$ it holds
\begin{equation}
\label{eq:cov-dist-equiv}
\|\Delta\|^2_{\He}/2 \le \|\Delta\|^2_{\He_n} \le 2 \|\Delta\|^2_{\He}, \quad \forall \Delta \in \R^d.
\end{equation}
\end{theorem}
Next we present an extension of this result to the high-dimensional setting. 
\begin{theorem}[{\cite[Theorem 1.6]{zhou2009restricted}}]
\label{th:covariance-sparse}
Let~$\He$,~$\He_n$, and~$\csX$ be as in the previous theorem, and suppose that~$\He$ satisfies the~$(\rho,\k,\s)$-restricted eigenvalues (RE) condition for some~$\rho, \k > 0$ and~$\s \le d$.
Namely, for any~$\Delta \in \R^d$ such that 
$
\|\Delta_{\cS_c}\|_1 \le 3\|\Delta_{\cS}\|_1,
$
where~$\cS$ is the subspace of~$\R^d$ correponding to~$\s$ largest coordinates of~$\Delta$, and~$\cS_c$ is the complement of~$\cS$, it holds
\[
{\|\Delta\|_2^2}/\rho \le \|\Delta\|_{\He}^2 \le \k \|\Delta\|_2^2.
\]
Then, whenever
\[
n \gtrsim \rho \k  K^4\s\log\left(ed/\delta\right),
\]
it holds that with probability~$\ge 1-\delta$, for any~$\Delta \in \R^d$ satisfying the RE condition,
\begin{equation*}
\|\Delta\|_{\He}^2/2 \le \|\Delta\|^2_{\He_n} \le 2\|\Delta\|_{\He}^2.
\end{equation*}
\end{theorem}
\section{Technical results for self-concordant (-like) functions}
\label{sec:qsc-properties}
Here we summarizing the technical results related to self-concordant-like functions. These results are used later on to control the population and empirical risks~$L(\theta), L_n(\theta)$ in the proofs in Appendix~\ref{sec:extra-proofs}.
In what follows, we fix~$\theta_0, \theta_1 \in \Theta$, and let~$\theta_t := \theta_0 + t(\theta_1-\theta_0)$,~$t \in [0,1]$.
We define functions~$\phi(\cdot)$,~$\phi_n(\cdot)$ on~$[0,1]$ by
\begin{equation}
\label{def:line-restriction}
\begin{aligned}
\phi(t) := L(\theta_t), \quad 
\phi_{n}(t) := L_n(\theta_t). 
\end{aligned}
\end{equation}
The next result follows from the assumptions of Section~\ref{sec:ass-qsc} via the chain rule.
\begin{proposition}
\label{prop:line-gsc}
Suppose that~$\ell_z(\cdot)$ is convex, and~$\ell'''_z(\cdot)$ exists on~$\Theta$.
\begin{enumerate}[label=\textup{(\textbf{\alph*})}]
\item \label{prop:line-gsc-bach}
If Assumption~\ref{ass:gsc-fake} is satisfied, then for any $t \in [0,1]$, one has
\begin{align}
\label{eq:line-gsc-fake-n}
&|\phi'''_n(t)| \le \phi''_n(t) \max_{i \in [n]} |\langle X_i, \theta_1-\theta_0 \rangle|, \\
\label{eq:line-gsc-fake-ave}
&|\phi'''(t)| \le \phi''(t) \sup_{x \in \X} |\langle x, \theta_1-\theta_0 \rangle|.
\end{align}
\item \label{prop:line-gsc-mine}
If Assumption~\ref{ass:gsc-true} is satisfied instead, then for any $t \in [0,1]$, one has
\begin{align}
\label{eq:line-gsc-mine-n}
|\phi'''_n(t)| &\le \phi''_n(t)  \big[\max_{i \in [n]}\phi_{Z_i}''(t)\big]^{1/2}, \\
\label{eq:line-gsc-mine-ave}
|\phi'''(t)| &\le \phi''(t) \big[\sup_{z \in \cZ} \phi_{z}''(t)\big]^{1/2}. 
\end{align}
\end{enumerate}
\end{proposition}

\begin{proof}
Recall that~$\theta_t = \theta_0 + t(\theta_1-\theta_0)$ for~$t \in [0,1]$, and denote~$\Delta := \theta_1-\theta_0$. 
Differentiating under the expectation, we obtain that the derivatives of~$\phi(t) = L(\theta_t)$ and~$\phi_n(t) = L_n(\theta_t)$ are given by
\begin{align}
\label{eq:chain-ave}
\phi^{(p)}(t) &= \E[\ell^{(p)}(Y, \langle X, \theta_t\rangle)  \langle X, \Delta \rangle^p], \\
\label{eq:chain-n}
\phi_n^{(p)}(t) &= \frac{1}{n}\sum_{i \in [n]} \ell^{(p)}(Y_i, \langle X_i, \theta_t\rangle)  \langle X_i, \Delta \rangle^p.
\end{align}
This holds for~$p \le 3$ due to the basic smoothness assumption. Now, let Assumption~\ref{ass:gsc-fake} be satisfied. Using~\eqref{eq:chain-ave} with~$p \in \{2,3\}$, we get
\[
\begin{aligned}
|\phi'''(t)|
&\le \E[|\ell'''(Y,\langle X,\theta_t \rangle)|  |\langle X, \Delta \rangle|^3] 
\le \E[\ell''(Y,\langle X,\theta_t \rangle)  \langle  X, \Delta \rangle^2]  \sup_{x \in \cX} |\langle x, \Delta \rangle|,
\end{aligned}
\]
arriving at~\eqref{eq:line-gsc-fake-ave}. Analogously, we obtain~\eqref{eq:line-gsc-fake-n} from~\eqref{eq:chain-n}, replacing~$\cX$ with~the set~$\{X_1, ..., X_n\}$. 
On the other hand, if Assumption~\ref{ass:gsc-true} is satisfied instead, 
\[
\begin{aligned}
|\phi'''(t)|
&\le \E[|\ell'''(Y,\langle X,\theta_t \rangle)|  |\langle X, \Delta \rangle|^3] 
 \le \E[\ell''(Y, \langle X, \theta_t \rangle)^{3/2}  |\langle X, \Delta \rangle|^3] \\
&\le \E[\ell''(Y, \langle X, \theta_t \rangle)  \langle X, \Delta \rangle^2] \sup_{x,y \in \cX \times \cY} \left\{ \sqrt{\ell''(y, \langle x, \theta_t \rangle)}  |\langle x, \Delta \rangle| \right\},
\end{aligned}
\]
implying~\eqref{eq:line-gsc-mine-ave}. We obtain~\eqref{eq:line-gsc-mine-n} by replacing~$\E[\cdot]$ with sample averaging. 
\end{proof}


The next proposition, whose proof follows~\cite{nesterov2013introductory}, allows to control the second derivative of the loss when it is restricted to a straight line.
\begin{proposition}
\label{prop:nesterov-second-derivative}
Suppose~$g: \R \to \R$ is differentiable, non-negative, and 
\[
\begin{aligned}
&|g'(t)| \le 2c[g(t)]^{3/2}, & \quad  \forall t \in \Rp :  c|t|\sqrt{g(0)} \le 1
\end{aligned}
\]
for some~$c \ge 0$. 
Then, for any~$t \in \Rp$ such that~$c|t|\sqrt{g(0)} \le 1$, it holds
\[
\begin{aligned}
\frac{g(0)}{(1+c|t|\sqrt{g(0)})^2} \le g(t) \le \frac{g(0)}{(1-c|t|\sqrt{g(0)})^2}.
\end{aligned}
\]
\end{proposition}

\begin{proof}
We first treat the case~$g(0) > 0$. 
Consider the segment
\[
T_0 = \big[-{1}/{c\sqrt{g(0)}}, {1}/{c\sqrt{g(0)}}\big],
\] 
and assume that~$g(t) > 0$ on the whole~$T_0$. Then, we can define~$\psi(t) := 1/\sqrt{g(t)}$ on~$T_0$, and the premise of the proposition translates to~$|\psi'(t)| \le c$. Now, let $t \in T_0$ be positive without loss of generality. 
Integrating from~$0$ to~$t$, we get
\[
-ct \le {1}/{\sqrt{g(t)}} - {1}/{\sqrt{g(0)}} \le ct.
\]
Multiplying by the product~$\sqrt{g(t) g(0)} > 0$, and rearranging the terms, we prove the claim in the case where~$g(t)$ does not vanish on~$T_0$ (the case of negative $t$ is treated analogously).
Now, let~$t_0 \in T_0$ be the leftmost zero of~$g(t)$ on~$T_0 \cup \R^+$ (recall that we still assume~$g(0) > 0$). Then the preceding argument is valid for any~$t \in [0,t_0]$, which implies that~$g(t_0) > 0$, thus yielding a contradiction. This argument can be repeated for negative~$t$, taking~$t_0$ to be the rightmost negative zero of~$g(t)$ on $T_0$. 
Hence,~$g(0) > 0$ in fact implies that $g(t) > 0$ on the whole~$T_0$.

Finally, assume that~$g(0) = 0$. Then if~$g(t) \equiv 0$ on the whole $T_0$, we are done. Otherwise, there is a point~$t' \in T_0$ in which $g(t') > 0$. W.l.o.g. assume that $t' > 0$, let $t_0$ be the rightmost zero of~$g(t)$ on~$T_0 \cup \R^+$, and take a pair of points $t_1, t_2 \in T_0$ such that~$t_0 < t_1 < t_2$. Integrating $\psi'(t)$ from~$t_1$ to~$t_2$, we get 
\[
-c(t_2-t_1) \le {1}/{\sqrt{g(t_2)}} - {1}/{\sqrt{g(t_1)}} \le c(t_2-t_1),
\]
which, after the mutiplication by~$\sqrt{g(t_1)g(t_2)}$ and rearrangement, results in 
\[
g(t_1) \ge \frac{g(t_2)}{1+(t_2-t_1)\sqrt{g(t_2)}}.
\]
When~$t_1 \to t_0$, by continuity of~$g(t)$ we get a contradiction with~$g(t_0) = 0$.
\end{proof}

The next proposition describes the local properties of multivariate functions whose restrictions to line segments behave as pseudo self-concordant functions (Case~\ref{prop:local-bach}), or similarly but with a weaker control of the third derivative (Case~\ref{prop:local-mine}). 
Case~\ref{prop:local-bach} repeats~\cite[Proposition~1]{bach2010self-concordant}, and suffices for pseudo self-concordant losses;
Case~\ref{prop:local-mine} allows to treat canonically self-concordant losses. 

\begin{proposition}
\label{prop:local}
Let $F: \Theta \to \R$ be a convex~$C^3$-mapping, fix $\theta_0, \theta_1 \in \Theta$, and let~$\phi_F(t) := F(\theta_t)$,~$\theta_t := \theta_0 + t(\theta_1 - \theta_0)$.
Assume that~$\He_0 := \nabla^2 F(\theta_0) \succ 0$.
Finally, for some~$W \in \R^d$, define
\[
S := |\lang W, \theta_1 - \theta_0 \rang|.
\] 


\begin{enumerate}[label=\textup{(\textbf{\alph*})}]

\item \label{prop:local-bach}
\cite[Proposition~1]{bach2010self-concordant}.
Suppose that~$\phi_F(t)$ satisfies
\[
|\phi'''_F(t)| \le S \phi''_F(t), \quad 0 \le t \le 1, \quad \text{then,}
\]
\begin{align}
\label{eq:local-upper-fake}
{F(\theta_1) - F(\theta_0) - \nabla F(\theta_0)^\T (\theta_1-\theta_0)} 
&\le \tfrac{e^{S}-S-1}{S^2} \|\theta_1 - \theta_0\|_{\He_0}^2,\\
\label{eq:local-lower-fake}
{F(\theta_1) - F(\theta_0) - \nabla F(\theta_0)^\T (\theta_1-\theta_0)} 
&\ge \tfrac{e^{-S}+S-1}{S^2} \|\theta_1 - \theta_0\|_{\He_0}^2.
\end{align}

\item \label{prop:local-mine}
Suppose that $\theta_{1/S} \in \Theta$, and~$\phi_F(t)$ satisfies, instead,
\[
|\phi'''_F(t)| \le \tfrac{S}{1-St} \phi''_F(t), \quad 0 \le t < {1}/{S}. \quad \text{Then}
\]
\begin{equation}
\label{eq:local-mine}
\tfrac{1}{3S^2} {\|\theta_1 - \theta_0 \|_{\He_0}^2} 
\le F(\theta_{1/S}) - F(\theta_0) - \tfrac{1}{S} {\nabla F(\theta_0)^\T (\theta_{1}-\theta_0)} 
\le \tfrac{1}{S^2} {\|\theta_1 - \theta_0 \|_{\He_0}^2}.
\end{equation}
Moreover, if~$S < 1$, we have
\begin{align}
\label{eq:local-aux}
\tfrac{1}{2+S} {\|\theta_1 - \theta_0\|_{\He_0}^2} \le {F(\theta_1) - F(\theta_0) - \nabla F(\theta_0)^\T (\theta_1-\theta_0)} 
&\le \tfrac{1}{2-S} {\|\theta_1 - \theta_0\|_{\He_0}^2}.
\end{align}

\end{enumerate}

\end{proposition}


\begin{proof}

We first treat the one-dimensional case, extending~Proposition~\ref{prop:nesterov-second-derivative}. 
\begin{lemma}[Lemma~1 in~\cite{bach2010self-concordant}]
\label{lem:bach-1d}
Let~$g: [0,1] \to \R$ be a three times differentiable and convex function such that $g''(0) > 0$, and for some~$S \ge 0$,
\[
|g'''(t)| \le S g''(t), \quad 0 \le t \le 1.
\]
Then, for any $0 \le t \le 1$, one has
\begin{equation}
\label{eq:bach-1d-first}
\tfrac{e^{-St} + St - 1}{S^2} g''(0)
\le g(t) - g(0) - g'(0)t 
\le \tfrac{e^{St} - St - 1}{S^2} g''(0), \quad 0 \le t \le 1.
\end{equation}
\end{lemma}
\begin{proof}
First assume that~$g''(t) > 0$ on~$[0,1]$. Then, the premise of the lemma implies that
$
-S\ud t \le {\ud \log g''(t)} \le S\ud t
$
for $0 \le t \le 1$. 
Integrating this, we get
\begin{equation}
\label{eq:bach-1d-second}
g''(0) e^{-St} \le g''(t) \le g''(0) e^{St}.
\end{equation}
Two more integrations successively give
\[
\tfrac{1 - e^{-St}}{S} g''(0) \le g'(t) - g'(0) \le \tfrac{e^{St}-1}{S} g''(0),
\]
and then~\eqref{eq:bach-1d-first}.
Now, let~$t_0 \in (0,1]$ be the leftmost zero of~$g''(t)$. 
Then, the preceding argument can be applied on~$[0,t_0]$, yielding a contradiction due to the left inequality in~\eqref{eq:bach-1d-second}.
\end{proof}

\begin{lemma}
\label{lem:aux-1d}
Let~$g: [0,1] \to \R$ be a three times differentiable and convex function such that $g''(0) > 0$, and for some~$S \ge 0$,
\[
|g'''(t)| \le \tfrac{S}{1-t}g''(t), \quad 0 \le t < 1.
\]
Then, for any $0 \le t \le 1$, one has
\begin{equation}
\label{eq:aux-1d-first}
\begin{aligned}
g(t) - g(0) - g'(0)t &\ge \tfrac{(1-t)^{2+S} + (2+S)t -1}{(1+S)(2+S)} g''(0), \\
g(t) - g(0) - g'(0)t &\le \tfrac{(1-t)^{2-S} + (2-S)t -1}{(1-S)(2-S)} g''(0),
\end{aligned}
\end{equation}
where the upper bound holds whenever~$S < 1$ for any $t \in [0,1)$, and whenever~$S < 2$ when~$t = 1$. In particular, taking~$t = 1$, we have
\[
\tfrac{1}{2+S}{g''(0)} 
\le g(1) - g(0) - g'(0)
\le \tfrac{1}{2-S}{g''(0)}.
\]
\end{lemma}
\begin{proof}
W.l.o.g.,~we assume~$g''(t) > 0$; the general case can be treated as in Lemma~\ref{lem:bach-1d}. We follow the same steps as in Lemma~\ref{lem:bach-1d}: after the first integration,
\begin{equation}
\label{eq:aux-1d-second}
 (1-t)^S g''(0) \le g''(t) \le  (1-t)^{-S} g''(0).
\end{equation}
Integrating two more times, and assuming~$S < 1$ for the upper bound, we get
\[
\tfrac{1-(1-t)^{1+S}}{1+S} g''(0) \le g'(t) - g'(0) \le \tfrac{1-(1-t)^{1-S}}{1-S} g''(0)
\]
and then~\eqref{eq:aux-1d-first}. When~$t = 1$, the term~$(1-S)$ vanishes from the denominator of the right-hand side of~\eqref{eq:aux-1d-first}, hence in this case we can take~$S < 2$.
\end{proof}

\begin{lemma}
\label{lem:mine-1d}
Let~$g: [0,1] \to \R$ be a three times differentiable and convex function such that $g''(0) > 0$, and for some~$S \ge 0$,
\[
|g'''(t)| \le \tfrac{S}{1-St} g''(t), \quad 0 \le t < 1/S.
\]
Then, for any $0 \le t \le 1/S$, one has
\begin{equation}
\label{eq:mine-1d-first}
\left(\tfrac{t^2}{2} -\tfrac{St^3}{6}\right) g''(0)
\le g(t) - g(0) - g'(0)t 
\le \tfrac{St + (1-St) \log(1-St)}{S^2}g''(0).
\end{equation}
In particular, taking~$t = 1/S$, we have
\[
\frac{g''(0)}{3S^2} \le g(1/S) - g(0) - \frac{g'(0)}{S}
\le \frac{g''(0)}{S^2}.
\]
\end{lemma}
\begin{proof}
We again assume w.l.o.g. that~$g''(t) > 0$. Integrating three times, we get
\[
(1-St) g''(0) \le g''(t) \le \tfrac{1}{1-St} {g''(0)},
\]
then
\begin{equation}
\label{eq:mine-1d-second}
\left(t-\tfrac{St^2}{2}\right) g''(0) \le g'(t) - g'(0) \le \left(-\tfrac{\log(1-St)}{S}\right) g''(0),
\end{equation}
implying~\eqref{eq:mine-1d-first}. The last claim follows by continuity of~$f(u) = u \log u$ at~$0$.
\end{proof}

\paragraph{Proof of the proposition}
Case~\ref{prop:local-bach}, the first statement of Case~\ref{prop:local-mine}, and the second statement of Case~\ref{prop:local-mine} follow, correspondingly, from Lemmas~\ref{lem:bach-1d},~\ref{lem:mine-1d}, and~\ref{lem:aux-1d} applied to~$g(t) = \phi_F(t)$ and using that~$g(t) = F(\theta_t)$,~$g'(0) = \lang F'(\theta_0), \theta_1 - \theta_0 \rang$, and~$g''(0) = \|\theta_1 - \theta_0\|_{\He_0}^2$.
Note that the inner-product structure of~$S$ does not play a role here, but is used in Proposition~\ref{prop:dikin}. 
\end{proof}

The next result describes the behavior of (pseudo) self-concordant functions close to the optimum. 
Case~\ref{prop:local-bach} corresponds to~\cite[Proposition 2]{bach2010self-concordant}. 
The argument for Case~\ref{prop:local-mine} appears to be new, and is of independent interest. We note that a very similar argument was independently invented by U. Marteau-Ferey in~\cite{marteau2019beyond}. 

\begin{proposition}
\label{prop:dikin}
Suppose that one of the Cases~\ref{prop:local-bach}--\ref{prop:local-mine} in Proposition~\ref{prop:local} holds with fixed~$\theta_0$, all~$\theta_1 \in \Theta$, and~$W \in \R^d$ which can depend on~$\theta_1$.
Whenever
\[
\|W\|_{\He_0^{-1}}  \| \nabla F(\theta_0) \|_{\He_0^{-1}} \le 1/4,
\]
function~$F(\theta)$ has a unique minimizer~$\tilde \theta \in \Theta$, and
$
\|\tilde \theta - \theta_0\|_{\He_0} \le 4 \| \nabla F(\theta_0) \|_{\He_0^{-1}}.
$
\end{proposition}
The key message of Proposition~\ref{prop:dikin} is that the \textit{local} information about~$F(\cdot)$ at one point efficiently amounts to the \textit{global} information about how close is this point to the optimum. 
When applied to the \textit{empirical risk} with~$\theta_0 = \theta_*$ and~$\tilde{\theta} = \wh\theta_n$, this proposition allows to localize $\wh\theta_n$ using that the quantity~$\| \nabla L_n(\theta_*) \|_{\He^{-1}}^2$ decreases at rate~$O(\deff/n)$ under the i.i.d.~assumption.

\begin{proof}

Note that from~\eqref{eq:bach-1d-second},~\eqref{eq:aux-1d-second}, or~\eqref{eq:mine-1d-second}, depending on the case, it follows that~$\nabla^2 F(\theta) \succ 0$ for any $\theta \in \Theta$, hence the minimum~$\tilde \theta$ is unique provided that it exists. 
Now, consider the level set
\[
\Theta_F(F(\theta_0)) := \{\theta \in \Theta: F(\theta) \le F(\theta_0)\}.
\]
Let~$\theta_1 \in \Theta_F(F(\theta_0))$ be arbitrary, and~$r = \|\theta_1 - \theta_0\|_{\He_0}$. 
Denote~$\nu := \|\nabla F(\theta_0)\|_{\He_0^{-1}}$ and~$R := \|W\|_{\He_0^{-1}}$; note that $S \le Rr$.
We now treat all cases of Proposition~\ref{prop:local}. 

\paragraph{Case~\ref{prop:local-bach}.}
By~\eqref{eq:local-lower-fake}, we have
\[
\begin{aligned}
F(\theta_1) 
&\ge F(\theta_0) + \langle \nabla F(\theta_0), \theta_1 - \theta_0 \rangle + \tfrac{e^{-Rr} - 1 + Rr}{R^2r^2} r^2 
\ge F(\theta_0) - \nu r + \tfrac{e^{-Rr} -1 + Rr}{R^2},
\end{aligned}
\]
where we first used that~$u \mapsto (e^{-u} - 1 + u)/u$ is a decreasing function, and then the Cauchy-Schwarz inequality.
Denoting $u = Rr$, we arrive at
\begin{equation}
\label{eq:bach-decrement-equation}
e^{-u} -1 + u \le \nu R u.
\end{equation}
By the premise, we know that~$\nu R \le 1/2$, hence
$
e^{-u} - 1 + {u}/{2} \le 0.
$
We can check numerically that this implies~$u \le 2$; moreover, one has
$
e^{-u} - 1 + u \ge {u^2}/{4}
$
for such~$u$.
Plugging this back into~\eqref{eq:bach-decrement-equation}, we arrive at~$u \le 4 \nu R$, that is,~$\|\theta_1 - \theta_0\|_{\He_0} \le 4\nu$. In other words, the level set~$\Theta_F(F(\theta_0))$ is compact and belongs to the $\|\cdot\|_{\He_0}$-ball of radius~$4\nu$ centered at~$\theta_0$. Hence, the minimum~$\tilde\theta$ exists and belongs to the same ball; it is also unique since~$F(\theta) \succ 0$.

\paragraph{Case~\ref{prop:local-mine} with~$S < 1$.}
By the lower bound in~\eqref{eq:local-aux}, we have
\[
\begin{aligned}
F(\theta_1) 
&\ge F(\theta_0) + \langle \nabla F(\theta_0), \theta_1 - \theta_0 \rangle + \tfrac{1}{2+Rr} r^2 
\ge F(\theta_0) - \nu r + \tfrac{1}{2+Rr} r^2,
\end{aligned}
\]
where we used that $u \mapsto 1/(2+u)$ is a decreasing function on~$\R^+$. Whence, 
\[
\frac{u}{u+2} \le \nu R,
\]
where~$u := Rr$. 
Since~$\nu R \le 1/2$, we have $u \le 2$. Thus, we get~$r \le 4\nu$ as required. 

\paragraph{Case~\ref{prop:local-mine} with arbitrary~$S \ge 0$.} 
First assume that~$Rr \ge S \ge 1$. Then,~$\theta_{1/S}$ belongs to the segment~$[\theta_0,\theta_1]$ and to~$\Theta$.
Whence~$F(\theta_{1/S}) \le F(\theta_0)$ by convexity of~$\Theta_F(F(\theta_0))$.
On the other hand, from the lower bound in~\eqref{eq:local-mine} we have
\[
F(\theta_{1/S}) \ge F(\theta_0) - \frac{\nu r}{S} + \frac{r^2}{3S^2}.
\]
Whence
$
\nu \ge \frac{r}{3S} \ge \frac{1}{3R},
$
and we arrive at the contradiction. Thus, the only possibility is that~$S < 1$, in which case the statement has already been proved. 
\end{proof}

\subsection{Properties of pseudo-Huber loss~\eqref{def:our-robust-loss}}

We can check that the Fenchel dual of~$\phi: (-1,1) \to \R$ defined in~\eqref{eq:log-barrier} is indeed~$\vphi(t)$, cf.~\eqref{def:our-robust-loss}, by solving a quadratic equation. 
Since~$\phi$ is a barrier on~$(-1,1)$, we have~$|\vphi'(t)| < 1$ for any~$t \in \R$. 
Now, we have~$\phi'(\vphi'(t)) = t$ for~$t \in \R$, see, e.g.,~\cite{rockafellar1970convex}. Differentiating this identity, we obtain
\begin{equation}
\label{eq:fenchel-second-derivative}
\phi''(\vphi'(t)) \cdot \vphi''(t) = 1.
\end{equation}
Clearly, the Fenchel dual of an even function is also even, hence~$\vphi'(0) = 0$, and~$\vphi''(0) = 1/\phi''(0)$. 
Differentiating once again, we get
\[
\phi'''(\vphi'(t)) \cdot [\vphi''(t)]^2 + \phi''(\vphi'(t)) \cdot \vphi'''(t) = 0,
\]
whence, using that~$\phi''(u) > 0$ for any $u \in (-1,1)$,
\[
|\vphi'''(t)| = \frac{|\phi'''(\vphi'(t))|}{\phi''(\vphi'(t))} [\vphi''(t)]^2.
\]
Whence, if~$|\phi'''(u)| \le c[\phi''(u)]^{{3}/{2}}$, we get that~$|\phi'''(u)| \le c[\phi''(u)]^{{3}/{2}}$ via~\eqref{eq:fenchel-second-derivative}.
\qed
\section{Proofs of theorems}
\label{sec:extra-proofs}

\subsection{Proof of Theorem~\ref{th:crb-fake-complex}}

$\boldsymbol{1^o}$.
Recall that~$\He = \nabla^2 L(\theta_*)$, and let~$\He_n := \nabla^2 L_n(\theta_*)$. 
Note that due to Assumption~\ref{ass:second-subg} and the first bound on~$n$ in the premise of the theorem, we can apply Theorem~\ref{th:covariance-subgaussian-simple} to~$\He_n$ and~$\He$. 
Thus, with probability at least~$1-\delta$ we have
\begin{equation}
\label{eq:hess-estimated}
\tfrac{1}{2} \He \prccq \He_n \prccq 2\He.
\end{equation}
On the other hand, we can prove~\eqref{eq:score-finite-complex} using Assumption~\ref{ass:first-subg}. 
Indeed, the vectors
\[
\nabla \ell_{Z_i}(\theta_*) = \ell'(Y_i,X_i^\T\theta_*) X_i, \quad i \in [n],
\] 
are independent, zero mean and with covariance~$\Gr$.
Hence, the random vectors~$\Gr^{-1/2} \nabla \ell_{Z_i}(\theta_*)$.~$i \in [n]$, are independent and isotropic (have zero mean and unit covariance). 
Moreover, by Assumption~\ref{ass:first-subg},~$\| \Gr^{-1/2} \nabla \ell_{Z_i}(\theta_*) \|_{\psi_2} \le K_1$.
Hence, by Lemma~\ref{lem:subg-sum} about the subgaussian norm of the sum of i.i.d.~random vectors, we have that the random vector 
$
V_n := \sqrt{n}\Gr^{-1/2}\nabla L_n(\theta_*),
$
is isotropic, satisfies~$\|V_n\|_{\psi_2} \lsim K_1$, and, moreover,
\begin{equation}
\label{eq:score-quad-form}
\| \nabla L_n(\theta_*) \|_{\He^{-1}}^2 = \tfrac{1}{n}{\|V_n\|_{\Ma}^2} \;\; \text{with} \;\; \Ma := \Gr^{1/2} \He^{-1} \Gr^{1/2}.
\end{equation}
Using that~$\|\Ma\|_\infty \le \|\Ma\|_2 \le \tr(\Ma) = \deff$, by Theorem~\ref{th:subg-quad}, we arrive at~\eqref{eq:score-finite-complex}.

$\boldsymbol{2^o}$.
Our next goal is proving~\eqref{eq:crb-finite-complex}.
Let~$\mu := \E[X]$ and~$\Cov_o := \E[(X-\mu)(X-\mu)^\T]$ so that~$\Cov = \Cov_o + \mu \mu^\T$. 
Denoting~$\Qu = \Cov_o^{1/2} \Cov^{-1} \Cov_o^{1/2}$, we have
\begin{equation}
\label{eq:Q-decomposition}
\begin{aligned}
&\|X_i\|_{\Cov^{-1}}^2 
= \|X_i-\mu\|_{\Cov^{-1}}^2 + 2 \langle \Cov^{-{1}/{2}} \mu, \Cov^{-{1}/{2}}(X_i - \mu)  \rangle + \|\mu\|_{\Cov^{-1}}^2\\
&= \|\Cov_o^{-{1}/{2}}(X_i-\mu)\|_{\Qu}^2 + 2 \langle \Qu^{{1}/{2}} \Cov_o^{-{1}/{2}} \mu, \Qu^{{1}/{2}}\Cov_o^{-{1}/{2}}(X_i - \mu)  \rangle 
 + \|\Cov^{-{1}/{2}}\mu\|_{2}^2.
\end{aligned}
\end{equation}
By construction,~$\Cov_o^{-1/2} (X_i - \mu)$ is isotropic. Moreover,~$\|\Cov_o^{-1/2} (X_i-\mu) \|_{\psi_2} \lsim K_0$ due to Assumption~\ref{ass:design-subg} and Lemma~\ref{lem:affine-comparison}.
Note that~$\|\Qu\|_2 \le \tr(\Qu) \le d$ and~$\|\Qu\|_\infty \le 1$. 
Hence, by Theorem~\ref{th:subg-quad}, with probability at least~$1-\delta$ one has
\[
\begin{aligned}
\|\Cov_o^{-1/2}(X_i-\mu)\|_{\Qu}^2 
&\lesssim K_0^2 d \big[\sqrt{\log\left({e}/{\delta}\right)} + \log\left({1}/{\delta}\right)\big]
\lsim K_0^2 d\log\left({e}/{\delta}\right).
\end{aligned}
\]
Now, the second term in the right-hand side of~\eqref{eq:Q-decomposition} can be controlled as follows: 
\[
\begin{aligned}
|\lang \Qu^{1/2} \Cov_o^{-1/2} \mu, \Qu^{1/2} \Cov_o^{-1/2}(X_i - \mu)  \rang| 
&\le \|\Qu\|_{\infty}^{1/2}   \|\Qu^{1/2} \Cov_o^{-1/2}\mu \|_{2}  \| \Cov_o^{-1/2}(X_i - \mu) \|_2 \notag \\
&= \|\Qu\|_{\infty}^{1/2}  \|\Cov^{-1/2}\mu \|_{2}  \| \Cov_o^{-1/2}(X_i - \mu) \|_2 \notag \\
&\le \|\Cov^{-1/2}\mu \|_{2}  \| \Cov_o^{-1/2}(X_i - \mu) \|_2 \notag \\
&\lsim K_0 \sqrt{d\log\left({e}/{\delta}\right)} \|\Cov^{-1/2}\mu \|_{2},
\end{aligned}
\]
where the last inequality holds with probability~$\ge 1-\delta$ by Corollary~\ref{cor:subg-norm}.
Finally, 
\[
\|\Cov^{-1/2}\mu\|_{2}^2 \le \mu^\T \Cov^{-1}\mu = \mu^\T (\Cov_o + \mu\mu^\T)^{-1}\mu \le 1.
\]
Combining these results with the union bound,~\eqref{eq:hess-estimated}, and Assumption~\ref{ass:curvature}, we have
\begin{equation}
\label{eq:design-bound-proof}
\max_{i \in [n]} \|X_i\|_{\He_n^{-1}}^2 \lsim \rho K_0^2 d \log\left({en}/{\delta}\right), \quad  \forall i \in [n]
\end{equation}
with probability~$\ge 1-\delta$.
Now,~\eqref{eq:score-finite-complex},~\eqref{eq:hess-estimated}, and the 2nd bound in~\eqref{eq:n-fake-complex} imply that
\begin{equation}
\label{eq:dikin-fake}
\max_{i \in [n]} \|X_i\|_{\He_n^{-1}}^2 \| \nabla L_n(\theta_*) \|_{\He_n^{-1}}^2 \le c.
\end{equation}
Now, putting~$c = 1/4$, this results in~\eqref{eq:crb-finite-complex}. 
Indeed, invoking the bound~\eqref{eq:line-gsc-fake-n} of~Proposition~\ref{prop:line-gsc}, we see that~$L_n(\cdot) $ falls into Case~\ref{prop:local-bach} of Proposition~\ref{prop:local} with~$\theta_0 = \theta_*,$~$\He_0 = \He_n$, and~$W(\theta) = X_{j(\theta)}$ for $j(\theta) \in \Argmax_{i \in [n]} |\lang X_i, \theta - \theta_* \rang|.$ 
Hence, we can apply Proposition~\ref{prop:dikin}: clearly,~$\|W(\theta)\|_{\He_n^{-1}} \le \max_{i \in [n]} \|X_i\|_{\He_n^{-1}}$ for all~$\theta \in \Theta$, and then~\eqref{eq:dikin-fake} with~$c = 1/4$ implies that the minimizer~$\wh\theta_n$ of~$\wh L_n(\cdot)$ is unique and satisfies~$\|\wh \theta - \theta_*\|_{\He_n}^2 \le 4 \| \nabla L_n(\theta_*) \|_{\He_n^{-1}}^2$.
By~\eqref{eq:hess-estimated}, this results in~\eqref{eq:crb-finite-complex}.

$\boldsymbol{3^o}$.
Let us now prove~\eqref{eq:restricted-risk-bound-neat}.
To this end, consider the restricted risk $L_{\cE_0}(\theta)$, fix two arbitrary points $\theta_0, \theta_1 \in \Theta$, and consider function $\phi_{\cE_0}(t) := L_{\cE_0}(\theta_t)$ where~$\theta_t = \theta_0 + t(\theta_1-\theta_0)$ for~$t \in [0,1]$. Differentiating~$\phi_{\cE_0}(t)$ three times (note that $\cE_0$ does not depend on $\theta$), we see that~\eqref{eq:line-gsc-fake-ave} can now be replaced with
\[
|\phi_{\cE_0}'''(t)| \le \phi_{\cE_0}''(t) \sup_{x \in \X_{\cE_0}} |\langle x, \theta_1-\theta_0 \rangle|,
\]
where
$
\cX_{\cE_0} := \{x \in \cX: \| x \|_{\He^{-1}} \le \sqrt{\rho} \B_0\},
$
with~$\B_0 := K_0 \sqrt{d \log(e/\delta)}$, is the~$(1-\delta)$-confidence set (under~$\cE_0$) for~$X$. 
(We used Assumption~\ref{ass:curvature}.)
Besides, let us momentarily assume that the new Hessian~$\He_{\cE_0} := \nabla^2 L_{\cE_0}(\theta_*)$ is invertible, and approximates~$\He$ in the positive-semidefinite sense: for some constants~$c,C > 0$,
\begin{equation}
\label{eq:hessian-equivalence-after-restriction}
c\He \prccq \He_{\cE_0} \prccq C\He.
\end{equation}
Later on, we will verify this under condition~\eqref{eq:restricted-risk-condition} on~$\delta$.
Now, under~\eqref{eq:hessian-equivalence-after-restriction}, we can apply Case~\ref{prop:local-bach} of Proposition~\ref{prop:local} to~$L_\cE(\cdot)$ with~$\theta_0 = \theta_*$,~$\theta_1 = \wh \theta_n$,~$\He_0 = \He_{\cE_0}$, and~$W = W(\theta) \in \Argmax_{x \in \cX_{\cE_0}} |\lang x, \theta - \theta_* \rang|$.
Observe that~$\|W(\theta)\|_{\He^{-1}} \le \sqrt{\rho} \B_0$, and let~$r := \|\wh\theta_n-\theta_*\|_{\He}^2$. 
By~\eqref{eq:local-upper-fake} combined with the Cauchy-Schwarz inequality,
%
\begin{equation*}
\begin{aligned}
L_{\cE_0}(\wh\theta_n) - L_{\cE_0}(\theta_*) 
&\lsim \left(\frac{e^{\sqrt{\rho} \B_0 r} - 1 - \sqrt{\rho} \B_0 r}{\rho \B_0^2 r^2}\right) r^2 + \nabla L_{\cE_0}(\theta_*)^\T (\wh\theta_n-\theta_*).
\end{aligned}
\end{equation*}
Now, observe that the term in the parentheses is at most a constant. 
Indeed,~$\sqrt{\rho}\B_0 r \lesssim 1$ follows from the combination of~\eqref{eq:n-fake-complex}--\eqref{eq:crb-finite-complex}, and then~$f(u) = e^{u}-1-u \lesssim u^2$ whenever~$u \lesssim 1$ (in particular,~$f(u) \le u^2$ when~$u \le 1$). 
Thus, 
\begin{equation}
\label{eq:restricted-risk-ugly}
L_{\cE_0}(\wh\theta_n) - L_{\cE_0}(\theta_*) 
\lsim r^2 + r\| \nabla L_{\cE_0}(\theta_*) \|_{\He^{-1}}.
\end{equation}
In order to prove~\eqref{eq:restricted-risk-bound-neat}, it remains to control~$\| \nabla L_{\cE_0}(\theta_*) \|_{\He^{-1}}$ and to verify~\eqref{eq:hessian-equivalence-after-restriction}. 

$\boldsymbol{4^o}$.
To estimate the additional term in~\eqref{eq:restricted-risk-ugly},
consider the \textit{complementary risk:}
\[
L_\cEc(\theta_*) := \E[\ell_Z(\theta_*) \ind_{\cEc}(X)],
\] 
where $\cEc$ is the complement of $\cE_0$, so that~$\Prob(\cEc) \le \delta$. 
Note that, since~$\nabla L(\theta_*) = 0$, we have~$\nabla L_{\cE_0}(\theta_*) = -\nabla L_\cEc(\theta_*),$ 
whence 
\[
\| \nabla L_{\cE_0}(\theta_*) \|_{\He^{-1}} = \| \nabla L_\cEc(\theta_*) \|_{\He^{-1}}.
\]
We now estimate $\| \nabla L_\cEc(\theta_*) \|_{\He^{-1}}$ through a technique inspired by the one in~\cite[Section~1.3]{vershynin2011approximating}.
For any~$p$,~$q$ such that~$1/p + 1/q = 1$, we have by H\"older's inequality:
\begin{align}
\label{eq:restricted-risk-young}
\| \nabla L_\cEc(\theta_*) \|_{\He^{-1}} 
\le  \E [ \| \nabla \ell_Z(\theta_*) \|_{\He^{-1}} \ind_{\cEc}]
&\le \E [\| \nabla \ell_Z(\theta_*) \|_{\He^{-1}}^p] ^{1/p}  \delta^{1/q},
\end{align}
Note that 
\[ 
\| \nabla \ell_Z(\theta_*) \|_{\He^{-1}}^2 = \|\Gr^{-1/2} \nabla \ell_Z(\theta_*)\|_{\Ma}^2
\] 
where $\Ma = \Gr^{1/2} \He^{-1} \Gr^{1/2}$, and~$\Gr^{-1/2} \nabla \ell_Z(\theta_*)$ is isotropic and satisfies
\[
\|\Gr^{-1/2} \nabla \ell_Z(\theta_*) \|_{\psi_2} \le K_1.
\] 
Hence, by Corollary~\ref{cor:subg-sqrt},~$\zeta := \| \He^{-1/2} \nabla \ell_Z(\theta_*) \|$ satisfies~$\|\zeta\|_{\psi_2} \lesssim K_1 \sqrt{\deff}.$
As such, we can bound the moments of $\zeta$ using Lemma~\ref{lem:subg-properties}:
\[
\E [\| \nabla \ell_Z(\theta_*) \|_{\He^{-1}}^p]^{1/p} \lsim K_1 \sqrt{p\deff}.
\]
Combining this with~\eqref{eq:restricted-risk-ugly}--\eqref{eq:restricted-risk-young} and~\eqref{eq:score-finite-complex}--\eqref{eq:crb-finite-complex}, and choosing~$p = \log(e\deff)$ and~$q = 1+{1}/{\log(\deff)}$, we obtain 
\[
\begin{aligned}
&L_{\cE_0}(\wh\theta_n) - L_{\cE_0}(\theta_*) 
\lesssim  K_1^2 \sqrt{\tfrac{\deff \log\left({e}/{\delta}\right)}{n}  } \bigg( \sqrt{\tfrac{\deff \log\left({e}/{\delta}\right)}{n}} + \delta^{\frac{\log(\deff)}{\log(\deff)+1}} \sqrt{\deff \log(e\deff)} \bigg).
\end{aligned}
\]
Finally,~\eqref{eq:restricted-risk-condition} implies that
$
\delta^{\frac{\log(\deff)}{\log(\deff)+1}} \sqrt{\log(\deff)} \lesssim \sqrt{{\log(e/\delta)}/{n}},
$
and~\eqref{eq:restricted-risk-bound-neat} follows.

$\boldsymbol{5^o}$.
It remains to verify~\eqref{eq:hessian-equivalence-after-restriction}, i.e., that the Hessians~$\He$ and~$\He_{\cE_0}$ are close. 
First, the upper bound in~\eqref{eq:hessian-equivalence-after-restriction} is trivial. Indeed defining the complementary~$\He_{\cEc} := \nabla^2 L_{\cEc}(\theta_*)$, we see that 
$
\He_{\cE_0} = \He - \He_{\cEc} \prccq \He
$
since~$\He_{\cEc} \succq 0$. On the other hand, the lower bound in~\eqref{eq:hessian-equivalence-after-restriction} with~$c \in (0,1)$ would follow from the bound
\[
\|\He^{-1/2} \He_{\cEc} \He^{-1/2}\|_{\infty} \le c',
\]
where $c' \in (0,1)$.
Let us show that this bound is satisfied under the second bound in~\eqref{eq:restricted-risk-condition}, using a technique similar to the one used to control~$\nabla L_{\cE_0}(\theta_*)$.
For any $p,q \ge 1$~such that~$1/p + 1/q = 1$, we have by H\"older's and Young's inequalities:
\[
\begin{aligned}
\|\He^{-1/2} \He_{\cEc} \He^{-1/2}\|_{\infty} 
&\le \E[\|\He^{-1/2} \nabla^2 \ell_{Z}(\theta_*)\He^{-1/2}\|_{\infty}^p]^{1/p} \delta^{1/q} \\
&= \E[\|\He^{-1/2} \csX \csX^\T \He^{-1/2}\|_{\infty}^p]^{1/p} \delta^{1/q} \\
&= \E[\|\He^{-1/2} \csX\|_{2}^{2p}]^{1/p} \delta^{1/q}
\lsim K_2^2 p d \delta^{1/q},
\end{aligned}
\]
where in the end we used that~$\zeta = \|\He^{-1/2} \csX\|_{2}$ satisfies~$\|\zeta\|_{\psi_2} \le K_2\sqrt{d}$ by Corollary~\ref{cor:subg-sqrt}. Choosing~$p = \log(ed)$, we see that~$K_2^2 p d \delta^{1/q} \lsim 1$ under~\eqref{eq:restricted-risk-condition}. 
\qed

\subsection{Proof of Theorem~\ref{th:crb-mine-complex}}

\paragraph{Disclaimer.}
The key distinction from Theorem~\ref{th:crb-fake-complex} is the absence of curvature parameter~$\rho$ in the derived critical sample size (cf.~\eqref{eq:n-mine-complex} viz.~\eqref{eq:n-fake-complex}). 
This improvement is achieved by carefully exploiting Assumption~\ref{ass:gsc-true}. 
In particular, we invoke Case~\ref{prop:local-mine}, instead of Case~\ref{prop:local-bach}, in Propositions~\ref{prop:line-gsc} and~\ref{prop:local}. 
Meanwhile, the role of the bounding vector~$W$ is now relegated from~$X$ to~$\wt X = \ell''(Y, X^\top\theta_*)^{1/2}X$. 

The proof of the theorem below recycles results from the proof of Theorem~\ref{th:crb-fake-complex}.

\begin{proof}

We repeat step~$\boldsymbol{1^o}$ in the previous proof {\em verbatim}, arriving at~\eqref{eq:score-finite-complex} and~\eqref{eq:hess-estimated}. 

$\boldsymbol{2^o}$.
In order to prove~\eqref{eq:crb-finite-complex}, we use Case~\ref{prop:line-gsc-mine} of Proposition~\ref{prop:line-gsc}. 
To this end, fix arbitrary~$\theta \in \Theta$, let~$\theta_t = \theta_* + t(\theta - \theta_*)$ for~$t \in [0,1)$, and define~$\phi_z(t) := \ell_z(\theta_t)$ for arbitrary~$z \in \cZ$.
Due to Assumption~\ref{ass:gsc-true}, for any~$z \in \cZ = \R^d \times \cY$, we have~$|\phi_z'''(t)| \le 2[\phi_z''(t)]^{3/2}.$
Hence, we can apply Proposition~\ref{prop:nesterov-second-derivative} to~$g(t) = \phi_z''(t)$ with~$c = 1$.
Thus, with~$\wt  x := [\ell''(y, x^\T\theta_*)]^{1/2}x$ for arbitrary~$(x,y) \in \cZ$, we have
\begin{equation}
\label{eq:individual-loss}
\begin{aligned}
\phi_{z}''(t) 
&\le \frac{\phi''_z(0)}{(1-t\sqrt{\phi''_z(0)})^2}
= \frac{\lang \wt  x, \theta-\theta_* \rang^2}{(1-t |\lang \wt x, \theta-\theta_* \rang|)^2}
\end{aligned}
\end{equation}
for any~$t \ge 0$ such that the denominator is non-zero.
Combining this with~\eqref{eq:line-gsc-mine-n},
\begin{equation}
\label{eq:curvature-extracted-n}
\begin{aligned}
|\phi'''_n(t)| 
&\le \phi''_n(t) \max_{i \in [n]} \frac{|\lang \wt X_i, \theta-\theta_* \rang|}{1-t |\lang \wt X_i, \theta-\theta_* \rang|} 
= \phi''_n(t) \frac{|\lang \wt X_{j(\theta)}, \theta-\theta_* \rang|}{1-t |\lang \wt X_{j(\theta)}, \theta-\theta_* \rang|},
\end{aligned}
\end{equation}
where~$j(\theta) \in \Argmax_{i \in [n]} |\lang \wt X_i, \theta-\theta_* \rang|$, and again we can take any~$t \ge 0$ such that the denominator is positive.  
Thus,~$L_n(\theta)$ falls into Case~\ref{prop:local-mine} of Proposition~\ref{prop:local} with $\theta_0 = \theta_*$,~$\He_0 = \He_n$, and~$W = W(\theta) = \wt X_{j(\theta)}$.
On the other hand, repeating the analysis that led to~\eqref{eq:design-bound-proof}, we obtain that, for any fixed~$\theta$,
\[
\|\wt X_{j(\theta)}\|_{\He_n^{-1}}^2 \lsim \B_2^2 := K_2^2 d \log\left({en}/{\delta}\right)
\]
with probability~$\ge 1-\delta$. Combining this result with the second bound in~\eqref{eq:n-mine-complex}, 
\begin{equation}
\label{eq:dikin-mine}
\|\wt X_{j(\theta)}\|_{\He_n^{-1}}^2 \| \nabla L_n(\theta_*) \|_{\He_n^{-1}}^2 \lsim 1,
\end{equation}
cf.~\eqref{eq:dikin-fake}. 
Hence, we can apply Proposition~\ref{prop:dikin} to~$L_n(\theta)$ at~$\theta_0 = \theta_*$, and repeating the final argument in step~$\boldsymbol{2^o}$ of the proof of Theorem~\ref{th:crb-fake-complex}, we arrive at~\eqref{eq:crb-finite-complex}.

$\boldsymbol{3^o}$.
We now prove~\eqref{eq:restricted-risk-bound-neat} with~$L_{\cE_0}$ replaced by~$L_{\cE_2}$. 
Similarly to~\eqref{eq:curvature-extracted-n}, from~\eqref{eq:individual-loss} and~\eqref{eq:line-gsc-mine-ave} we have
\begin{equation}
\label{eq:curvature-extracted-ave}
\begin{aligned}
|\phi'''(t)| 
&\le \phi''(t) \frac{|\lang W(\theta), \theta-\theta_* \rang|}{1-t |\lang W(\theta), \theta-\theta_* \rang|},
\end{aligned}
\end{equation}
with probability~$\ge 1-\delta$, where~$W(\theta) \in \Argmax_{x \in \wt\cX_{\cE_2}} |\lang x, \theta-\theta_* \rang|$ for the set
\[
\wt\cX_{\cE_2} := \{\wt x =  [\ell''(y, x^\T\theta_*)]^{1/2}x : \; \| \wt x \|_{\He^{-1}}^2 \lsim \B_2^2\}.
\]
(Clearly,~$\wt\cX_{\cE_2}$ is the~$(1-\delta)$-confidence set for the new observation~$\wt X$.)
Thus, 
\[
|\lang W(\theta), \wh \theta_n - \theta_* \rang| \le \B_2 r, \quad r := \|\wh\theta_n - \theta_*\|_{\He};
\]
moreover, due to~\eqref{eq:score-finite-complex},~\eqref{eq:crb-finite-complex}, and the 2nd bound in~\eqref{eq:n-mine-complex} we have~$\B_2 r \lsim 1$.
As such, whenever
$
c\He \prccq \nabla^2 L_{\cE_2}(\theta_*) \prccq C\He,
$
the restricted risk~$L_{\cE_2}(\cdot)$, cf.~\eqref{eq:restricted-risks}, falls under Case~\ref{prop:local-mine} of Proposition~\ref{prop:local} with~$\theta_1 = \wh\theta_n$ and~$S < 1$; 
the upper bound in~\eqref{eq:local-aux} then gives the analogue of~\eqref{eq:restricted-risk-ugly}:
\begin{equation*}
\begin{aligned}
L_{\cE_2}(\wh\theta_n) - L_{\cE_2}(\theta_*) 
&\lsim \frac{r^2}{2-\B_2 r} + \nabla L_{\cE_2}(\theta_*)^\T (\wh\theta_n-\theta_*) 
\lsim r^2 + r \|\nabla L_{\cE_2}(\theta_*) \|_{\He^{-1}}.
\end{aligned}
\end{equation*}
It remains to estimate the right-hand side and to verify~$c\He \prccq \nabla^2 L_{\cE_2}(\theta_*) \prccq C\He$, using~\eqref{eq:restricted-risk-condition} in both cases. This repeats steps~$\boldsymbol{4^o}$--$\boldsymbol{5^o}$ in the proof of Theorem~\ref{th:crb-fake-complex}.
\end{proof}


\subsection{Proof of Theorem~\ref{th:crb-mine-improved}}

$\boldsymbol{1^o}$.
Without loss of generality, we assume that~$\Theta = \R^d$; the argument can be extended to the general case simply by replacing all arising Dikin ellipsoids with their intersections with~$\Theta$.
For simplicity, we also assume that Assumption~\ref{ass:second-subg++} holds with~$r = 1$, and denote~$\bar K_2 := \bar K_2(1)$.
First of all, for any $r \ge 0$ and $\theta \in \Theta_1(\theta_*)$, we define the Dikin ellipsoid with center~$\theta$ and radius~$r$:
\[
\Theta_r(\theta) := \{\theta' \in \R^d: \|\theta' - \theta\|_{\He(\theta)} \le r\}.
\]
We will prove that the Hessians~$\He(\theta) := \nabla^2 L(\theta)$ are close to~$\He(\theta_*)$ within the Dikin ellipsoid with radius~$\Omega(1/\bar K_2^3)$.
To this end, fix~$\theta_0 = \theta_*$ and arbitrary~$\theta_1 \in \R^d$, and let~$\theta_t = \theta_0 + t(\theta_1-\theta_0)$,~$t \ge 0$.
By using Assumptions~\ref{ass:gsc-true} and~\ref{ass:second-subg++}, we can prove that for the function~$\phi(t) = L(\theta_t)$ it holds
\[
\phi'''(t) \le 2\bar c [\phi''(t)]^{3/2}
\]
for any~$t \ge 0$ such that~$\theta_t \in \Theta_{1/\bar c}(\theta_*)$ with~$\bar c \gsim 1/\bar K_2^3$.
Indeed, let~$\Delta := \theta_1 - \theta_0$, and recall that
\[
\phi^{(p)}(t) = \E[\ell^{(p)}(Y, \langle X, \theta_t\rangle)  \langle X, \Delta \rangle^p], \quad p \in \{2,3\},
\] 
cf. the proof of Proposition~\ref{prop:line-gsc}. Putting~$\wt X(\theta_t) := [\ell''(Y, \langle X, \theta_t\rangle)]^{1/2}X$, this gives
\[
\begin{aligned}
\phi''(t) 
&= \E[\langle \wt X(\theta_t), \Delta \rangle^2]  
= \E[\langle \He(\theta_t)^{-1/2} \wt X(\theta_t), \He(\theta_t)^{1/2} \Delta \rangle^2] = \|\Delta\|_{\He(\theta_t)}^2,
\end{aligned}
\]
On the other hand, due to Assumption~\ref{ass:gsc-true},
\[
\begin{aligned}
|\phi'''(t)| 
&\le \E[|\ell'''(Y, \langle X, \theta_t\rangle)| \cdot |\langle X, \Delta \rangle|^3] \\
&\le 2\E[|\langle [\ell''(Y, \langle X, \theta_t\rangle)]^{1/2} X, \Delta \rangle|^3] \\
&= 2\E[|\langle \He(\theta_t)^{-1/2} \wt X(\theta_t), \He(\theta_t)^{1/2} \Delta \rangle|^3].
\end{aligned}
\]
Now, recall that whenever $\theta \in \Theta_c(\theta_*)$, one has~$\|\He(\theta_t)^{-1/2} \wt X(\theta_t)\|_{\psi_2} \le \bar K_2$ due to Assumption~\ref{ass:second-subg++}.
Thus, for such~$\theta_t$ we have
\[
\|\langle \He(\theta_t)^{-1/2} \wt X(\theta_t), \He(\theta_t)^{1/2} \Delta \rangle\|_{\psi_2} \le \bar K_2 \|\Delta\|_{\He(\theta_t)},
\] 
and by Lemma~\ref{lem:subg-properties},
\[
\E[|\langle \He(\theta_t)^{-1/2} \wt X(\theta_t), \He(\theta_t)^{1/2} \Delta \rangle|^3] \le C\bar K_2^3 \|\Delta\|_{\He(\theta_t)}^3
\]
for some absolute constant~$C > 0$. 
Without the loss of generality we can assume that~$C \ge 1$. 
Combining the above inequalities, we observe that
\[
|\phi'''(t)| \le 2C\bar K_2^3 [\phi''(t)]^{3/2}, \quad 0 \le t[\phi''(0)]^{1/2} \le 1,
\]
where we used that~$\theta_t \in \Theta_1(\theta_*)$ is equivalent to~$t^2 \phi''(0) \le 1$.
We can now apply Proposition~\ref{prop:nesterov-second-derivative} to~$g(t) = \phi''(t)$, putting~
\[
\bar c := C\bar K_2^3 \gsim 1,
\] 
and arriving at
\begin{align*}
\frac{\phi''(0)}{(1+\bar ct\sqrt{\phi''(0)})^2} \le \phi''(t) \le \frac{\phi''(0)}{(1-\bar ct\sqrt{\phi''(0)})^2}
\end{align*}
whenever~$0 \le \bar  ct[\phi''(0)]^{1/2} \le {1}$.
Finally, since~$\phi''(t) = \|\Delta\|_{\He(\theta_t)}^2$, this results in
\begin{align*}
\frac{\He(\theta_*)}{(1 + \bar c \|\theta - \theta_*\|_{\He(\theta_*)})^2} \prccq \He(\theta) \prccq \frac{\He(\theta_*)}{(1- \bar c \|\theta - \theta_*\|_{\He(\theta_*)})^2},
\end{align*}
whenever~$\theta \in \Theta_{1/\bar c}(\theta_*)$.
In particular, for any~$\theta \in \Theta_{1/{(2\bar c)}}(\theta_*)$ we have
\begin{alignat}{3}
\tfrac{4}{9}\He(\theta_*) \prccq \He(\theta) \prccq 4\He(\theta_*).
\label{eq:hess-ave-within-half-dikin}
\end{alignat}

$\boldsymbol{2^o}$.
Next, we derive a similar approximation result for the Hessian of \textit{empirical} risk~$\He_n(\theta) := \nabla^2 L_n(\theta)$. 
This can be done by constructing an epsilon-net on $\Theta_{1/{(2\bar c)}}(\theta_*)$ with respect to the~$\|\cdot\|_{\He(\theta_*)}$-norm. 
Then, one can control the uniform deviations of $\He_n(\theta)$ from~$\He(\theta)$ for $\theta$ on the net, while approximating~$\He_n(\theta)$ for~$\theta$ outside the net, by exploiting the self-concordance of the \textit{individual} losses, and appropriately choosing the net resolution.
To this end, recall that~$\He_n(\theta)$ writes
\[
\He_n(\theta) = \frac{1}{n} \sum_{i = 1}^n \ell''(Y_i, X_i^\T\theta) X_i X_i^\T.
\]
Hence, we can relate~$\He_n(\theta)$ to~$\He_n(\theta')$ at some other point~$\theta'$ by relating~$\ell''(Y_i, X_i^\T\theta)$ to~$\ell''(Y_i, X_i^\T\theta')$.
Namely, fix arbitrary~$\theta_0 \in \Theta_{1/{(2\bar c)}}(\theta_*)$ and~$\theta_1 \in \Theta$,
and observe that, by Assumption~\ref{ass:gsc-true},~$\phi_Z(t) = \ell_Z(\theta_t)$ satisfies 
\[
|\phi_Z'''(t)| \le 2[\phi_Z''(t)]^{3/2},
\]
hence we can apply Proposition~\ref{prop:nesterov-second-derivative} to~$\phi_Z''(t)$.
For~$0 \le t[\phi_Z''(0)]^{1/2} \le 1$ this gives
\[
\frac{\phi''_Z(0)}{(1+t[\phi''_Z(0)]^{1/2})^2} \le \phi_Z''(t) \le \frac{\phi''_Z(0)}{(1-t[\phi''_Z(0)]^{1/2})^2}.
\]
cf.~\eqref{eq:individual-loss}. Recalling that~$\phi_Z''(t) = \ell''(Y,X^\T \theta_t) \cdot \lang X, \Delta \rang^2 = \lang \wt X(\theta_t), \Delta \rang^2,$
where again~$\Delta = \theta_1 - \theta_0$ but now without the constraint that~$\theta_0 = \theta_*$, we arrive at 
\[
\frac{\ell''(Y,X^\T\theta_0)}{(1+t|\lang \wt X(\theta_0), \Delta \rang|)^2} \le \ell''(Y,X^\T\theta_t) \le \frac{\ell''(Y,X^\T\theta_0)}{(1 - t|\lang \wt X(\theta_0), \Delta \rang|)^2}
\]
when~$t |\lang \wt X(\theta_0), \Delta \rang| \le 1.$
By the Cauchy-Schwarz inequality and~\eqref{eq:hess-ave-within-half-dikin}, this gives
\[
\begin{aligned}
\ell''(Y,X^\T\theta_t) 
&\ge \frac{\ell''(Y,X^\T\theta_0)}{(1+2t \|\wt X(\theta_0) \|_{\He(\theta_0)^{-1}} \| \Delta \|_{\He(\theta_*)})^2} \\
\ell''(Y,X^\T\theta_t) 
&\le \frac{\ell''(Y,X^\T\theta_0)}{(1-2t \|\wt X(\theta_0) \|_{\He(\theta_0)^{-1}} \| \Delta \|_{\He(\theta_*)})^2},
\end{aligned}
\]
where~$t \ge 0$ is such that the denominator is strictly positive. 
As a result, we have
\begin{equation}
\label{eq:invididual-loss-second-derivative-approximation}
\begin{aligned}
\ell''(Y,X^\T\theta') 
&\ge \frac{\ell''(Y,X^\T\theta)}{(1+2\|\He(\theta)^{-1/2}\wt X(\theta) \|_2 \| \theta'-\theta \|_{\He(\theta_*)})^2}, \\
\ell''(Y,X^\T\theta') 
&\le \frac{\ell''(Y,X^\T\theta)}{(1-2\|\He(\theta)^{-1/2}\wt X(\theta) \|_2 \| \theta'-\theta \|_{\He(\theta_*)})^2}
\end{aligned}
\end{equation}
for any~$\theta \in \Theta_{1/(2\bar c)}(\theta_*)$, and any~$\theta'$ for which the denominator is strictly positive.

$\boldsymbol{3^o}$.
Now, consider the smallest epsilon-net~$\cN_\veps$ for~$\Theta_{1/{(2\bar c)}}(\theta_*)$ with respect to the norm~$\|\cdot\|_{\He(\theta_*)}$, i.e., the smallest subset of~$\Theta_{1/(2\bar c)}(\theta_*)$ such that for any~$\theta \in \Theta_{1/(2\bar c)}(\theta_*)$ there exists a point~$\theta' \in \cN_{\veps}$ such that~$\|\theta' - \theta\|_{\He(\theta_*)} \le \veps$. Note that such~$\cN_\veps$ can be obtained as the affine image of the epsilon-net for the ~$\|\cdot\|_2$-ball with radius~$1/(2\bar c)$ with respect to the standard~$\|\cdot\|_2$-norm. 
Hence, we can apply the bound for covering numbers of Euclidean balls: for any~$\veps \le 1$,
\begin{equation}
\label{eq:epsilon-covering}
|\cN_\veps| \le \left( \frac{3}{2\bar c \veps} \right)^d.
\end{equation}
Consider random vectors~$\He(\theta)^{-1/2} \wt X_i(\theta)$, where~$i \in [n]$ and~$\theta \in \cN_{\veps}$ for some~$\veps$ to be defined later. 
Each of them has unit covariance matrix, and is subgaussian with~$\psi_2$-norm at most~$\bar K_2$ due to Assumption~\ref{ass:second-subg++}. 
Repeating the argument from part~$\boldsymbol{1^o}$ of the proof of Theorem~\ref{th:crb-fake-complex} (to account for the fact that the vectors are not centered), we can show that with probability at least~$1-\delta$,
\[
\|\He(\theta)^{-1/2} \wt X_i(\theta)\|_{2} \le C_2 \bar K_2 \sqrt{d \log\left({e}/{\delta}\right)}
\]
for some constant~$C_2 \ge 1$.
Here we used that~$\cN_{\veps} \subset \Theta_{1/{2\bar c}}(\theta_*) \subseteq \Theta_1(\theta_*)$. Thus,
\begin{equation}
\label{eq:unif-control-on-the-net}
\begin{aligned}
\sup_{i \in [n], \; \theta_0 \in \cN_{\veps}} \|\He(\theta)^{-1/2} \wt X_i(\theta)\|_{2} 
&\le C_2 \bar K_2 \sqrt{d\log\left(\frac{en|\cN_{\veps}|}{\delta}\right)} 
\le C_2 \bar K_2 d\sqrt{\log\left(\frac{3en}{\delta \veps}\right)},
\end{aligned}
\end{equation}
with probability~$\ge 1-\delta$, where in the second step we used~\eqref{eq:epsilon-covering}. 
Now, we choose
\begin{equation}
\label{eq:choice-of-epsilon}
\veps = \frac{1}{64C_2^2 \bar K_2^2 d^2 \log\left({en}/{\delta}\right)}.
\end{equation}
By some simple algebra, such choice of~$\veps$ ensures that
\[
\veps \sqrt{\log\left(\frac{3en}{\delta \veps}\right)} \le \frac{1}{4C_2 \bar K_2 d}.
\]
Combining this with~\eqref{eq:invididual-loss-second-derivative-approximation} and~\eqref{eq:unif-control-on-the-net}, we see that the following is true with probability $\ge 1-\delta$:
for any~$\theta' \in \Theta_{1/{(2\bar c)}}(\theta_*)$, there exists~$\theta \in \cN_\veps$ such that
\[
\tfrac{4}{9} \ell''(Y_i,X_i^\T\theta) \le \ell''(Y_i,X_i^\T\theta') \le 4\ell''(Y_i,X_i^\T\theta), \quad i \in [n].
\]
This implies that with probability~$\ge 1-\delta$, it holds
\begin{equation}
\label{eq:net-approximation}
\tfrac{4}{9} \He_n(\pi_*(\theta)) \prccq \He_n(\theta) \prccq 4 \He_n(\pi_*(\theta)), \quad \forall \theta \in \Theta_{1/(2\bar c)}(\theta_*),
\end{equation}
where~$\pi_*(\cdot)$ is the operation of~$\|\cdot\|_{\He(\theta_*)}$-projection on the epsilon-net~$\cN_\veps$. 
Finally, to establish the uniform approximation of~$\He_n(\cdot)$ on~$\Theta_{1/(2\bar c)}(\theta_*)$, it remains to control~$\He_n(\theta)$ on the net itself. 
This can be done by combining the deviation bounds for sample covariance matrices with the results of~$\boldsymbol{1^o}$. 
First, by Theorem~\ref{th:covariance-subgaussian-simple}, for any $\theta \in \cN_{\veps}$ we have that with probability at least~$1-\delta$,
\[
\tfrac{1}{2} \He(\theta) \prccq \He_n(\theta) \prccq 2\He(\theta),
\]
provided that~$n \gsim \bar K_2^4(d + \log(1/\delta))$. Taking the union bound over~$\cN_{\veps}$, and using~\eqref{eq:epsilon-covering} and~\eqref{eq:choice-of-epsilon}, we see that
\begin{equation}
\label{eq:on-the-net}
\tfrac{1}{2} \He(\theta) \prccq \He_n(\theta) \prccq 2\He(\theta), \quad \forall \theta \in \cN_{\veps}
\end{equation}
holds with probability~$\ge 1-\delta$, provided that 
\[
\begin{aligned}
n \gsim \bar K_2^4 d \log\left(\frac{e}{\bar c \delta \veps}\right) 
\gsim \bar K_2^4 d \left[ \log \left( {e d}/{\delta} \right) + \log\log \left({en}/{\delta}\right) \right].
\end{aligned}
\]
By simple algebra, it suffices that 
\begin{equation}
\label{eq:matrix-concentration-uniform}
n \gsim \bar K_2^4 d \log \left( {e}/{\delta} \right).
\end{equation}
Combining~\eqref{eq:net-approximation},~\eqref{eq:on-the-net}, and~\eqref{eq:hess-ave-within-half-dikin}, we see that the sample size satisfying~\eqref{eq:matrix-concentration-uniform} guarantees uniform approximation of empirical Hessians on the Dikin ellipsoid~$\Theta_{1/(2\bar c)}(\theta_*)$: with probability~$\ge 1-\delta$, for any~$\theta \in \Theta_{1/(2\bar c)}(\theta_*)$ it holds
\begin{equation}
\label{eq:emp-hess-constant-within-dikin}
0.09 \, \He(\theta_*) \prccq \He_n(\theta) \prccq 32\, \He(\theta_*).
\end{equation}

$\boldsymbol{4^o}$. 
With~\eqref{eq:emp-hess-constant-within-dikin} at hand, we can localize the estimate through a similar argument to that in Proposition~\ref{prop:dikin}, but with~$S$ replaced with a constant.
Indeed, fixing~$\theta_0 = \theta_*$ and taking arbitrary~$\theta_1 \in \Theta_{1/(2\bar c)}(\theta_*)$, we see that~\eqref{eq:emp-hess-constant-within-dikin} reduces to
\[
0.09 \phi''(0) \le \phi''_n(t) \le 32 \phi''(0), \quad 0 \le t \le 1.
\]
Integrating this twice, we get
$
{0.045 \phi''(0)t^2} \le \phi_n(t) - \phi_n(0) - \phi_n'(0)t \le {16\phi''(0)t^2}.
$
Putting~$t = 1$, and noting that~$\phi''(0) = \|\theta_1 - \theta_*\|_{\He(\theta_*)}^2$, we obtain that for any $\theta \in \Theta_{1/(2\bar c)}(\theta_*)$, with high probability it holds
\begin{equation}
\label{eq:local-quadratic}
\begin{aligned}
0.045 \|\theta - \theta_*\|_{\He(\theta_*)}^2 \le L_n(\theta) - L_n(\theta_*) - \lang \nabla L_n(\theta_*), \theta - \theta_* \rang \le 16  \|\theta - \theta_*\|_{\He(\theta_*)}^2.
\end{aligned}
\end{equation}
cf.~\eqref{eq:local-mine}. 
Now we can proceed as in the proof of Proposition~\ref{prop:dikin}, Case~\ref{prop:local-mine}. 
Namely, consider the event~$\wh \theta_n \notin \Theta_{1/(2\bar c)}(\theta_*)$. Under this event, there exists~$\bar \theta_n \in [\theta_*, \wh\theta_n]$ such that~$\|\bar \theta_n - \theta_*\|_{\He(\theta_*)} = 1/2\bar c$. 
On the other hand, clearly,~$L_n(\bar\theta_n) \le L_n(\theta_*)$.
Combining these facts with~\eqref{eq:local-quadratic}, we obtain that with probability at least~$1-\delta$, 
\[
\|\nabla L_n(\theta_*)\|_{\He(\theta_*)^{-1}}^2 \gsim {1}/{\bar c^2} \gsim {1}/{\bar K_2^6}.
\]
On the other hand, we know (see part~$\boldsymbol{1^o}$ of the proof of Theorem~\ref{th:crb-fake-complex}) that 
\[
\|\nabla L_n(\theta_*)\|_{\He(\theta_*)^{-1}}^2 \lesssim \frac{K_1^2 \deff \log\left({e}/{\delta}\right)}{n}
\]
with probability~$\ge 1-\delta$. 
Thus, whenever
$
n \gsim K_1^2 \bar K_2^6 \deff \log(e/\delta),
$
we have a contradiction, so~$\wh\theta_n$ must belong to~$\Theta_{1/{(2\bar c)}}(\theta_*)$. Then,~\eqref{eq:local-quadratic} with~$\theta = \wh\theta_n$ yields
\[
\|\wh\theta_n - \theta_*\|_{\He(\theta_*)}^2 \lesssim \|\nabla L_n(\theta_*)\|_{\He(\theta_*)^{-1}}^2.
\]
It remains to bound the excess risk. To this end, recall~\eqref{eq:hess-ave-within-half-dikin} which translates to 
\[
\tfrac{4}{9} \phi''(0) \le \phi''(t) \le 4\phi''(0), \;\; 0 \le t \le 1.
\]
Integrating this twice on~$[0,1]$, we obtain
$
\tfrac{4}{9} \phi''(0)t^2 \le \phi(t) - \phi(0) \le 4\phi''(0)t^2,
$
The upper bound translates to~$L(\theta) - L(\theta_*) \le \|\theta - \theta_* \|_{\He(\theta_*)}^2$ for any~$\theta \in \Theta_{1/{2\bar c}}(\theta_*)$. But we have already proved that~$\wh\theta_n \in \Theta_{1/{2\bar c}}(\theta_*)$ with high probability.
\qed


\subsection{Proof of Theorem~\ref{th:crb-fake-improved}}
We use the same conventions as in the proof of Theorem~\ref{th:crb-mine-improved}. We assume w.l.o.g. that Assumption~\ref{ass:second-subg++} holds with~$r  = 1/\sqrt{\rho}$, and use~$\bar K_2 := \bar K_2(1/\sqrt{\rho})$ for brevity.

$\boldsymbol{1^o}$.
Our first goal is to prove that the Hessians~$\He(\theta) := \nabla^2 L(\theta)$ are close to~$\He(\theta_*)$ within the Dikin ellipsoid with radius~$1/(\bar c \sqrt{\rho})$ for some~$\bar c$ depending on constants~$K_0, \bar K_2$. 
Fix~$\theta_0 = \theta_*$ and arbitrary~$\theta_1 \in \R^d$, and let~$\theta_t = \theta_0 + t(\theta_1-\theta_0)$,~$\Delta := \theta_1 - \theta_0$.
Putting~$\wt X(\theta_t) := [\ell''(Y, \langle X, \theta_t\rangle)]^{1/2}X$ as before, we have
\[
\begin{aligned}
\phi''(t) 
&= \E[\ell''(Y, \langle X, \theta_t\rangle)  \langle X, \Delta \rangle^2] 
= \E[\langle \He(\theta_t)^{-1/2} \wt X(\theta_t), \He(\theta_t)^{1/2} \Delta \rangle^2] = \|\Delta\|_{\He(\theta_t)}^2.
\end{aligned}
\]
On the other hand, due to Assumption~\ref{ass:gsc-fake},
\[
\begin{aligned}
|\phi'''(t)| 
&\le \E[|\ell'''(Y, \langle X, \theta_t\rangle)| \cdot |\langle X, \Delta \rangle|^3] \\
&\le \E[\ell''(Y, \langle X, \theta_t\rangle) \cdot |\langle X, \Delta \rangle|^3] \\
&\le \E[\langle \wt X(\theta_t), \Delta \rangle^2 \cdot |\langle X, \Delta \rangle|] \\
&= \E[\langle \He(\theta_t)^{-1/2} \wt X(\theta_t), \He(\theta_t)^{1/2} \Delta \rangle^2 \cdot |\langle \Cov^{-1/2} X, \Cov^{1/2} \Delta \rangle|] \\
&\le \sqrt{\E[\langle \He(\theta_t)^{-1/2} \wt X(\theta_t), \He(\theta_t)^{1/2} \Delta \rangle^4]} \cdot \sqrt{\E[\langle \Cov^{-1/2} X, \Cov^{1/2} \Delta \rangle^2]},
\end{aligned}
\]
where the last step is by Cauchy-Schwarz.
Now, for~$\theta_t \in \Theta_{1/\sqrt{\rho}}(\theta_*)$, one has~$\|\He(\theta_t)^{-1/2} \wt X(\theta_t)\|_{\psi_2} \le \bar K_2$ due to Assumption~\ref{ass:second-subg++}. 
On the other hand,~$\| \Cov^{-1/2} X \|_{\psi_2} \le K_0$. 
Hence, by Lemma~\ref{lem:subg-properties} and Assumption~\ref{ass:curvature}, we have
\begin{align*}
\E[\langle \He(\theta_t)^{-1/2} \wt X(\theta_t), \He(\theta_t)^{1/2} \Delta \rangle^4] &\le C \bar K_2^4 \|\Delta\|_{\He(\theta_t)}^4,\\
\E[\langle \Cov^{-1/2} X, \Cov^{1/2} \Delta \rangle^2] \le C K_0^2 \|\Delta\|_{\Cov}^2 &\le \rho C \bar K_0^2 \|\Delta\|_{\He(\theta_*)}^2,
\end{align*}
for some constant~$C > 0$; moreover, we can safely assume that~$C > 1$ by weakening the bounds otherwise.
Combining the above results, we arrive at
\[
|\phi'''(t)| \le C K_0 \bar K_2^2 [\rho \phi''(0)]^{1/2} \phi''(t) , \quad 0 \le t[\rho \phi''(0)]^{1/2} \le 1.
\]
We now formulate a specification of Proposition~\ref{prop:nesterov-second-derivative} for the present situation.
\begin{proposition}
\label{prop:bach-second-derivative}
Assume~$g: \R \to \R$ is differentiable, non-negative, and 
\[
\begin{aligned}
&|g'(t)| \le c\sqrt{g(0)}g(t), & |t| \le T
\end{aligned}
\]
for~$c \ge 0$.
Then for~$t: |t| \le T$ one has
$g(0)e^{-c|t|\sqrt{g(0)}} \le g(t) \le g(0) e^{c|t|\sqrt{g(0)}}.$
\end{proposition}

\begin{proof}
We assume that~$g(t) > 0$ for~$t: |t| \le T$; the argument can be generalized in exactly the same way as in the proof of Proposition~\ref{prop:nesterov-second-derivative}.
Denoting~$\psi(t) = \log g(t)$, we obtain by integrating~$\psi'(t)$ that
$
-c\sqrt{g(0)}t \le \log(g(t)) - \log(g(0)) \le c\sqrt{g(0)}t,
$
Rearranging this, we arrive at the claim. 
\end{proof}


Now, putting~
\begin{equation}
\label{eq:const-fake}
\bar c := C K_0 \bar K_2^2,
\end{equation}
and applying Proposition~\ref{prop:bach-second-derivative} to~$g(t) = \phi''(t)$, under~$\bar c|t|\sqrt{\bar \rho\phi''(0)}\le{1}$ we get
\begin{align*}
\phi''(0) e^{-\bar c|t|\sqrt{\rho \phi''(0)}} \le \phi''(t) \le \phi''(0) e^{\bar c|t|\sqrt{\rho \phi''(0)}}.
\end{align*}
Finally, since~$\phi''(t) = \|\Delta\|_{\He(\theta_t)}^2$, this translates to the analogue of~\eqref{eq:hess-ave-within-half-dikin}:
\begin{align}
\frac{1}{e}\He(\theta_*) \prccq \He(\theta) \prccq e\He(\theta_*), \quad \theta \in \Theta_{\bar r}(\theta_*), \quad \bar r := \frac{1}{\bar c \sqrt{\bar \rho}}.
\label{eq:hess-ave-within-half-dikin-fake}
\end{align}
Here we used that~$\Theta_{\bar r}(\theta_*) \subseteq \Theta_{1/\sqrt{\rho}}(\theta_*)$ since~$\bar c \ge 1$.

$\boldsymbol{2^o}$.
We now provide a local approximation of~$\He_n(\theta)$ using pseudo self-concordance of individual losses. 
Fix~$\theta_0 \in \Theta_{\bar r}(\theta_*)$ and~$\theta_1 \in \Theta$, and note that
\begin{align*}
|\phi'''_Z(t)| 
&= |\ell'''(Y,X^\T \theta_t) \cdot \lang X, \Delta \rang|^3 \\
&\le |\ell'''(Y,X^\T \theta_t) \cdot \lang X, \Delta \rang|^3 = \lang \wt X(\theta_t), \Delta \rang^2 \cdot |\lang X, \Delta \rang| = \phi''_Z(t) \cdot |\lang X, \Delta \rang|.
\end{align*}
By the argument analogous to those in Propositions~\ref{prop:nesterov-second-derivative} and~\ref{prop:bach-second-derivative}, we obtain
\[
\phi_Z''(0) e^{-t|\lang X, \Delta \rang|} \le \phi_Z''(t) \le \phi_Z''(0) e^{t |\lang X, \Delta \rang|},
\]
which translates to
$
\ell''(Y,X^\T\theta_0) e^{-t|\lang X, \Delta \rang|} \le \ell''(Y,X^\T\theta_t) \le \ell''(Y,X^\T\theta_0) e^{t |\lang X, \Delta \rang|}.
$
Thus, denoting~$\He := \He(\theta_*)$ for brevity, we have
\[
\ell''(Y,X^\T\theta_0) e^{-t\|X\|_{\He^{-1}} \| \Delta \|_{\He}} \le \ell''(Y,X^\T\theta_t) \le \ell''(Y,X^\T\theta_0) e^{t \|X\|_{\He^{-1}} \| \Delta \|_{\He}}.
\]
Equivalently, for any~$\theta \in \Theta_{\bar r}(\theta_*)$ and~$\theta' \in \Theta$,
\begin{align}
\label{eq:invididual-loss-second-derivative-approximation-fake}
\ell''(Y,X^\T\theta_0) e^{-\|X\|_{\He^{-1}} \| \theta' - \theta  \|_{\He}} \le \ell''(Y,X^\T\theta_t) \le \ell''(Y,X^\T\theta_0) e^{\|X\|_{\He^{-1}} \| \theta' - \theta  \|_{\He}}.
\end{align}
By Assumption~\ref{ass:design-subg}, random vector~$\Cov^{-1/2}X$ has~$\psi_2$-norm at most~$\bar K_0$.
Hence, repeating the argument from~$\boldsymbol{1^o}$ in the proof of Theorem~\ref{th:crb-fake-complex} we can show that, for some constant~$C_0$, with probability at least~$1-\delta$ one has 
\begin{equation}
\label{eq:local-radius-fake}
\max_{i \in [n]}\|X_i\|_{\He^{-1}} \le C_0 K_0 \sqrt{\rho d \log\left(\frac{en}{\delta}\right)}.
\end{equation}

$\boldsymbol{3^o}$.
Let~$\cN_\veps$ be the epsilon-net on~$\Theta_{\bar r}(\theta_*)$, with respect to the norm~$\|\cdot\|_{\He}$, with
\begin{equation}
\label{eq:choice-of-epsilon-fake}
\veps = \frac{1}{C_0 K_0 \sqrt{\rho d \log\left({en}/{\delta}\right)}}.
\end{equation}
Combining this with~\eqref{eq:invididual-loss-second-derivative-approximation-fake} and~\eqref{eq:local-radius-fake}, we obtain that with probability at most~$1-\delta$,
\begin{equation}
\label{eq:net-approximation-fake}
\tfrac{1}{e} \He_n(\pi(\theta)) \prccq \He_n(\theta) \prccq e\He_n(\pi(\theta)), \quad \forall \theta \in \Theta_{\bar r}(\theta_*),
\end{equation}
where~$\pi(\cdot)$ is the projection operator on the net~$\cN_\veps$. 
On the other hand,  by Theorem~\ref{th:covariance-subgaussian-simple}, it holds that
\begin{equation}
\label{eq:uniform-on-the-net-fake}
\tfrac{1}{2} \He(\theta) \le \He_n(\theta) \le 2\He(\theta), \quad \forall \theta \in \cN_{\veps}
\end{equation}
with probability at least~$1-\delta$, whenever~$n \gsim d +\log\left({|\cN_\veps|}/{\delta}\right)$. 
Recaling that~$|\cN_\veps| \le (3\bar r/\veps)^d$, it is sufficient that
\[
n \gsim d\log\left(\frac{e\bar r}{\veps \delta}\right) \gsim d\log\left(\frac{e K_0 \sqrt{d \log(en/\delta)}}{\bar c \delta}\right) \gsim d\log\left(\frac{e \sqrt{d \log(en/\delta)}}{\bar K_2^2\delta}\right),
\]
where we used~\eqref{eq:const-fake} and~\eqref{eq:choice-of-epsilon-fake}. Noting that~$\bar K_2 \ge 1$, by simple algebra we have that~\eqref{eq:uniform-on-the-net-fake} holds with probability at least~$1-\delta$ whenever
\[
n \gsim d \log\left({ed}/{\delta}\right).
\]
Finally, if this is the case, with probability at least~$1-\delta$ it holds
\[
\tfrac{e^2}{2}\He(\theta_*) \prccq \He_n(\theta) \prccq 2e^2 \He(\theta_*), \quad \forall \theta \in \Theta_{\bar r}(\theta_*),
\]
where we combined~\eqref{eq:uniform-on-the-net-fake} with~\eqref{eq:net-approximation-fake} and~\eqref{eq:hess-ave-within-half-dikin-fake}. 

$\boldsymbol{4^o}$.
As the empirical Hessians are uniformly approximated by~$\He(\theta_*)$ in the Dikin ellipsoid with radius~$\bar r = 1/(CK_0 \bar K_2^2\sqrt{\rho})$, we can proceed in the same way as in step~$\boldsymbol{4^o}$ in the proof of Theorem~\ref{th:crb-mine-improved}, showing that~\eqref{eq:excess-risk-bound-neat} holds whenever~$\|\nabla L_n(\theta_*)\|_{\He^{-1}}^2 \lsim 1/(\rho \bar c^2) \lsim 1/(\rho K_0^2 \bar K_2^4),$ cf.~\eqref{eq:const-fake}.
This leads to the second bound on the critical sample size from the premise of the theorem. 
\qed

\subsection{Proof of Theorem~\ref{th:crb-sparse-fake}}

$\boldsymbol{0^o}$.
First, we follow the standard idea in the analysis of $\ell_1$-penalized estimators (see, e.g.,~\cite{belloni2011square}): using the convexity of $L_n(\theta)$, we show that whenever~$\lambda$ dominates~$\nabla L_n(\theta)$ -- which is in fact enforced by the lower bound in~\eqref{eq:sparse-condition-lambda} -- the essential part of the residual~$\Delta := \wh\theta_{\lambda,n} - \theta_*$
with high probability concentrates on the support~$\cS$.
Indeed, due to the optimality of~$\wh \theta := \wh\theta_{\lambda,n}$, we have
\begin{equation}
\label{eq:oracle-feasibility}
L_n(\wh\theta) - L_n(\theta_*) \le \lambda (\|\theta_*\|_1  - \|\wh\theta\|_1).
\end{equation}
Let~$\Delta_\cS$ be the orthogonal projection of~$\Delta$ onto~$\cS$, and denote~$\Delta_{\cSc} = \Delta - \Delta_{\cS} = \wh\theta_{\cS}$ its projection onto~$\cSc$, the orthogonal complement of~$\cS$. 
By the triangle inequality,
\begin{equation}
\label{eq:lasso-triangular}
\|\theta_*\|_1  - \|\wh\theta\|_1 \le \|\Delta_\cS\|_1 - \|\Delta_{\cSc}\|_1. 
\end{equation}
On the other hand, by convexity of~$L_n(\theta)$, we have
\begin{equation}
\label{eq:kkt}
L_n(\wh\theta) - L_n(\theta_*) 
\ge -\|\nabla L_n(\theta_*)\|_{\infty} \|\wh\theta - \theta_*\|_1 \ge -\|\nabla L_n(\theta_*)\|_{\infty} (\|\Delta_{\cS}\|_1 +  \|\Delta_{\cSc}\|_1).
\end{equation}
Collecting~\eqref{eq:oracle-feasibility}--\eqref{eq:kkt}, we get
\[
\left({\lambda - \| \nabla L_n(\theta_*)\|_{\infty}}\right)  \|\Delta_{\cSc}\|_1 \le \left({\lambda + \| \nabla L_n(\theta_*)\|_{\infty}}\right)  \|\Delta_{\cS}\|_1.
\]
Whence if 
\begin{equation}
\label{eq:gradient-dominated}
\lambda \ge 2 \| \nabla L_n(\theta_*)\|_{\infty}, 
\end{equation}
we have that~$\Delta$ satisfies the restricted subspace condition:
\begin{equation}
\label{eq:compatibility}
\|\Delta_{\cSc}\|_1 \le 3\|\Delta_{\cS}\|_1,
\end{equation}
combining which with~$\|\Delta_{\cS}\|_1 \le \sqrt{s} \|\Delta_{\cS}\|_2 \le \sqrt{s} \|\Delta\|_2$
results in
\begin{equation}
\label{eq:l1tol2}
\|\Delta\|_1 \le 4\sqrt{s} \|\Delta\|_2.
\end{equation}
Later on, we will show that the lower bound in~\eqref{eq:sparse-condition-lambda} implies~\eqref{eq:gradient-dominated} with probability at least~$1-\delta$. For now, let us assume that~\eqref{eq:gradient-dominated} holds.

$\boldsymbol{1^o}$.
To localize the estimate, we now use a similar technique to the one used in the proof of Proposition~\ref{prop:dikin}, but replace the Cauchy-Schwarz inequality with Young's inequality. 
First, applying~\eqref{eq:local-lower-fake} to~$L_n(\theta)$ with~$\theta_0 = \theta_*$,~$\theta_1 = \wh\theta$, and~$W = X_j$ for some (random)~$j \in [n]$, we have
\[
\frac{e^{-|\lang X_j, \Delta \rang|}-1+|\lang X_j, \Delta \rang|}{|\lang X_j, \Delta \rang|^2} \|\Delta\|_{\He_n}^2 \le L_n(\wh\theta_{}) - L_n(\theta_*) - \lang \nabla L_n(\theta_*), \Delta \rang,
\]
Since function~$u \mapsto ({e^{-u}-1+u})/{u^2}$ is non-increasing, we can replace~$|\lang X_j, \Delta \rang|$ with~$\|X_j\|_{\infty} \|\Delta\|_1$.
Combining this with~\eqref{eq:oracle-feasibility} and~\eqref{eq:lasso-triangular}, bounding~$-\lang \nabla L_n(\theta_*), \Delta \rang$ via Young's inequality, and using~\eqref{eq:gradient-dominated}, we get
\begin{equation}
\label{eq:sparse-dikin}
\frac{e^{-\|X_j\|\|\Delta\|_{1}} - 1 + \|X_j\|_{\infty} \|\Delta\|_{1}}{\|X_j\|_{\infty}^2 \|\Delta\|_{1}^2} \|\Delta\|_{\He_n}^2
\le \frac{3\lambda\|\Delta\|_1}{2}.
\end{equation}
We now use the standard result from compressed sensing theory (see Theorem~\ref{th:covariance-sparse} in Appendix) which states the following.
Suppose that all~$\s$-restricted eigenvalues of~$\He$ belong to~$[1/\rho, \k_2]$, meaning that
\[
{\|\Delta\|^2}/{\rho} \le \| \Delta \|_{\He}^2 \le \k \|\Delta\|^2
\] 
for any~$\Delta$ satisfying the restricted  subspace property~\eqref{eq:compatibility} -- which is clearly the case for~$\He$ in question, due to Assumptions~\ref{ass:curvature} and~\ref{ass:lipshitz}.
Then, the corresponding sample covariance matrix~$\He_n$  with probability at least~$1-\delta$ satisfies
\begin{equation}
\label{eq:cov-matrix-estimated-sparse}
\tfrac{1}{2} \|\Delta\|_{\He}^2 \prccq \|\Delta\|_{\He_n}^2 \prccq 2  \|\Delta\|_{\He}^2,
\end{equation}
for any~$\Delta$ satisfying~\eqref{eq:compatibility}, provided that
\[
n \gtrsim \rho \k_2 K_2^4 \s \log\left({ed}/{\delta}\right),
\]
cf.~\eqref{eq:sparse-condtion-hessian}.
Combining this result with
\begin{equation*}
\|\Delta\|_{\He}^2 \ge \frac{\|\Delta\|_{2}^2}{\rho} \ge \frac{\|\Delta\|_{1}^2}{16 \rho \s},
\end{equation*}
where we used~\eqref{eq:l1tol2}, we have that under~\eqref{eq:sparse-condtion-hessian} with probability~$1 - \delta$ it holds
\begin{equation}
\label{eq:sparse-hessians}
\|\Delta\|_{\He_n}^2 \ge \frac{\|\Delta\|_{1}^2}{32 \rho \s}.
\end{equation}
Combining this with~\eqref{eq:sparse-dikin}, and denoting
\[
\B_{\sup} := \max_{i \in [n]} \|X_i\|_{\infty}, \quad u := \B_{\sup} \|\Delta\|_1,
\]
we obtain
$
e^{-u} - 1 + u \le 48\rho \s \lambda \B_{\sup} u.
$ 
From now on, we proceed as in the proof of Proposition~\ref{prop:dikin}, cf.~\eqref{eq:bach-decrement-equation}. That is, under
\begin{equation}
\label{eq:sparse-decrement-small}
48 \rho \s \lambda \B_{\sup} \le 1/2,
\end{equation}
we sequentially obtain~$u \le 2$,~$e^{-u} - 1 + u \ge \frac{u^2}{4}$, then~$u \le 192\rho \s  \B_{\sup} \lambda$, and
\[
\|\Delta\|_1 \le 192 \rho \s \lambda.
\]
This is the first inequality in~\eqref{eq:sparse-errors}, and the second one is obtained by combining it with~\eqref{eq:sparse-dikin}--\eqref{eq:cov-matrix-estimated-sparse}.
Thus, both inequalities in~\eqref{eq:sparse-errors} are satisfied under the two assumed conditions~\eqref{eq:gradient-dominated} and~\eqref{eq:sparse-decrement-small}. 
It remains to show that these conditions are indeed guatanteed to be satisfied with high probability under~\eqref{eq:sparse-condition-lambda}. For that, we have to bound the quantities~$\|\nabla L_n(\theta_*)\|_{\infty}$ and~$\B_{\sup}$ from above. 
Indeed, due to Assumption~\ref{ass:first-subg}, we have 
\[
\| \nabla \ell_Z(\theta_*) \|_{\psi_2} \le K_1\sqrt{\k_1}.
\]
By Lemma~\ref{lem:subg-sum}, this gives~$\|\nabla L_n(\theta_*)\|_{\psi_2} \lsim K_1  \sqrt{{\k_1}/{n}}.$ 
Whence,~$\|[\nabla L_n(\theta_*)]_i\|_{\psi_2} \lsim K_1\sqrt{\k_1/n}$ componentwise for any~$i \in [n]$. 
Whence, by Lemma~\ref{lem:subg-sup}, one has 
\[
\|\nabla L_n(\theta_*)\|_{\infty} \lsim K_1 \sqrt{\frac{\k_1 \log\left({ed}/{\delta}\right)}{n}}
\]
with probability at least~$1-\delta$.
This guarantees~\eqref{eq:gradient-dominated} under the lower bound in~\eqref{eq:sparse-condition-lambda}.
Similarly, we can show that with probability at least~$1-\delta$,
\[
\B_{\sup} \lsim K_0 \sqrt{\log\left({e d n}/{\delta}\right)},
\]
which guarantees~\eqref{eq:sparse-decrement-small} under the upper bound in~\eqref{eq:sparse-condition-lambda}. The first claim of the theorem is proved; note that the upper bound in~\eqref{eq:sparse-condtion-hessian} is a corollary of~\eqref{eq:sparse-condition-lambda}. 

$\boldsymbol{2^o}$. To prove the second claim, we bound the excess risk using a similar technique as in the proof of Theorem~\ref{th:crb-fake-complex}. 
Note that~$\mathds{P}(\cE) \ge 1-\delta$ by the results of~$\boldsymbol{1^o}$. 
As in the proof of Theorem~\ref{th:crb-fake-complex}, let~$\He_{\cE} := \nabla^2 L_{\cE}(\theta_*)$; recall that~$\He_{\cE} \prccq \He$.
Applying~\eqref{eq:local-upper-fake} to~$L_{\cE}(\theta)$ with~$S \le \|X\|_{\infty} \|\Delta\|_1$ (recall that~$X \in \cE$), we have
\[
\begin{aligned}
L_{\cE}(\wh\theta) - L_{\cE}(\theta_*) 
&\le \|\nabla L_{\cE}(\theta_*) \|_{\infty}  \|\Delta\|_1 + \frac{e^{\|X\|_{\infty} \|\Delta\|_1} -1 - \|X\|_{\infty} \|\Delta\|_1}{\|X\|_{\infty}^2 \|\Delta\|_1^2}  \|\Delta\|_{\He}^2 \\
&\lsim \|\nabla L_{\cE}(\theta_*) \|_{\infty} \|\Delta\|_1 + \|\Delta\|_{\He}^2,
\end{aligned}
\]
where we bounded the factor ahead of~$\|\Delta\|_{\He}^2$ by a constant using the results of~$\boldsymbol{1^o}$. 
Now, define~$L_{\cE_c}(\theta) := \E[\ell_Z(\theta)\ind_{\cE_c}(X)]$ where~$\cEc$ is the complimentary event to~$\cE$.
Since $\nabla L(\theta_*) = 0$, we have $\nabla L_{\cE}(\theta_*) = \nabla L_{\cE_c}(\theta_*)$.
On the other hand, for any~$p, q \ge 1$ such that $1/p + 1/q = 1$, we have
\begin{equation*}
\begin{aligned}
\|\nabla L_{\cE_c}(\theta_*)\|_{\infty} 
&\le \E[\|\nabla \ell_{Z}(\theta_*)\|_{\infty}\ind_{\cE_c}(X)]  
\le \E[\|\nabla \ell_{Z}(\theta_*)\|_{\infty}^{p}]^{\frac{1}{p}} \delta^{\frac{1}{q}}
\le K_1 \sqrt{p\k_1} \, d^{\frac{1}{p}} \delta^{\frac{1}{q}}. 
\end{aligned}
\end{equation*}
where we applied H\"older's and Young's inequalities, and then Lemma~\ref{lem:subg-sup-moments}. 
Recall that in~$\boldsymbol{1^o}$ we obtained that~$\|\Delta\|_1 \lsim \rho \s \lambda$ and~$\|\Delta\|_{\He}^2 \lsim \rho \s \lambda^2$ with probability at least~$1-\delta$.
Combining these observations, we arrive at
\[
L_{\cE}(\wh\theta) - L_{\cE}(\theta_*) \le (\lambda +  K_1 \sqrt{p\k_1} \, d^{1/p} \delta^{1/q}) \rho \s \lambda.
\]
Choosing~$p = \log(ed)$, so that $q = \log(ed)/\log(d)$, we arrive at the claim. 
\qed

\subsection{Proof of Theorem~\ref{th:crb-sparse-true}}

$\boldsymbol{1^o}.$
Let $\wh\theta = \wh\theta_{\lambda,n}$ for brevity.
The step~$\boldsymbol{0^o}$ of the proof of Theorem~\ref{th:crb-sparse-fake} can be repeated \textit{verbatim}. As a result, whenever
\begin{equation}
\label{eq:gradient-dominated-true}
\lambda \ge 2 \| \nabla L_n(\theta_*)\|_{\infty},
\end{equation}
we have
\begin{equation}
\label{eq:lasso-optimality-true}
L_n(\wh\theta) - L(\theta_*) \le \lambda (\|\Delta_{\cS}\|_1 - \|\Delta_{\cSc}\|_1) \le \lambda \|\Delta\|_1,
\end{equation}
\begin{equation}
\label{eq:compatibility-true}
\|\Delta_{\cSc}\|_1 \le 3\|\Delta_{\cS}\|_1,
\end{equation}
\begin{equation}
\label{eq:l1tol2-true}
\|\Delta\|_1 \le 4\sqrt{s} \|\Delta\|_2.
\end{equation}
Moreover, we know (cf. the end of step~$\boldsymbol{1^o}$ of the proof of Theorem~\ref{th:crb-sparse-fake}) that~\eqref{eq:gradient-dominated-true} holds with probability at least~$1-\delta$ as long as 
\begin{equation}
\label{eq:first-threshold-true}
\|\nabla L_n(\theta_*)\|_{\infty} \lsim K_1 \sqrt{\frac{\k_1 \log\left({ed}/{\delta}\right)}{n}}.
\end{equation}
Hence,~\eqref{eq:gradient-dominated-true} and~\eqref{eq:first-threshold-true} are satisfied under the lower bound in~\eqref{eq:sparse-condition-lambda-true}.
Finally, under~\eqref{eq:compatibility-true} we have 
\begin{equation}
\label{eq:cov-matrix-estimated-sparse-true}
\tfrac{1}{2} \|\Delta\|_{\He}^2 \prccq \|\Delta\|_{\He_n}^2 \prccq 2  \|\Delta\|_{\He}^2
\end{equation}
and
\begin{equation}
\label{eq:sparse-hessians-true}
\|\Delta\|_{\He_n}^2 \ge \frac{\|\Delta\|_{1}^2}{32 \rho \s},
\end{equation}
both with probability at least~$1-\delta$, whenever
$
n \gtrsim \rho \k_2 K_2^4 \s \log\left({ed}/{\delta}\right).
$

$\boldsymbol{2^o}.$
However,~\eqref{eq:sparse-dikin} does not hold since we cannot use~\eqref{eq:local-lower-fake}. 
Instead, we prove 
\begin{equation}
\label{eq:sparse-dikin-true-main}
\frac{\|\Delta\|_{\He_n}^2}{1+3\|\wt X_j\|_{\infty} \|\Delta\|_1} \le L_n(\wh\theta) - L(\theta_*) - \lang \nabla L_n(\theta_*), \Delta \rang,
\end{equation}
where~$j \in \Argmax_{i \in [n]} |\lang \wt X_i, \Delta\rang|$.
Indeed, to this end denote~$S = |\lang \wt X_j, \Delta\rang|$. Whenever~$S < 1$, function~$L_n(\theta)$ satisfies the second statement of Case~\ref{prop:local-mine} of Proposition~\ref{prop:local}, and we obtain~\eqref{eq:sparse-dikin-true-main} from the lower bound in~\eqref{eq:local-aux}. On the other hand, when~$S \ge 1$ function~$L_n(\theta)$ satisfies the basic statement of Case~\ref{prop:local-mine} of Proposition~\ref{prop:local}, and we can use the lower bound in~\eqref{eq:local-mine}, i.e.,
\begin{equation}
\label{eq:sparse-dikin-true-aux}
\tfrac{1}{3S^2} {\|\Delta\|_{\He_n}^2} \le L_n(\theta_{1/S}) - L(\theta_*) - \tfrac{1}{S} \lang \nabla L_n(\theta_*), \Delta \rang,
\end{equation}
where~$\theta_{1/S}$ is the convex combination of~$\theta_*$ and~$\wh\theta$ given by
\[
\theta_{1/S} = \left(1-{1}/{S}\right) \cdot \theta_* + {1}/{S} \cdot \wh\theta.
\]
By convexity, we have~$L_n(\theta_{1/S}) \le (1-\frac{1}{S}) L_n(\theta_*) + \frac{1}{S} L_n(\wh\theta)$, 
whence~$L_n(\wh\theta) - L_n(\theta_*) \le (L_n(\wh\theta) - L_n(\theta_*))/S$. When combined with~\eqref{eq:sparse-dikin-true-aux}, this results in
\[
\tfrac{1}{3S} \|\Delta\|_{\He_n}^2 \le L_n(\wh\theta) - L(\theta_*) - \lang \nabla L_n(\theta_*), \Delta \rang.
\]
Whence~\eqref{eq:sparse-dikin-true-main} follows by Young's inequality. Now,~\eqref{eq:sparse-dikin-true-main},~\eqref{eq:lasso-optimality-true}, and~\eqref{eq:gradient-dominated-true} imply
\begin{equation}
\label{eq:sparse-dikin-final}
\frac{\|\Delta\|_{\He_n}^2}{1+3\|\wt X_j\|_{\infty} \|\Delta\|_1} \le \frac{3\lambda \|\Delta\|_1}{2},
\end{equation}
which is an analogue of~\eqref{eq:sparse-dikin}. Starting from this point, we can proceed in a similar way as in the proof of Theorem~\ref{th:crb-sparse-true}. Namely, let~$\wt\B_{\sup} := \|\wt X\|_{\infty}$ and~$u := \wt\B_{\sup} \|\Delta\|_1$, then~\eqref{eq:sparse-dikin-final} and~\eqref{eq:sparse-hessians-true} imply 
\[
\frac{u}{1+3u} \le 48\rho\s\lambda\wt\B_{\sup}.
\]
Hence, whenever
\begin{equation}
\label{eq:sparse-second-condition-true}
48\rho\s\lambda\wt\B_{\sup} \le 1/4,
\end{equation}
we have~$u \le 1$ and~${u}/{(1+3u)} \ge u/4$, which implies~${u} \le 192\rho\s\lambda\wt\B_{\sup}$ and~$\|\Delta\|_1 \le 192 \rho \s \lambda.$
This is the first inequality in~\eqref{eq:sparse-errors-true}. To obtain the second inequality, we combine~\eqref{eq:sparse-dikin-final} and~\eqref{eq:cov-matrix-estimated-sparse-true}. 
Thus, for~\eqref{eq:sparse-errors-true} it remains to show that~\eqref{eq:sparse-second-condition-true} holds under the upper bound in~\eqref{eq:sparse-condition-lambda-true}. 
We have
$
\|\wt X\|_{\psi_2} \le \|\He^{1/2}\|_2 \|\He^{-1/2}\wt X\|_{\psi_2} \le K_2 \sqrt{\k_2},
$
where we used Assumptions~\ref{ass:second-subg} and~\ref{ass:lipshitz}. This leads to
$
\wt\B_{\sup} \lsim K_2 \sqrt{\k_2 \log(edn/\delta)}
$ 
with probability~$1-\delta$, which guarantees~\eqref{eq:sparse-second-condition-true} under the upper bound in~\eqref{eq:sparse-condition-lambda-true}. 

$\boldsymbol{2^o}.$ 
We now adapt the proof of the second claim of Theorem~\ref{th:crb-sparse-fake}. Recall that in our case~$\cE := \{\|\wt X\|_{\infty} \lsim K_2 \sqrt{\k_2 \log \left({ed}/{\delta}\right)}\}$, and~$\mathds{P}(\cE) \ge 1-\delta$ by the results of~$\boldsymbol{1^o}$. 
As before, we put~$\He_{\cE} := \nabla^2 L_{\cE}(\theta_*) \prccq \He$, but this time we note that~$L_{\cE}(\theta)$ satisfies Case~\ref{prop:local-mine} of Proposition~\ref{prop:local} with~$S \le \|\wt X\|_{\infty} \|\Delta\|_1 < 1$, cf.~$\boldsymbol{1^o}$. Thus, by the upper bound in~\eqref{eq:local-aux} we have
\[
L_{\cE}(\wh\theta) - L_{\cE}(\theta_*) \lsim \|\nabla L_{\cE}(\theta_*) \|_{\infty} \|\Delta\|_1 + \|\Delta\|_{\He}^2.
\]
Thence we proceed as in the proof of the second claim of Theorem~\ref{th:crb-sparse-fake}.
\qed
\section{Logistic regression with Gaussian design}
\label{sec:subgaussian-check}

\paragraph{Change of variables.}
Consider a canonical GLM~\eqref{def:glm} with cumulant~$a(\eta)$.
Here,~$\ell''(y,\eta) = a''(\eta)$ does not depend on~$y$, hence~$\wt X(\theta) = [a''(X^\T\theta)]^{1/2}X$ is fully defined by the distribution of~$X$ and the value of~$\theta$. 
Hence, the validity of Assumptions~\ref{ass:curvature},~\ref{ass:second-subg},~\ref{ass:second-subg++} only depends on the distribution of~$X$, the expression for~$a''(\eta)$, and, possibly, the value of~$\theta_*$ (or~$\theta$ in the unit Dikin ellipsoid of~$\theta_*$ in the case of Assumption~\ref{ass:second-subg++}). Note, however, that the distribution of~$Y$ does influence Assumption~\ref{ass:first-subg} since the loss gradient~$\ell'(Y,X^\T\theta)X = (a'(X^\T \theta) - Y)X$ contains~$Y$. 
Now, consider the case of \emph{zero-mean design}, which only makes sence when~$\eta$ is unrestricted, i.e.,~$\Rp = \R$ (note that this excludes the exponential responce model).
In this case, it is natural to pass from~$X$ and~$\theta$ to the decorrelated design~$Z:=\Cov^{-1/2}X$ and parameter~$\vtheta := \Cov^{-1/2}\theta$.
Indeed,~$X^\T \theta = Z^\T \vtheta$, and the corresponding vector~$\wt Z(\vtheta)$,
\[
\wt Z(\vtheta) := [a''(Z^\T \vtheta)]^{1/2}Z,
\] 
writes~$\wt Z(\vtheta)  = \Cov^{-1/2} \wt X(\theta),$ so that its 2nd-moment matrix
$
\Psi(\vtheta) := \E[\wt Z(\vtheta) \wt Z(\vtheta)^\T]
$
is given by~$\Psi(\vtheta) = \Cov^{-1/2} \He(\theta) \Cov^{-1/2}.$
Verifying Assumption~\ref{ass:curvature} thus reduces to bounding the lowest eigenvalue of~$\Psi(\vtheta_*)$ at~$\vtheta_* := \Cov^{1/2} \theta_*$, while Assumptions~\ref{ass:second-subg} and~\ref{ass:second-subg++} reduce to checking~$\|\Psi(\vtheta)^{-1/2} \wt Z(\vtheta)\|_{\psi_2} \lsim K_2$ in the neighborhood of~$\vtheta_*$.
Similarly, Assumption~\ref{ass:first-subg} can be reformulated in terms of the variables~$Z,\vtheta,Y$.\\ 

Here we consider the case of logistic regression with zero-mean Gaussian design (with arbitrary covariance), verifying the assumptions presented in Section~\ref{sec:ass-distrib}.

\begin{proposition}
\label{prop:logistic-case-study}
In logistic regression with~$X \sim \cN(0,\Cov)$, the following holds:
\begin{enumerate}
\item Assumption~\ref{ass:curvature} holds with
\[
\rho \lsim 1 + \|\theta_*\|_{\Cov}^3.
\]
\item Assumption~\ref{ass:second-subg} holds with
$
K_2 \lsim \left(1+\log(1+\|\theta_*\|_{\Cov})\right)\sqrt{1 + \|\theta_*\|_{\Cov}}.
$\\
Moreover, Assumption~\ref{ass:second-subg++} with radius~$r$ of the Dikin ellipsoid holds with
\[
\bar K_2(r) 
\lsim \left(1 + \log(1+\|\theta_*\|_{\Cov} + r\sqrt{\rho})\right)\sqrt{1 + \|\theta_*\|_{\Cov} + r\sqrt{\rho}}.
\]
That is,~$\bar K_2(1/\sqrt{\rho})$ admits the same bound as~$K_2$ up to a constant factor.
\item If the model is well-specified, Assumption~\ref{ass:first-subg} holds with
\[
K_1 \lsim \sqrt{\rho} \lsim (1 + \|\theta_*\|_{\Cov})^{3/2}.
\]
Moreover, for \emph{subexponential} norm~$\|\cdot\|_{\psi_1}$, see~\cite[Sec.~5.2.4]{vershynin2010introduction}, one has
\begin{equation*}
\|\Gr(\theta_*)^{-1/2} \ell'(Y,X^\T\theta_*)X \|_{\psi_1} \lsim \log(1+\|\theta_*\|_{\Cov})^2 \sqrt{1+\|\theta_*\|_{\Cov}};
\end{equation*}
equivalently,
$
\big(\E[\lang \Gr(\theta_*)^{-\frac{1}{2}} \ell'(Y,X^\T\theta_*)X, u \rang^p  \big)^{\frac{1}{p}} \lsim Kp
$
for all~$u \in \cS^{d-1}$ with
\[
K = \log(1+\|\theta_*\|_{\Cov})^2 \sqrt{1+\|\theta_*\|_{\Cov}}.
\]
\end{enumerate}
\end{proposition}

\begin{proof}
Note that~$Z \sim \cN(0, \Id_d)$, and since this law is rotation-invariant, we can w.l.o.g. assume that the first coordinate vector is parallel to~$\vtheta$.
Using the symmetries of~$\cN(0,1)$, we can make sure that~$\Psi(\vtheta) = \Cov^{-1/2} \He(\theta) \Cov^{-1/2}$ writes
\begin{equation}
\label{eq:spiked-covariance}
\Psi(\vtheta) = 
\begin{bmatrix}
\kappa & 0_{d-1}^\T \\
0_{d-1} & \kappa_{\perp}\Id_{d-1},\
\end{bmatrix},
\end{equation}
where~$0_{d-1}$ is the zero column, and~$\kappa, \kappa_{\perp}$ are given in terms of the standard Gaussian density~$\phi(\cdot)$ and
\[
t := \|\vtheta_*\|_2 = \|\theta_*\|_{\Cov}
\] 
by
\[
\kappa := \int_{-\infty}^\infty a''(tu) u^2 \phi(u) \ud u, \quad \kappa_{\perp} := \int_{\R} a''(tu) \phi(u) \ud u.
\]
In fact, the form~\eqref{eq:spiked-covariance} for~$\Psi(\vtheta)$ will be preserved with any elliptical distribution of~$X$, with somewhat more complicated expressions for~$\kappa$ and~$\kappa_{\perp}$. 
Our next step is to lower-bound~$\kappa$ and~$\kappa_{\perp}$, which automatically yields an upper bound for~$\rho$ in Assumption~\ref{ass:curvature}:
\begin{equation}
\label{eq:bounding-rho}
\rho \le \frac{1}{\min(\kappa,\kappa_{\perp})}.
\end{equation}
$\boldsymbol{1^o}.$
We bound~$\kappa$ and~$\kappa_{\perp}$ for logistic regression. Here one has~$a(\eta) = \log(1+e^{\eta})$,
\begin{equation*}
a'(\eta) = \sigma(\eta), \quad a''(\eta) = \sigma(\eta)(1-\sigma(\eta)),
\end{equation*}
where~$\sigma(\eta) := 1/(1+e^{-\eta})$ is the sigmoid. Clearly, we can bound~$a''(\eta),$~$\forall \eta \in \R:$
\[
\frac{1}{2(1+e^{|\eta|})} \le a''(\eta) \le \frac{1}{1+e^{|\eta|}},
\]
which yields
\begin{equation}
\label{eq:symm-sigmoid-bound}
\tfrac{1}{4}e^{-|\eta|} \le a''(\eta) \le e^{-|\eta|}.
\end{equation}
Hence, letting~$a \approx b$ denote the intersection of~$a \lsim b$ and~$a \gsim b$, we have
\begin{align*}
\kappa_{\perp} \approx  \int_{0}^{\infty} e^{-tu} \phi(u) \ud u \approx \int_{0}^{\infty} e^{-tu-u^2/2} \ud u = e^{t^2/2} G(t),
\end{align*}
where
\[
G(t) = \int_{t}^{+\infty} e^{-v^2/2} \ud v
\] 
is the partial Gaussian integral. Now,~\cite[Eq.~7.1.13]{abramowitz1965handbook} gives sharp bounds for~$G(t)$:
\begin{equation}
\label{eq:abramovitz}
\frac{2e^{-t^2/2}}{t + \sqrt{t^2 + 4}} \le G(t) \le \frac{2e^{-t^2/2}}{t + \sqrt{t^2 + {8}/{\pi}}}, \quad t \ge 0.
\end{equation}
In particular, these bounds imply
$
G(t) \approx {e^{-t^2/2}}/{(t + 1)},
$ 
whence,
\begin{equation}
\label{eq:kappa_perp}
\kappa_{\perp} \approx {1}/{(t+1)}.
\end{equation}
We can similarly bound~$\kappa$:
\begin{align*}
\kappa \approx  \int_{0}^{\infty} e^{-tu} u^2 \phi(u) \ud u 
\approx e^{t^2/2} \int_{0}^{\infty} e^{-(u+t)^2/2} u^2 \ud u 
= (t^2+1)G(t) - te^{-t^2/2}. 
\end{align*}
Using the lower bound in~\eqref{eq:abramovitz}, this gives 
\begin{equation}
\label{eq:kappa_par}
\kappa \ge \frac{4}{(t+\sqrt{t^2 + 4})(t^2+2+\sqrt{t^4 + 4t^2})} \gsim \frac{1}{1+t^3}.
\end{equation}
Plugging~\eqref{eq:kappa_perp} and~\eqref{eq:kappa_par} into~\eqref{eq:bounding-rho}, we arrive at~
$
\rho \lsim 1 + \|\theta_*\|_{\Cov}^3,
$
as claimed.
The dependency on~$t$ cannot be improved since the lower bound in~\eqref{eq:abramovitz} is sharp.

$\boldsymbol{2^o}.$
On the other hand, we can estimate~$K_2$ from Assumption~\ref{ass:second-subg} (and similarly~$\bar K_2(r)$ from Assumption~\ref{ass:second-subg++}). 
Indeed, note that
\[
K_2 = \|\Psi(\vtheta_*)^{-1/2}\wt Z(\theta_*)\|_{\psi_2} = \sup_{u \in \cS^{d-1}} \| \lang u, \Psi(\vtheta_*)^{-1/2}\wt Z(\theta_*) \rang \|_{\psi_2}.
\]
Let us consider separately the marginals for $u = \vtheta_*/t$ and for~$u$ from the othogonal complement of the span of~$\vtheta$. When~$u = \vtheta_*/t$, we have 
\[
\begin{aligned}
| \lang u, \Psi(\vtheta_*)^{-1/2}\wt Z(\theta_*) \rang | = \sqrt{\frac{a''(tZ_1)}{\kappa}} |Z_1| 
\lsim (1+t^{3/2}) e^{-\frac{t|Z_1|}{2}} |Z_1|,
\end{aligned}
\]
where~$Z_1 \sim \cN(0,1)$, and we used~\eqref{eq:symm-sigmoid-bound} and~\eqref{eq:kappa_par}. Thus, when~$t \lsim 1$, we have
\begin{equation*}
\label{eq:subgaussian-check-unit-subg-norm}
\|\lang u, \Psi(\vtheta_*)^{-1/2}\wt Z(\theta_*) \rang \|_{\psi_2} \lsim \|Z_1\|_{\psi_2} \lsim 1.
\end{equation*}
Let, on the contrary,~$t \gsim 1$. 
Note that in the case where
$
|Z_1| \ge \frac{3\log (1+t)}{t},
$
we have~$(1+t^{3/2}) e^{-t|Z_1|/2} \lsim 1$, whence 
$
|\lang u, \Psi(\vtheta_*)^{-1/2}\wt Z(\theta_*) \rang| \lsim |Z_1|.
$
On the other hand, when
$
|Z_1| \le \frac{3\log (1+t)}{t},
$
we have
\[
(1+t^{3/2}) e^{-t|Z_1|/2} |Z_1| \lsim (1+t^{1/2})\log(1+t).
\] 
Hence, when~$u$ is parallel to~$\vtheta_*$, we have
\[
\| \lang u, \Psi(\vtheta_*)^{-1/2}\wt Z(\theta_*) \rang \|_{\psi_2} \lsim (1 +\log(1+t)) \sqrt{1+t}.
\]
Finally, when~$u$ is orthogonal to~$\vtheta_*$, we can use the trivial estimate
\[
\begin{aligned}
\| \lang u, \Psi(\vtheta_*)^{-\frac{1}{2}}\wt Z(\theta_*) \rang \|_{\psi_2} 
= \left\| \sqrt{\tfrac{a''(tZ_1)}{\kappa_{\perp}}} \lang u, Z \rang\right\|_{\psi_2} 
\lsim \sqrt{1+t} \| \lang u, Z \rang \|_{\psi_2} \lsim \sqrt{1+t}.
\end{aligned}
\]
In fact, this bound is tight, which can be verified by Item 2 of Lemma~\ref{lem:subg-properties} (note that~$Z_1$ and~$\lang Z, u \rang$ are independent).
Thus, overall we have
\begin{equation}
\label{eq:bound-logistic-k2}
K_2 \lsim \left(1+\log(1+\|\theta_*\|_{\Cov})\right)\sqrt{1 + \|\theta_*\|_{\Cov}}.
\end{equation}
Moreover, for~$\bar K_2(r)$ from Assumption~\ref{ass:second-subg++}, we clearly have~
\[
\begin{aligned}
\bar K_2(r) 
&\lsim \sup_{\theta \in \Theta_r(\theta_*)} \left(1+\log(1+\|\theta\|_{\Cov})\right) \sqrt{1 + \|\theta\|_{\Cov}}\\
&\lsim \left(1 + \log(1+\|\theta_*\|_{\Cov} + r\sqrt{\rho})\right)\sqrt{1 + \|\theta_*\|_{\Cov} + r\sqrt{\rho}}.
\end{aligned}
\]
This still gives~\eqref{eq:bound-logistic-k2} when~$r \lsim 1/\sqrt{\rho}$, motivating our condition in Theorem~\ref{th:crb-fake-improved}.

$\boldsymbol{3^o}.$ 
Finally, let us verify Assumption~\ref{ass:first-subg}, assuming well-specified model. 
In this case,~$\Gr(\theta_*) = \He(\theta_*)$, and the trivial bound using~$|Y -  \sigma(X^\T \theta_*)| \le 1$ is
\[
K_1 \lsim \sqrt{\rho} \lsim 1+t^{3/2}.
\]
This is a rather discouraging result. 
However, we can show a weaker (subexponential) version of Assumption~\ref{ass:first-subg} with a milder dependency on~$t$, replacing the~$\|\cdot\|_{\psi_2}$ norm with the~$\|\cdot\|_{\psi_1}$-norm as defined in~\cite[Section~5.2.4]{vershynin2010introduction}: 
\begin{equation}
\label{eq:subexponential-norm}
\| \ell'(Y,X^\T\theta_*)Z  \|_{\psi_1} \lsim \log(1+t)^2 \sqrt{1+t}.
\end{equation}
An equivalent definition of the subexponential norm is as follows: a random variable~$\xi \in \R$ satisfies~$\|\xi\|_{\psi_1} \le K$ when its moments grow as~$(\E[|\xi|^p])^{1/p} \lsim K p,$ 
i.e., same as the moments of the exponential distribution; then, the $\psi_1$-norm of a random vector is defined as the maximum norm of its one-dimensional marginals.
Recall that for subgaussian variables the scaling is~$K\sqrt{p}$ (cf. Lemma~\ref{lem:subg-properties}).
For~\eqref{eq:subexponential-norm}, note that in the well-specified case for~$y \in \{0,1\}$ we have
\[
\bP\{Y = y\} = \sigma(X^\T \theta_*)^y (1-\sigma(X^\T \theta_*))^{1-y},
\] 
thus we bound the moments of the marginals of~$\ell'(Y,X^\T\theta_*)Z = (Y-\sigma(Z^\T\vtheta_*)Z$:
\[
\begin{aligned}
\E_{Z,Y}[(Y-\sigma(Z^\T\vtheta_*)) \lang Z, u \rang]^{p} 
&\le 2\E_{Z}\left[\sigma(Z^\T\vtheta_*)(1-\sigma(Z^\T\vtheta_*))\lang Z, u \rang^{p} \right]  \\
&\lsim 2\E_{Z}\left[e^{-|Z^\T\vtheta_*|}\lang Z, u \rang^{p} \right], \quad p \ge 1,
\end{aligned}
\]
where we used~\eqref{eq:symm-sigmoid-bound}. 
For~$u$ parallel to~$\vtheta_*$, we should prove that
\begin{equation}
\label{eq:lateral-projection-subexp}
(1+t)^{3/2} \left(\int_{0}^{+\infty} e^{-tu} u^p e^{-u^2/2} \ud u \right)^{1/p} \lsim p\log^2(1+t) \sqrt{1+t}.
\end{equation}
We proceed similarly to~$\boldsymbol{2^o}$, using that~$(1+t)^{3p/2} e^{-tu} \le 1$ 
for~$u \ge {3p \log(1+t)}/{(2t)}.$
Thus, when~$t \gsim 1$,
\begin{align*}
&(1+t)^{3p/2} \int_{0}^{+\infty} e^{-tu} u^p e^{-u^2/2} \ud u \\
\le &(1+t)^{3p/2} \int_{0}^{\frac{3p \log(1+t)}{2t}} u^p \ud u + \int_{\frac{3p \log(1+t)}{2t}}^{+\infty} u^p e^{-u^2/2} \ud u \\
\lsim &(1+t)^{3p/2} \frac{1}{p+1}\left(\frac{3p \log(1+t)}{2t}\right)^{p+1} + p^{p/2} 
\lsim (2p)^p (1+t)^{p/2} \log(1+t)^{p+1},
\end{align*}
which implies~\eqref{eq:lateral-projection-subexp}. The remaining cases ($u$ parallel to~$\vtheta_*$ with~$t \lsim 1;$~$u \perp \vtheta_*$) are straightforward, by using that~$\|\cdot\|_{\psi_1} \le C \|\cdot\|_{\psi_2}$ for some constant~$C$, see~\cite{vershynin2010introduction}.  
\end{proof}

\bibliography{references}
\bibliographystyle{alpha}

\end{document}